\documentclass[11pt,letterpaper]{article}
\usepackage[english]{babel}
\usepackage{amsmath}
\usepackage{amsfonts, mathrsfs}
\usepackage{amssymb}
\usepackage{verbatim}
\usepackage{graphicx}
\usepackage{color}
\usepackage{url}
\usepackage{subcaption}
\usepackage{tikz}
\usepackage{pgflibraryarrows}
\usepackage{pgflibrarysnakes}
\usepackage{float}
\usepackage{multirow}
\usepackage{epsfig}
\usepackage{epstopdf}

\numberwithin{equation}{section}
\newtheorem{theo}{Theorem}[section]

\newtheorem{prop}[theo]{Proposition}
\newtheorem{lemma}[theo]{Lemma}
\newtheorem{example}[theo]{Example}
\newtheorem{assum}[theo]{Assumption}
\newtheorem{defn}[theo]{Definition}
\newtheorem{remark}[theo]{Remark}
\newenvironment{proof}[1][Proof]{\textbf{#1.} }{\ \rule{0.5em}{0.5em}}

\newcommand{\var}{{\rm Var} \mspace{1mu}}

\renewcommand{\labelenumi}{\alph{enumi}.)}

\newcommand{\Poi}{{\rm Poi} \mspace{1mu}}
\newcommand\norm[1]{\left\lVert#1\right\rVert}

\begin{document}

\title{On a general class of inhomogeneous random digraphs}
\author{
	Junyu Cao \\ {\small University of California, Berkeley}
	\and
	Mariana Olvera-Cravioto \\ {\small University of California, Berkeley}}

\maketitle

\begin{abstract}
	We study a family of directed random graphs whose arcs are sampled independently of each other, and are present in the graph with a probability that depends on the attributes of the vertices involved. In particular, this family of models includes as special cases the directed versions of the Erd\H os-R\'enyi model, graphs with given expected degrees, the generalized random graph, and the Poissonian random graph. We establish the phase transition for the existence of a giant strongly connected component and provide some other basic properties, including the limiting joint distribution of the degrees and the mean number of arcs. In particular, we show that by choosing the joint distribution of the vertex attributes according to a multivariate regularly varying distribution, one can obtain scale-free graphs with arbitrary in-degree/out-degree dependence. 

\vspace{5mm}
\end{abstract}

\section{Introduction} 

Complex networks appear in essentially all branches of science and engineering, and since the pioneering work of Erd\H os and R\'enyi in the early 1960s \cite{erdos1959random, erdos1960evolution}, people from various fields have used random graphs to model, explain and predict some of the properties commonly observed in real-world networks. Most of the work until now has been mainly focused on the study of undirected graphs, however, some important networks, such as the World Wide Web, Twitter, and ResearchGate, to name a few, are directed. The present paper describes a framework for analyzing a large class of directed random graphs, which includes as special cases the directed versions of some of the most popular undirected random graph models.

Specifically, we study directed random graphs where the presence or absence of an arc is independent of all other arcs. This independence among arcs is the basis of the classical Erd\H os-R\'enyi model \cite{erdos1959random, erdos1960evolution}, where the presence of an edge is determined by the flip of coin, with all possible edges having the same probability of being present.  However, it is well-known that the Erd\H os-R\'enyi model tends to produce very homogeneous graphs, that is, where all the vertices have close to the same number of neighbors, a property that is almost never observed in real-world networks. In the undirected setting, a number of models have been proposed to address this problem while preserving the independence among edges. Some of the best known models include the Chung-Lu model \cite{chung2002average, chung2002connected, chung2003spectra, miller2012edge}, the generalized random graph \cite{britton2006generating, bollobas2007phase, frieze2015introduction}, and the Norros-Reittu model or Poissonian random graph \cite{norros2006conditionally, van2016random, bollobas2007phase}. In the undirected case, all of these models were simultaneously studied in \cite{bollobas2007phase} under a broader class of graphs, which we will refer to as kernel-based models. In all of these models the inhomogeneity of the degrees is accomplished by assigning to each vertex a {\em type}, which is used to make the edge probabilities different for each pair of vertices. From a modeling perspective, the types correspond to vertex attributes that influence how likely a vertex is to have neighbors, and inhomogeneity among the types translates into inhomogeneous degrees. 

Our proposed family of directed random graphs, which we will refer to as {\em inhomogeneous random digraphs}, provides a uniform treatment of essentially any model where arcs are present independently of each other, in the same spirit as the work in \cite{bollobas2007phase} for the undirected case. The main results in this paper establish some of the basic properties studied on random graphs, including the expected number of arcs, the joint distribution of the in-degree and out-degree, and the phase transition for the size of the largest strongly connected component. We pay special attention to the so-called {\em scale-free} property, which states that the tail degree distribution(s) decay according to a power law. Since many real-world directed complex networks exhibit the scale-free property in either their in-degrees, their out-degrees, or both, we provide a theorem stating how the family of random directed graphs studied here can be used to model such networks. Our main result on the connectivity properties of the graphs produced by our model shows that there exists a phase transition, determined by the types, after which the largest strongly connected component contains (with high probability) a positive fraction of all the vertices in the graph, i.e., the graph contains a  ``giant" strongly connected component. 

That the undirected models mentioned above satisfy these basic properties (e.g., scale-free degree distribution, existence of a giant connected component, etc.) constitutes a series of classical results within the random graph literature, however, considerably fewer results are available for directed graphs in general. Some notable exceptions include the existence of a giant strongly connected component in the directed configuration model  \cite{Cooper2004size}, the corresponding result for the deterministic kernel directed model with a finite number of types  \cite{bloznelis2012birth}, the scale-free property on a directed preferential attachment model  \cite{samorodnitsky2016nonstandard, resnick2015tauberian}, and the limiting degree distributions in the directed configuration model \cite{chen2013directed}\footnote{Neither the configuration model nor the preferential attachment model have independent arcs, and therefore fall outside the scope of this paper.}. Our present work includes as a special case the main theorem in \cite{bloznelis2012birth} and extends it to a larger family of directed random graphs, and it also compiles several results for the number of arcs and the joint distribution of the degrees. It is also worth pointing out that the directed nature of our framework introduces some non-trivial challenges that are not present in the undirected setting, which is the reason we chose to provide a different approach from the one used in \cite{bollobas2007phase} for establishing some of our main results. We refer the reader to Section~\ref{S.PhaseTransition} for more details on these challenges and what they imply. 

The paper is organized as follows.  In Section~\ref{model} we specify a class of directed random graphs via their arc probabilities, and explain how the models mentioned above fit into this framework. In Section~\ref{mainresults} we provide our main results on the basic properties of the graphs produced by our model, and in Section~\ref{proof} we give all the proofs.

 \section{The Model}\label{model}

As mentioned in the introduction, we study directed random graphs with independent arcs. Since we are particularly interested in graphs with inhomogeneous degrees, each vertex in the graph will be assigned a {\em type}, which will determine how large its in-degree and out-degree are likely to be. In applications, the type of a vertex can also be used to model other vertex attributes not directly related to its degrees.  We will assume that the types take values in a separable metric space $\mathcal{S}$, which we will refer to as the ``type space".

In order to describe our family of directed random graphs, we start by defining the vertex set $V_n = \{ 1, 2, \dots, n\}$ and the type sequence $\{ {\bf x}_1^{(n)}, \dots, {\bf x}_n^{(n)} \}$, where ${\bf x}_i^{(n)}$ denotes the type of vertex $i$ in a graph on the vertex set $V_n$.  Note that, depending on how we construct the type sequence,  it is possible for ${\bf x}_i^{(n)}$ to be different from ${\bf x}_i^{(m)}$ for $n \neq m$. 
Define $G_n(\kappa (1+\varphi_n))$ to be the graph on the vertex set $V_n$ whose arc probabilities are given by
\begin{equation} \label{eq:KernelProb}
p_{ij}^{(n)} = \left( \frac{\kappa({\bf x}_i^{(n)}, {\bf x}_j^{(n)})}{n}(1+\varphi_{n}(\bold{x}_i^{(n)},\bold{x}_j^{(n)}))\right) \wedge 1, \qquad 1 \leq i \neq j \leq n,
\end{equation}
where $\kappa$ is a nonnegative function on $\mathcal{S} \times \mathcal{S}$,  
$$\varphi_n({\bf x}, {\bf y}) = \varphi \left(n, \{ {\bf x}_k^{(n)}: 1 \leq k \leq n\}, {\bf x}, {\bf y} \right) > -1 \qquad \text{a.s.,}$$
and $x \wedge y = \min\{x, y\}$ ($x \vee y = \max\{x,y\}$).  In other words, $p_{ij}^{(n)}$ denotes the probability that there is an arc from vertex $i$ to vertex $j$ in $G_n(\kappa (1+ \varphi_n))$. The presence or absence of arc $(i,j)$ is assumed to be independent of all other arcs. Note that the function $\varphi_n({\bf x}, {\bf y})$ may depend on $n$, on the types of the two vertices involved, or on the entire type sequence. Following the terminology used in \cite{bollobas2007phase} and \cite{bloznelis2012birth}, we will refer to $\kappa$ as the kernel of the graph. Note that we have decoupled the dependence on $n$ and on the type sequence by including it in the term $\varphi_n({\bf x}, {\bf y})$, which implies that with respect to the notation used in \cite{bollobas2007phase}, $\kappa_n({\bf x}, {\bf y})$ there corresponds to $\kappa({\bf x}, {\bf y}) (1 + \varphi_n({\bf x}, {\bf y}))$ here.

Throughout the paper, we will refer to any directed random graph generated through our model as an {\em inhomogeneous random digraph} (IRD). 

We end this section by explaining how the directed versions of the Erd\H os-R\'enyi graph \cite{erdos1960evolution, erdos1959random, erdos1964strength, bollobas1998random}, the Chung-Lu (or ``given expected degrees'') model \cite{chung2002average, chung2002connected, chung2003spectra, miller2012edge}, the generalized random graph \cite{britton2006generating, bollobas2007phase, frieze2015introduction}, and the Norros-Reittu model (or ``Poissonian random graph'') \cite{norros2006conditionally, van2016random, bollobas2007phase}, as well as the directed deterministic kernel model in \cite{bloznelis2012birth}, fit into our framework. The first four examples fall into the category of so-called rank-1 kernels, where  the graph kernel is of the form $\kappa({\bf x}, {\bf y}) = \kappa_-({\bf x}) \kappa_+({\bf y})$ for some nonnegative continuous functions $\kappa_+$ and $\kappa_-$ on $\mathcal{S}$.

\begin{example} \label{E.KnownModels}
Directed versions of some well-known inhomogeneous random graph models. All of them, with the exception of the last one, are defined on the space $\mathcal{S} = \mathbb{R}_+$ for a type of the form ${\bf x} = (x^+, x^-)$, and correspond to rank-1 kernels with $\kappa_+({\bf x}) = x^+/\sqrt{\theta}$ and $\kappa_-({\bf x}) = x^-/\sqrt{\theta}$, with $\theta > 0$ a constant. For convenience, we have dropped the superscript $^{(n)}$ from the type sequence, i.e.,  $\{{\bf x}_1, \dots, {\bf x}_n\} = \{{\bf x}_1^{(n)}, \dots, {\bf x}_n^{(n)}\}$.

\begin{enumerate}
\item \emph{Directed Erd\H os-R\'enyi Model:} the arc probabilities are given by
$$p_{ij}^{(n)}=\lambda/n$$
where $\lambda$ is a given constant and $n$ is the total number of vertices; $\varphi_n({\bf x}_i, {\bf x}_j) = 0$. 

\item \emph{Directed Given Expected Degree Model (Chung-Lu):} the arc probabilities are given by  
$$p_{ij}^{(n)}=\frac{x_i^-x_j^+}{l_n}\wedge 1,$$ 
where $l_n=\sum_{i=1}^n (x_i^++x_i^-)$. In terms of \eqref{eq:KernelProb}, it satisfies $\varphi_{n}({\bf x}_i, {\bf x}_j) =\frac{\theta n-l_n}{l_n}$, where $\theta= \lim_{n \to \infty} l_n/n$.

\item \emph{Generalized Directed Random Graph:} the arc probabilities are given by
$$p_{ij}^{(n)}=\frac{x_i^-x_j^+}{l_n+x_i^-x_j^+},$$
which implies that $\varphi_{n}(\bold{x}_i,\bold{x}_j)=\frac{\theta n-l_n-x_i^-x_j^+}{l_n+x_i^-x_j^+}$, with $l_n$ and $\theta$ defined as above.

\item \emph{Directed Poissonian Random Graph (Norros-Reittu):} the arc probabilities are given by
$$p_{ij}^{(n)}=1-e^{-x_i^-x_j^+/l_n},$$
which implies that $\varphi_{n}(\bold{x}_i,\bold{x}_j)= \left( n\theta(1-e^{-x_i^-x_j^+/l_n})-x_i^-x_j^+ \right)/ (x_i^-x_j^+)$, with $l_n$ and $\theta$ defined as above.

\item \emph{Deterministic Kernel Model:} the arc probabilities are given by
$$p_{ij}^{(n)} = \frac{\kappa({\bf x}_i, {\bf x}_j)}{n} \wedge 1,$$
for a finite type space $\mathcal{S} = \{ {\bf s}_1, \dots, {\bf s}_M \}$, and a strictly positive function $\kappa$ on $\mathcal{S} \times \mathcal{S}$; in terms of \eqref{eq:KernelProb}, $\varphi_n({\bf x}_i, {\bf x}_j) = 0$. 
\end{enumerate}
\end{example}

\section{Main Results}\label{mainresults}

We now present our main results for the family of inhomogeneous random digraphs defined through  \eqref{eq:KernelProb}. As mentioned in the introduction, we focus on establishing some of the basic properties of this family, including the distribution of the degrees, the mean number of arcs, and the size of the largest strongly connected component. When analyzing the degree distributions, we specifically explain how to obtain the scale-free property under degree-degree correlations. 

% Make sure the proper statement has been made by this point
As mentioned in the previous section, we assume throughout the paper that the $n$th graph in the sequence is constructed using the types $\{ {\bf X}_1, \dots, {\bf X}_n\} = \{ {\bf X}_1^{(n)}, \dots, {\bf X}_n^{(n)} \}$, where we have dropped the superscript $^{(n)}$ to simplify the notation. From now on we will use upper case letters to emphasize the possibility that the $\{ {\bf X}_i\}$ may themselves be generated through a random process. To distinguish between these two levels of randomness, we define $\mathscr{F} =  \sigma(\{\bold{X}_i^{(n)},1\leq i\leq n\},n\geq 1)$ and the corresponding conditional probability and expectation $\mathbb{P}(\cdot) = P( \cdot | \mathscr{F})$ and $\mathbb{E}[\cdot] = E[ \cdot | \mathscr{F}]$, respectively.

Our first assumption will be to ensure that the $\{ {\bf X}_i\}$ converge in distribution under the unconditional probability $P$.  As is to be expected from the work in \cite{bollobas2007phase} for the undirected case, we will also need to impose some regularity conditions on the kernel $\kappa$, as well as on the function $\varphi_n$. Our main assumptions are summarized below.

\begin{assum} \label{reg}
\begin{enumerate}
\item There exists a Borel probability measure $\mu$ on $\mathcal{S}$ such that for any $\mu$-continuity set $A \subseteq \mathcal{S}$, 
$$\mu_n(A) := \frac{1}{n}\sum_{i=1}^n 1(\bold{X}_i\in A)\xrightarrow{P}\mu(A) \qquad n \to \infty ,$$
where $\xrightarrow{P}$ denotes convergence in probability.
\item $\kappa$ is nonnegative and continuous a.e.~on $\mathcal{S}\times\mathcal{S}$.
\item $\varphi_n({\bf x}, {\bf y})$ is continuous on $\mathcal{S} \times \mathcal{S}$ and it satisfies $\varphi_n(\bold{x},\bold{y}) \rightarrow 0$ $\mathbb{P}$-a.s.as $n\rightarrow\infty$ for any $\bold{x},\bold{y}\in \mathcal{S}$.
\item The following limits hold:
 $$\lim_{n\rightarrow\infty}\frac{1}{n^2}E\left[\sum_{i=1}^n\sum_{j=1}^n \kappa(\bold{X}_i,\bold{X}_j)\right]= \lim_{n\rightarrow\infty}\frac{1}{n}E\left[\sum_{i=1}^n\sum_{j\neq i} p_{ij}^{(n)}\right]=\iint_{\mathcal{S}^2}\kappa(\bold{x},\bold{y})\,\mu(d\bold{x})\mu(d\bold{y}) < \infty.$$
%\item For any $\epsilon > 0$, the function $\varphi_n$ satisfies:
%$$\lim_{n \to \infty} \frac{1}{n} E\left[ \sum_{i=1}^n \sum_{1\leq j \leq n, j \neq i}  1(|\varphi_n({\bf X}_i, {\bf X}_j)| > \epsilon) \right] = 0.$$
\end{enumerate}
\end{assum}

\begin{remark} \label{R.GroundSpace}
The pair $(\mathcal{S}, \mu)$, where $\mathcal{S}$ is a separable metric space and $\mu$ is a Borel probability measure, is referred to in \cite{bollobas2007phase} as a {\em generalized ground space}.  For convenience, we will adopt the same terminology throughout the paper. 
\end{remark}

\subsection{Number of arcs}

Our assumption that the types $\{{\bf X}_i\}$ converge in distribution as the size of the graph grows implies that the graphs produced by our model are sparse, in the sense that the mean number of arcs is of the same order as the number of vertices. Our first result provides an expression for the exact ratio between the number of arcs and the number of vertices.

\begin{theo} \label{numberofarcs}
Define $e(G_n(\kappa (1+ \varphi_n)))$ to be the number of arcs in $G_n(\kappa(1+\varphi_n))$. Then, under Assumption~\ref{reg}(a)-(d) we have
$$\frac{1}{n} \mathbb{E}[e(G_n(\kappa (1+ \varphi_n)))]\xrightarrow{L_1}\iint_{\mathcal{S}^2}\kappa(\bold{x},\bold{y})\,\mu(d\bold{x})\mu(d\bold{y})$$
and
$$\frac{1}{n} e(G_n(\kappa (1+ \varphi_n)))\xrightarrow{L_1}\iint_{\mathcal{S}^2}\kappa(\bold{x},\bold{y})\,\mu(d\bold{x})\mu(d\bold{y})$$
as $n \to \infty$, where $\xrightarrow{L_1}$ denotes convergence in $L_1$. 
\end{theo}

\subsection{Distribution of vertex degrees}

We now move on to describing the vertex degree distribution, which is best accomplished by looking at the properties of a typical vertex, i.e., one chosen uniformly at random. In particular, if $D_{n,i}^+$ and $D_{n,i}^-$ denote the in-degree and out-degree, respectively, of vertex $i \in V_n$, and we let $\xi$ be a uniform random variable in $\{1, 2, \dots, n\}$, then we study the distribution of  $(D_{n,\xi}^+, D_{n,\xi}^-)$. We point out that the distribution of $(D_{n,\xi}^+, D_{n,\xi}^-)$ also allows us to compute the proportion of vertices in the graph having in-degree $k$ and out-degree $l$ for any $k,l \in \mathbb{N}$.  In the sequel, $\Rightarrow$ denotes convergence in distribution. 

\begin{theo}\label{degree}
Under Assumption~\ref{reg} we have
$$\left( D_{n,\xi}^+,D_{n,\xi}^- \right)\Rightarrow( Z^+,Z^-), \quad E[D_{n,\xi}^\pm]\rightarrow E[Z^\pm], \quad  \text{ as }n\rightarrow\infty,$$
where $Z^+$ and $Z^-$ are conditionally independent (given {\bf X}) mixed Poisson random variables with mixing distributions 
$$\lambda_+(\bold{X}):=\int_\mathcal{S}\kappa(\bold{y},\bold{X}) \, \mu(d\bold{y}) \quad \text{and} \quad \lambda_-(\bold{X}):=\int_\mathcal{S}\kappa(\bold{X},\bold{y}) \, \mu(d\bold{y}) ,$$ 
respectively, and ${\bf X}$ is distributed according to $\mu$. 
\end{theo}

As mentioned earlier, we are particularly interested in models capable of creating scale-free graphs, perhaps with a significant correlation between the in-degree and out-degree of the same vertex. To see that our family of inhomogeneous random digraphs can accomplish this, we first introduce the notion of non-standard regular variation (see~\cite{resnick2015tauberian, samorodnitsky2016nonstandard}), which extends the notion of regular variation in the real line to multiple dimensions, with each dimension having potentially different tail indexes. In our setting we only need to consider two dimensions, so we only give the bivariate version of the definition.

\begin{defn}
A nonnegative random vector $(X,Y ) \in \mathbb{R}^2$ has a distribution that is non-standard regularly varying if there
exist scaling functions $a(t) \uparrow \infty$ and $b(t) \uparrow \infty$ and a non-zero limit measure $\nu(\cdot)$, called the limit or tail measure, such that 
$$tP\left( (X/a(t),Y/b(t))\in\cdot \right) \xrightarrow{v}\nu(\cdot), \qquad t \to \infty,$$ where $\xrightarrow{v}$ denotes vague convergence of measures in $M_+([0,\infty]^2 \backslash \{{\bf 0} \})$, the space of Radon measures on $[0,\infty]^2 \setminus\{ {\bf 0} \}$.
\end{defn}

In particular, if the scaling functions $a(t)$ and $b(t)$ are regularly varying at infinity with indexes $1/\alpha$ and $1/\beta$, respectively, that is $a(t) = t^{1/\alpha} L_a(t)$ and $b(t) = t^{1/\beta} L_b(t)$ for some $\alpha, \beta > 0$ and slowly varying functions $L_a$ and $L_b$, then the marginal distributions $P(X > t)$ and $P(Y > t)$ are regularly varying with tail indexes $-\alpha$ and $-\beta$, respectively (see Theorem~6.5 in \cite{resnick2007heavy}). Throughout the paper we use the notation $\mathcal{R}_\alpha$ to denote  the family of regularly varying functions with index $\alpha$.

To see how our family of IRDs can be used to model complex networks where both the in-degrees and the out-degrees possess the scale-free property, perhaps with different tail indexes, we give a theorem stating that the non-standard regular variation of the limiting degrees $(Z^+, Z^-)$ follows from that of the vector $(\lambda_+({\bf X}), \lambda_-({\bf X}))$. Moreover, for the models (a)-(d) in Example~\ref{E.KnownModels}, we have
$$(\lambda_+({\bf X}), \lambda_-({\bf X}) ) =  \left( \kappa_+({\bf X}) \int_{\mathcal{S}} \kappa_-({\bf y}) \mu(d{\bf y}), \,    \kappa_-({\bf X}) \int_{\mathcal{S}} \kappa_+({\bf y}) \mu(d{\bf y})\right) = \left( c X^+ , \, (1-c) X^- \right), $$
where $c = E[X^-]/\theta$ and $\theta = E[ X^+ + X^-]$, so the non-standard regular variation of $(Z^+, Z^-)$ can be easily obtained by choosing a non-standard regularly varying type distribution $\mu$.

\begin{theo}\label{regularvarying}
Let ${\bf X}$ denote a random vector in the type space $\mathcal{S}$ distributed according to $\mu$. Suppose that $\mu$ is such that $(\lambda_+(\bold{X}), \lambda_-(\bold{X}))$ is non-standard regularly varying with scaling functions $a(t) \in \mathcal{R}_{1/\alpha}$ and $b(t) \in \mathcal{R}_{1/\beta}$ and limiting measure $ \nu(\cdot)$. Then, $(Z^+, Z^-)$ is non-standard regularly varying with scaling functions $a(t)$ and $b(t)$ and limiting measure $\nu(\cdot)$ as well.
\end{theo}

\subsection{Phase transition for the largest strongly connected component} \label{S.PhaseTransition}

Our last result in the paper establishes a phase transition for the existence of a giant strongly connected component in $G_n(\kappa (1+ \varphi_n))$. That is, we provide a critical threshold for a functional of the kernel $\kappa$ and the type distribution $\mu$, such that above this threshold the graph will have a giant strongly connected component with high probability, and below it will not. Before stating the corresponding theorem, we give a brief overview of some basic definitions. 

For any two vertices $i,j$ in the graph, we say that there is a directed path from $i$ to $j$ if the graph contains a set of arcs $\{ (i, k_1), (k_1, k_2), \dots, (k_t, j)\}$ for some $t \geq 0$. A set of vertices $V \subseteq V_n$ is {\em strongly connected}, if for any two vertices $i,j \in V$ we have that there exists a directed path from $i$ to $j$ and one from $j$ to $i$. Moreover, we say that $V$ is a {\em giant} strongly connected component if $|V| \geq \epsilon n$ for some $\epsilon > 0$, where $| V|$ denotes the cardinality of set $V$.

For undirected graphs, the phase transition for the Erd\H os-R\'enyi model ($p_{ij}^{(n)} = \lambda/n$ for some $\lambda > 0$) dates back to the classical work of Erd\H os and R\'enyi in \cite{erdos1960evolution}, where the threshold for the existence of a giant connected component is $\lambda = 1$. The critical case, i.e., $\lambda =1$, was studied in \cite{luczak1994structure} using edge probabilities of the form $p_{ij}^{(n)} = (1 + c n^{-1/3})/n$ for some $c > 0$, in which case the size of the largest connected component was shown to be of order $n^{2/3}$. Somewhat unrelated, the corresponding phase transition was established for the (undirected) configuration model in \cite{molloy1995critical}, where the threshold was shown to be $E[D(D-1)]/E[D] = 1$, with $D$ distributed according to the limiting  degree distribution (as the number of vertices grows to infinity).  Back to the (undirected) inhomogeneous random graph setting,  i.e., $p_{ij}^{(n)} = \kappa({\bf x}_i, {\bf x}_j) (1 + \varphi_n({\bf x}_i, {\bf x}_j))/n$ with $\kappa$ symmetric,  the phase transition was first proven for various forms of rank-1 kernels. In particular, Chung and Lu established in \cite{chung2002connected}  the phase transition for the existence of a giant connected component in the so-called ``given expected degree" model. The same authors also give in \cite{chung2002average} a phase transition for the average distance between vertices when the type distribution $\mu$ follows a power-law.
%, in which case the threshold for the average distances to be of order $\log n$ vs.~$\log\log n$ is given in terms of the variance of $\mu$ (finite vs.~infinite).  
Norros and Reittu proved the phase transition for the existence of a giant connected component for the Poissonian random graph in  \cite{norros2006conditionally}, along with a characterization of the distance between two randomly chosen vertices, and Riordan proved it in~\cite{riordan2005small} for the $c/\sqrt{ij}$ model, which is equivalent to the rank-1 kernel $\kappa({\bf x}, {\bf y}) = \psi({\bf x}) \psi({\bf y})$ with $\psi({\bf x}) = \sqrt{c{\bf x}}$ and $\mu$ the distribution of a Pareto(2,1). More generally, the work in \cite{bollobas2007phase} gives the phase transition for the giant connected component for the general kernel case, along with some  other properties (e.g., second largest connected component, distances between vertices, and stability). The threshold for the existence of a giant connected component is $\| T_\kappa \|_{op} = 1$, with $\| \cdot \|_{op}$ the operator norm\footnote{$\| T \|_{op} := \sup\{ \| Tf \|_2: f \geq 0 ,\|f \|_2 \leq 1 \}$ and $\| f \|_2^2 = \int_{\mathcal{S}} f({\bf x})^2 \mu(d{\bf x})$.}, where $T_\kappa$ is a linear operator induced by $\kappa$, which in the rank-1 case becomes $\| T_\kappa \|_{op}^2 = E[ \psi({\bf X})^2] = 1$, with ${\bf X}$ distributed according to $\mu$.

For the directed case, the phase transition for the existence of a giant strongly connected component was proven for the directed Erd\H os-R\'enyi model ($p_{ij}^{(n)} = \lambda/n$ for some $\lambda > 0$) in \cite{karp1990transitive} and for the ``given number of arcs" version of the Erd\H os-R\'enyi model (number of arcs $= \lambda n$ for some $\lambda > 0$) in \cite{luczak1990phase}, with the threshold being $\lambda = 1$. The work in \cite{luczak1992giant} studies a related model where each vertex $i$ can have three types of arcs: up arcs for $j > i$, down arcs for $j < i$, and bidirectional arcs, and proved the corresponding phase transition for the appearance of a giant strongly connected component. For the directed configuration model the phase transition for the existence of a giant strongly connected component was given in \cite{frieze2015introduction} under the assumption that the limiting degrees have finite variance and satisfy some additional conditions on the growth of the maximum degree, and can also be indirectly obtained from the results in \cite{hoorn2017distances} under only finite covariance between the in-degree and out-degree. The threshold for the directed configuration model is $E[D^+ D^-]/E[D^+ + D^-] = 1$, where $(D^+, D^-)$ are the limiting in-degree and out-degree. A hybrid model where the out-degree has a general distribution with finite mean and the destinations of the arcs are selected uniformly at random among the vertices (which gives Poisson in-degrees) was studied in \cite{penrose2015} and was shown to have a phase transition at $E[D^-] = 1$.  Finally, for general inhomogeneous random digraphs such as those studied here, the main theorem in \cite{bloznelis2012birth} establishes the phase transition for the deterministic kernel in Example~\ref{E.KnownModels}(d) with finite type space $\mathcal{S} = \{1, 2, \dots, M\}$, without characterizing the strict positivity of the survival probability. The authors in \cite{bloznelis2012birth} also suggest that the general case can be obtained using the same techniques used in \cite{bollobas2007phase} to go from a finite type space to the general one, however, the proof in \cite{bollobas2007phase} requires a critical step that does not hold for directed graphs; see Section~\ref{S.ProofsPhaseTransition} for more details.

Our Theorem~\ref{giant} provides the full equivalent of the main theorem in \cite{bollobas2007phase} (Theorem~3.1) for the directed case, and its proof is based on a coupling argument between the exploration of both the inbound and outbound components of a randomly chosen vertex and a double multi-type branching process with a finite number of types. Our approach differs from that of \cite{bollobas2007phase}, done for undirected graphs, in the order in which the couplings are done, and it leverages on the main theorem in \cite{bloznelis2012birth} to obtain a lower bound for the size of the strongly connected component. We give more details on how our proof technique compares to that used in \cite{bollobas2007phase} in Section~\ref{S.ProofsPhaseTransition}. 

As in the undirected case, the size of the largest strongly connected component is related to the survival probability of a suitably constructed double multi-type branching process. To define it, let $\mathcal{T}_\mu^+(\kappa)$ and $\mathcal{T}_\mu^-(\kappa)$ denote two conditionally independent (given their common root) multi-type branching processes defined on the type space $\mathcal{S}$ whose roots are chosen according to $\mu$ and such that the number of offspring having types in a subset $A \subseteq \mathcal{S}$ that an individual of type ${\bf x} \in \mathcal{S}$ can have, is Poisson distributed with means
\begin{equation} \label{eq:OffspringDistr}
\int_A \kappa({\bf y}, {\bf x}) \mu(d{\bf y}) \quad \text{for } \mathcal{T}_\mu^+(\kappa) \qquad \text{and} \qquad  \int_A \kappa({\bf x}, {\bf y}) \mu(d{\bf y}) \quad \text{for } \mathcal{T}_\mu^-(\kappa),
\end{equation}
respectively. Next, let $\rho_+(\kappa; {\bf x})$ and $\rho_-(\kappa; {\bf x})$ denote the survival probabilities of $\mathcal{T}_\mu^+(\kappa; {\bf x})$ and $\mathcal{T}_\mu^-(\kappa; {\bf x})$, respectively, where $\mathcal{T}_\mu^+(\kappa; {\bf x})$ and $\mathcal{T}_\mu^-(\kappa; {\bf x})$ denote the trees whose root has type ${\bf x}$.  We recall that 
a branching process is said to survive if its total population is infinite.  We refer the reader to \cite{mode1971multitype,athreya1972} for more details on multi-type branching processes, including those with uncountable type spaces as the ones defined above.

In order to state our result for the phase transition in IRDs we first need to introduce the following definitions.

\begin{defn} \label{D.Irreducible}
A kernel $\kappa$ defined on a separable metric space $\mathcal{S}$ with respect to a Borel probability measure $\mu$ is said to be {\em irreducible} if for any subset $A \subseteq \mathcal{S}$ satisfying $\kappa = 0$ a.e.~on $A \times A^c$, we have either $\mu(A) = 0$ or $\mu(A^c) = 0$. We say that $\kappa$ is {\em quasi-irreducible} if there is a $\mu$-continuity set $\mathcal{S}'\subseteq \mathcal{S}$ with $\mu(\mathcal{S}')>0$ such that the restriction of $\kappa$ to $\mathcal{S}'\times\mathcal{S}'$ is irreducible, and $\kappa(\bold{x},\bold{y})=0$ if $\bold{x}\notin \mathcal{S}'$ or $\bold{y}\notin \mathcal{S}'$.
\end{defn}

\begin{defn} \label{D.Finitary}
 A kernel $\kappa$ on a a separable metric space $\mathcal{S}$ with respect to a Borel probability measure $\mu$ is {\em regular finitary} if $\mathcal{S}$  has a finite partition into sets $\mathcal{J}_1, . . . , \mathcal{J}_r$ such that $\kappa$ is constant on each $\mathcal{J}_i \times \mathcal{J}_j$, and each $\mathcal{J}_i$ is a $\mu$-continuity set, i.e., it is measurable and has $\mu(\partial \mathcal{J}_i) = 0$.
\end{defn}

To give the condition under which a giant strongly connected component exists we also need to define the operators induced by kernel $\kappa$, i.e., 
$$T_{\kappa}^-f(\bold{x})=\int_{\mathcal{S}}\kappa(\bold{x},\bold{y})f(\bold{y})\mu(d\bold{y}) \qquad \text{and} \qquad T_{\kappa}^+f(\bold{x}) = \int_{\mathcal{S}}\kappa(\bold{y},\bold{x})f(\bold{y})\mu(d\bold{y}),$$
as well as their corresponding spectral radii $r(T_\kappa^-)$ and $r(T_\kappa^+)$,  where the spectral radius of operator $T$ is defined as
$$r(T) = \sup \{ |\lambda| : T - \lambda I \text{ is not invertible} \},$$
with $I$ the operator that maps $f$ onto itself. 

%We also define
%$$\|T_\kappa^\pm\|:=\sup\{\|T_\kappa^\pm f\|_2:f\geq 0,\|f\|_2\leq 1\}\leq \infty.$$

The phase transition result for the largest strongly connected component is given below. 

\begin{theo} \label{giant}
Suppose Assumption~\ref{reg} is satisfied and $\kappa$ is irreducible. Let $\mathcal{C}_1(G_n(\kappa (1+\varphi_n)))$ denote the size of the largest strongly connected component of $G_n(\kappa (1+\varphi_n))$. Then, for any $\epsilon > 0$,
$$\frac{\mathcal{C}_1(G_n(\kappa (1+\varphi_n)))}{n} \xrightarrow{P}  \rho(\kappa)   \qquad n \to \infty,$$
where
$$\rho(\kappa) = \int_{\mathcal{S}} \rho_+(\kappa; {\bf x}) \rho_-(\kappa; {\bf x}) \mu(d{\bf x}).$$
Furthermore, if $\rho(\kappa) > 0$ then $r(T_\kappa^+)> 1$ and $r(T_\kappa^-) > 1$, and if there exists a regular finitary quasi-irreducible kernel $\tilde \kappa$ such that $\tilde \kappa \leq \kappa$ a.e.~and $r(T_{\tilde \kappa}^+) > 1$ (equivalently, $r(T_{\tilde \kappa}^-) > 1$), then $\rho(\kappa) > 0$. 
\end{theo}

\begin{remark} \label{R.NotIFF}
We point out that we do not have a full if and only if condition for the strict positivity of $\rho(\kappa)$, since our operators $T_\kappa^+$ and $T_\kappa^-$ may be unbounded, in which case the continuity of the spectral radius is not guaranteed. However, when $\kappa$ satisfies
$$\int_{\mathcal{S}} \int_{\mathcal{S}} \kappa(x,y)^2 \mu(d{\bf x}) \mu(d{\bf y}) < \infty,$$
then the operators $T_\kappa^+$ and $T_\kappa^-$ are compact (see Lemma~5.15 in \cite{bollobas2007phase}), and 
Theorem~2.1(a) in \cite{degla2008overview} gives the continuity of the spectral radius for a sequence of quasi-irreducible kernels $\kappa_m \nearrow \kappa$ as $m \to \infty$,  ensuring the existence of $\tilde \kappa$ in Theorem~\ref{giant}.  Interestingly, for the rank-1 case we can indeed provide a full characterization even when the operators $T_\kappa^+$ and $T_\kappa^-$ are unbounded, as Proposition~\ref{P.Rank-1} shows. 
\end{remark}

We end the expository part of the paper with a compilation of all our results for the rank-1 case, which includes the first four models in Example~\ref{E.KnownModels}.

\begin{prop}[IRDs with rank-1 kernel] \label{P.Rank-1}
Suppose that Assumption~\ref{reg} is satisfied with $\kappa$ irreducible and of the form $\kappa({\bf x}, {\bf y}) = \kappa_-({\bf x}) \kappa_+({\bf y})$. Let ${\bf X}$ denote a random variable distributed according to $\mu$. Then, the following properties hold:
\begin{enumerate}
\item {\bf Number of arcs:} let $e(G_n(\kappa (1+\varphi_n)))$ denote the number of arcs in $G_n(\kappa (1+\varphi_n))$, then
\begin{align*}
&\frac{1}{n} \mathbb{E}[ e(G_n(\kappa (1+\varphi_n))) ] \stackrel{L_1}{\longrightarrow} E[\kappa_+(\bold{X})]E[\kappa_-(\bold{X})] \qquad \text{and} \\
&\frac{e(G_n(\kappa (1+\varphi_n)))}{n} \stackrel{L_1}{\longrightarrow} E[\kappa_+(\bold{X})]E[\kappa_-(\bold{X})],
\end{align*}
as $n \to \infty$. 

\item {\bf Distribution of vertex degrees:} let $(D_{n,\xi}^+, D_{n,\xi}^-)$ denote the in-degree and out-degree of a randomly chosen vertex in $G_n(\kappa(1+\varphi_n))$. Set $\lambda_-(\bold{x})= \kappa_-(\bold{x})E[\kappa_+(\bold{X})]$ and $\lambda_+(\bold{x})=\kappa_+(\bold{x})E[\kappa_-(\bold{X})]$. Then, 
$$(D_{n,\xi}^+, D_{n,\xi}^-) \rightarrow (Z^+, Z^-), \qquad E[ D_{n,\xi}^\pm] \to E[Z^\pm],$$
as $n \to \infty$, where $Z^+$ and $Z^-$ are conditionally independent (given ${\bf X}$) mixed Poisson random variables with mixing distributions $\lambda_+({\bf X})$ and $\lambda_-({\bf X})$. 

\item {\bf Scale-free degrees:} suppose that $(\kappa_+(\bold{X}), \kappa_-(\bold{X}))$ is non-standard regularly varying with scaling functions $a(t) \in \mathcal{RV}(1/\alpha)$ and $b(t) \in \mathcal{RV}(1/\beta)$ and limiting measure $\tilde \nu(\cdot)$. Then, $(Z^+, Z^-)$ is non-standard regularly varying with scaling functions $a(t)$ and $b(t)$ and limiting measure $\nu(\cdot)$ satisfying 
$$\nu( (x,\infty] \times (y, \infty]) = \tilde \nu\left( \left( \frac{x}{E[\kappa_+({\bf X})]}, \infty\right] \times \left( \frac{y}{E[\kappa_-({\bf X})]}, \infty\right] \right).$$

\item {\bf Phase transition for the largest strongly connected component:} suppose $\kappa$ is irreducible and let $\mathcal{C}_1(G_n(\kappa (1+\varphi_n)))$ denote the size of the largest strongly connected component of $G_n(\kappa (1+\varphi_n))$. Then, 
$$\frac{\mathcal{C}_1(G_n(\kappa (1+\varphi_n))) }{n}\xrightarrow{P} \rho(\kappa), \qquad n \to \infty,$$
with $\rho(\kappa) > 0$ if and only if $E[\kappa_-({\bf X}) \kappa_+({\bf X})] > 1$. 

\end{enumerate}
\end{prop}

The remainder of the paper is devoted to the proofs of all the results mentioned above.

\section{Proofs}\label{proof}

This section contains all the proofs of the theorems in Section~\ref{mainresults}. They are organized according to the same order in which their corresponding statements appear. Throughout this section we use the notation 
$$q_{ij}^{(n)} = \frac{\kappa(\bold{X}_i,\bold{X}_j)}{n} \qquad 1 \leq i,j \leq n,$$
to denote the asymptotic limit of the arc probabilities in the graph, and to avoid having to explicitly exclude possible self-loops, we define $p_{ii}^{(n)} = 0$ for all $1 \leq i \leq n$. We also use $f(x) = O(g(x))$ as $x \to \infty$ to mean that $\lim_{x \to \infty} |f(x)/g(x)| < \infty$.

\subsection{Number of Arcs}

The first result we prove corresponds to Theorem~\ref{numberofarcs}, which gives the asymptotic number of edges in $G_n(\kappa(1+\varphi_n))$. Before we do so, we state and prove a preliminary technical lemma that will be used several times throughout the paper.

\begin{lemma}\label{bound}
Assume Assumption \ref{reg} holds and define for any $0< \epsilon < 1/2$ the events
\begin{equation} \label{eq:BijDef}
B_{ij} = \left\{ (1-\epsilon)q_{ij}^{(n)} \leq p_{ij}^{(n)} \leq (1+\epsilon)q_{ij}^{(n)}, \, q_{ij}^{(n)} \leq \epsilon  \right\}.
\end{equation}
Then,
$$\lim_{n\rightarrow\infty}\frac{1}{n}E\left[\sum_{i=1}^n\sum_{j=1}^n \left(p_{ij}^{(n)}+q_{ij}^{(n)}\right) 1(B_{ij}^c) \right]=0.$$
\end{lemma}

\begin{proof}
We start by defining $A_{ij} = \{ q_{ij}^{(n)} \leq \epsilon \}$ and noting that the expectation is bounded from above by 
\begin{align}
&\frac{2}{n} E\left[\sum_{i=1}^n\sum_{j=1}^n q_{ij}^{(n)}1\left(p_{ij}^{(n)}<(1-\epsilon)q_{ij}^{(n)}, A_{ij} \right)+\sum_{i=1}^{n}\sum_{j=1}^n p_{ij}^{(n)}1\left(p_{ij}^{(n)}>(1+\epsilon)q_{ij}^{(n)}, \, A_{ij} \right)\right] \label{star3} \\
&\hspace{5mm} + \frac{1}{n} E\left[ \sum_{i=1}^n \sum_{j=1}^n (1 + q_{ij}^{(n)} ) 1(A_{ij}^c) \right]. \label{eq:TailQ}
\end{align}
To show that \eqref{eq:TailQ} converges to zero, let ${\bf X}^{(n)}$ and ${\bf Y}^{(n)}$ be random variables distributed according to $\mu_n$, conditionally independent given $\mathscr{F}$, and note that
\begin{align*}
 \frac{1}{n} E\left[ \sum_{i=1}^n \sum_{j=1}^n (1 + q_{ij}^{(n)} ) 1(A_{ij}^c) \right] &\leq  \frac{1}{n} E\left[ \sum_{i=1}^n \sum_{j=1}^n (\epsilon^{-1} + 1) q_{ij}^{(n)} 1(A_{ij}^c) \right] \\
 &= (\epsilon^{-1}+1) E\left[ \kappa({\bf X}^{(n)}, {\bf Y}^{(n)}) 1(  \kappa({\bf X}^{(n)}, {\bf Y}^{(n)}) > \epsilon n) \right] .
\end{align*}
Note that Assumption~\ref{reg}(a)-(b) imply that $\kappa({\bf X}^{(n)}, {\bf Y}^{(n)}) \Rightarrow \kappa({\bf X}, {\bf Y})$ as $n \to \infty$, where ${\bf X}$ and ${\bf Y}$ are i.i.d.~with distribution $\mu$. It follows that there exists a coupling $({\bf X}^{(n)}, {\bf Y}^{(n)}, {\bf X}, {\bf Y})$ such that $\kappa({\bf X}^{(n)}, {\bf Y}^{(n)}) \to \kappa({\bf X}, {\bf Y})$ a.s. From here on $({\bf X}^{(n)}, {\bf Y}^{(n)})$ always refers to the random variables in the coupling. Moreover, Assumption~\ref{reg}(d) implies that $E[ \kappa({\bf X}^{(n)}, {\bf Y}^{(n)}) ] \to E[ \kappa({\bf X}, {\bf Y})] < \infty$, and therefore, $\{\kappa({\bf X}^{(n)}, {\bf Y}^{(n)}): n \geq 1\}$ is uniformly integrable (see Chapter 5, Theorem 5.2.2 in \cite{durrett2010probability}). It follows that
$$\lim_{n \to\infty} E\left[ \kappa({\bf X}^{(n)}, {\bf Y}^{(n)}) 1(  \kappa({\bf X}^{(n)}, {\bf Y}^{(n)}) > \epsilon n) \right] = 0.$$

To analyze the first sum in \eqref{star3}, note that 
\begin{align}
&\frac{1}{n}E\left[\sum_{i=1}^n\sum_{j=1}^n q_{ij}^{(n)}1\left(p_{ij}^{(n)}<(1-\epsilon)q_{ij}^{(n)},  A_{ij} \right)\right]\notag\\
\hspace{5mm}
&=\frac{1}{n}E\left[\sum_{i=1}^n\sum_{j=1}^n q_{ij}^{(n)}1\left( q_{ij}^{(n)} (1 + \varphi_n({\bf X}_i, {\bf X}_j) ) < (1-\epsilon)q_{ij}^{(n)} < \epsilon (1-\epsilon) \right)\right]  \notag\\
&\leq\frac{1}{n^2}E\left[\sum_{i=1}^n\sum_{j=1}^n\kappa(\bold{X}_i,\bold{X}_j)1\left(\varphi_n(\bold{X}_i,\bold{X}_j)<-\epsilon\right)\right] \notag \\
&= E\left[ \kappa({\bf X}^{(n)}, {\bf Y}^{(n)}) \right] - E\left[ \kappa({\bf X}^{(n)}, {\bf Y}^{(n)}) 1( \varphi_n({\bf X}^{(n)}, {\bf Y}^{(n)}) \geq -\epsilon) \right]. \label{eq:KappaError}
\end{align}
Similarly, the second sum in \eqref{star3} can be bounded as follows
\begin{align}
&\frac{1}{n} E\left[ \sum_{i=1}^{n}\sum_{j=1}^n p_{ij}^{(n)}1\left(p_{ij}^{(n)}>(1+\epsilon)q_{ij}^{(n)}, \, A_{ij} \right)\right] \notag \\
&= \frac{1}{n} E\left[ \sum_{i=1}^{n}\sum_{j=1}^n p_{ij}^{(n)}1\left( q_{ij}^{(n)} (1 + \varphi_n({\bf X}_i, {\bf X}_j) )  >(1+\epsilon)q_{ij}^{(n)}, \, q_{ij}^{(n)} \leq \epsilon \right)\right] \notag \\
&\leq \frac{1}{n} E\left[ \sum_{i=1}^{n}\sum_{j=1}^n p_{ij}^{(n)} 1(\varphi_n({\bf X}_i, {\bf X}_j) > \epsilon) \right] \notag \\
&= \frac{1}{n} E\left[ \sum_{i=1}^{n}\sum_{j=1}^n p_{ij}^{(n)}  \right] \notag \\
&\hspace{5mm} - E\left[ \left( \left\{ \kappa({\bf X}^{(n)}, {\bf Y}^{(n)}) (1 + \varphi_n({\bf X}^{(n)}, {\bf Y}^{(n)})) \right\} \wedge n \right) 1( \varphi_n({\bf X}^{(n)}, {\bf Y}^{(n)}) \leq \epsilon ) \right]. \label{eq:PijError}
\end{align}
Since by Assumption~\ref{reg}(d) we have
$$\lim_{n \to \infty} E\left[ \kappa({\bf X}^{(n)}, {\bf Y}^{(n)}) \right] = \lim_{n \to \infty} \frac{1}{n} E\left[ \sum_{i=1}^{n}\sum_{j=1}^n p_{ij}^{(n)}  \right] = \iint_{\mathcal{S}^2} \kappa({\bf x}, {\bf y}) \mu(d{\bf x}) \mu(d{\bf y}),$$
it only remains to show that the limit inferior of the second expectations in \eqref{eq:KappaError} and \eqref{eq:PijError} equals $\iint_{\mathcal{S}^2} \kappa({\bf x}, {\bf y}) \mu(d{\bf x}) \mu(d{\bf y}) = E[\kappa({\bf X}, {\bf Y})]$. 

To this end, note that Assumption~\ref{reg}(c) gives that $\varphi_n$ is continuous and satisfies $|\varphi_n({\bf x}, {\bf y})| \stackrel{P}{\to} 0$ as $n \to \infty$ for any ${\bf x}, {\bf y} \in \mathcal{S}$. It follows that $|\varphi_n({\bf X}^{(n)}, {\bf Y}^{(n)})| \stackrel{P}{\to} 0$ as $n \to \infty$, and Fatou's lemma gives that
\begin{align*}
&\liminf_{n \to \infty} E\left[ \kappa({\bf X}^{(n)}, {\bf Y}^{(n)}) 1( \varphi_n({\bf X}^{(n)}, {\bf Y}^{(n)}) \geq -\epsilon) \right] \\
&\geq E\left[ \liminf_{n \to \infty} \kappa({\bf X}^{(n)}, {\bf Y}^{(n)}) 1( \varphi_n({\bf X}^{(n)}, {\bf Y}^{(n)}) \geq -\epsilon) \right] \\
&= E\left[ \kappa({\bf X}, {\bf Y}) \right]
\end{align*}
and
\begin{align*}
&\liminf_{n \to \infty} E\left[ \left( \left\{ \kappa({\bf X}^{(n)}, {\bf Y}^{(n)}) (1 + \varphi_n({\bf X}^{(n)}, {\bf Y}^{(n)}))  \right\} \wedge n \right) 1( \varphi_n({\bf X}^{(n)}, {\bf Y}^{(n)}) \leq \epsilon ) \right] \\
&\geq E\left[ \liminf_{n \to \infty} \left( \left\{ \kappa({\bf X}^{(n)}, {\bf Y}^{(n)}) (1 + \varphi_n({\bf X}^{(n)}, {\bf Y}^{(n)})) \right\} \wedge n \right) 1( \varphi_n({\bf X}^{(n)}, {\bf Y}^{(n)}) \leq \epsilon ) \right] \\
&= E\left[ \kappa({\bf X}, {\bf Y}) \right].
\end{align*}
This completes the proof. 
\end{proof}

\bigskip

We are now ready to prove Theorem~\ref{numberofarcs}.

\begin{proof}[Proof of Theorem \ref{numberofarcs}]
We start by defining $W_n$ to be the average number of arcs in the graph $G_n(\kappa (1+\varphi_n))$, that is, $W_n:=\mathbb{E}[e(G_n(\kappa (1+\varphi_n)))]/n=\frac{1}{n}\sum_{i=1}^n\sum_{j=1}^n p_{ij}^{(n)}$. As in the proof of Lemma~\ref{bound}, let $({\bf X}^{(n)}, {\bf Y}^{(n)}, {\bf X}, {\bf Y})$ be a random vector such that ${\bf X}^{(n)}$ and ${\bf Y}^{(n)}$ are distributed according to $\mu_n$, conditionally independent given $\mathscr{F}$, ${\bf X}$ and ${\bf Y}$ are i.i.d.~with distribution $\mu$, and ${\bf X}^{(n)} \to {\bf X}$, ${\bf Y}^{(n)} \to {\bf Y}$ a.s.~as $n \to \infty$. Note that 
$$E[\kappa({\bf X}, {\bf Y})] = \iint_{\mathcal{S}^2} \kappa({\bf x}, {\bf y}) \mu(d{\bf x}) \mu(d{\bf y}).$$

We will first show that $W_n \stackrel{L_1}{\longrightarrow} E[\kappa({\bf X}, {\bf Y})]$ as $n \to \infty$. To do this, define the events $B_{ij}$ according to \eqref{eq:BijDef} and note that
\begin{align*}
W_n &= \mathbb{E}\left[ \kappa({\bf X}^{(n)}, {\bf Y}^{(n)}) (1 + \varphi_n({\bf X}^{(n)}, {\bf Y}^{(n)}) ) \wedge n \right].
\end{align*}
Moreover, note that Assumption~\ref{reg}(a)-(c) imply that
$$\kappa({\bf X}^{(n)}, {\bf Y}^{(n)}) (1 + \varphi_n({\bf X}^{(n)}, {\bf Y}^{(n)}) ) \wedge n \stackrel{P}{\longrightarrow} \kappa({\bf X}, {\bf Y}), \qquad n \to \infty,$$
while Assumption~\ref{reg}(d) gives
$$\lim_{n \to \infty} E\left[ W_n \right] = \lim_{n \to \infty} \frac{1}{n} E\left[ \sum_{i=1}^n \sum_{j=1}^n p_{ij}^{(n)} \right] =  E[\kappa({\bf X}, {\bf Y})] .$$
Therefore, by Theorem~5.2.2 in Chapter 5 of \cite{durrett2010probability}, we have that 
$$\kappa({\bf X}^{(n)}, {\bf Y}^{(n)}) (1 + \varphi_n({\bf X}^{(n)}, {\bf Y}^{(n)}) ) \wedge n \stackrel{L_1}{\longrightarrow} \kappa({\bf X}, {\bf Y}), \qquad n \to \infty.$$
 It follows that
\begin{align*}
E\left[ \left| W_n- E[\kappa({\bf X}, {\bf Y})] \right| \right] &= E\left[ \left| \mathbb{E}\left[ \left\{ \kappa({\bf X}^{(n)}, {\bf Y}^{(n)}) (1 + \varphi_n({\bf X}^{(n)}, {\bf Y}^{(n)}) ) \right\} \wedge n \right] - \mathbb{E}[\kappa({\bf X}, {\bf Y})]  \right|\right]  \\
&\leq E\left[ \mathbb{E}\left[  \left| \left\{ \kappa({\bf X}^{(n)}, {\bf Y}^{(n)}) (1 + \varphi_n({\bf X}^{(n)}, {\bf Y}^{(n)}) ) \right\} \wedge n - \kappa({\bf X}, {\bf Y})  \right| \right] \right] \\
&= E\left[  \left| \left\{ \kappa({\bf X}^{(n)}, {\bf Y}^{(n)}) (1 + \varphi_n({\bf X}^{(n)}, {\bf Y}^{(n)}) ) \right\} \wedge n - \kappa({\bf X}, {\bf Y})  \right|\right]  \to 0,
\end{align*}
as $n \to \infty$. We conclude that $W_n \stackrel{L_1}{\longrightarrow} E[\kappa({\bf X}, {\bf Y})] $ as $n \to \infty$. 

Next, we need to show that $e(G_n(\kappa (1+\varphi_n)))/n \xrightarrow{L_1} E[\kappa({\bf X}, {\bf Y})] $ as $n \to \infty$. To do this, let $Y_{ij}$ denote the indicator of whether arc $(i,j)$ is present in $G_n(\kappa (1+\varphi_n))$ and note that
$$e(G_n(\kappa (1+\varphi_n))) = \sum_{i=1}^n \sum_{j \neq i} Y_{ij},$$
where the $\{Y_{ij} \}$ are Bernoulli random variables with means $\{ p_{ij}^{(n)}\}$, conditionally independent given $\mathscr{F}$. Hence,
$$\var(e(G_n(\kappa (1+ \varphi_n))) | \mathscr{F}) = \sum_{i=1}^n \sum_{j=1}^n \var(Y_{ij} | \mathscr{F}) \leq \sum_{i=1}^n \sum_{j=1}^n \mathbb{E}[Y_{ij} ] = \sum_{i=1}^n \sum_{j=1}^n p_{ij}^{(n)} = n W_n. $$
Therefore we have that
\begin{align*}
E\left[(e(G_n(\kappa (1+\varphi_n)))/n-W_n)^2\right] &= n^{-2}E[\var(e(G_n(\kappa (1 + \varphi_n)))|\mathscr{F})] \\
&\leq n^{-2}E[n W_n]\rightarrow 0, 
\end{align*}
as $n \to \infty$. 
Hence, $e(G_n(\kappa (1 + \varphi_n)))/n-W_n\xrightarrow{L_2} 0$. Combined with our earlier conclusion that $W_n \xrightarrow{L_1} E[\kappa({\bf X}, {\bf Y})]$, we obtain that
$$e(G_n(\kappa (1+\varphi_n)))/n \xrightarrow{L_1} E[\kappa({\bf X}, {\bf Y})], \qquad n \to \infty.$$
\end{proof}

\subsection{Distribution of Vertex Degrees}\label{degreesequence}

We now move on to the proof for Theorem~\ref{regularvarying}. The proof of Theorem~\ref{degree} is given in Section~\ref{S.ProofsPhaseTransition}, since it can be obtained as a  corollary to Theorem~\ref{T.Coupling}. We will show that $(Z^+, Z^-)$ has a non-standard regularly varying distribution whenever their conditional means $(\lambda^+(\bold{X}),\lambda^-(\bold{X}))$ have a non-standard regularly varying distribution. Throughout the proof we use the notation ${\bf [ a,  b]} = \{ {\bf x} \in \mathbb{R}^2: {\bf a} \leq {\bf x} \leq {\bf b} \}$ to denote the rectangles in $\mathbb{R}^2$. 

\bigskip

\begin{proof}[Proof of Theorem \ref{regularvarying}]
To simplify the notation, let ${\bf W} = (W^+,W^-)=(\lambda^+(\bold{X}),\lambda^-(\bold{X}))$, and recall that we need to show that $\tilde \nu_t(\cdot ) = tP( (Z^+/a(t) \in du, \, Z^-/b(t)) \in \cdot )$ converges vaguely to $\nu(\cdot)$ in $M_+([0,\infty]^2\setminus\{{\bf 0}\})$ as $t \to \infty$. Note that by Lemma~6.1 in \cite{resnick2007heavy}, it suffices to show that $\tilde \nu_t([{\bf 0}, {\bf x} ]^c) \to \nu([{\bf 0}, {\bf x}]^c)$ as $t \to \infty$ for any continuity point ${\bf x}  \in [{\bf 0}, \boldsymbol{\infty}) \setminus \{ {\bf 0}\}$ of $\nu([{\bf 0}, \cdot ]^c)$. 

To start, fix $(p, q)  \in [{\bf 0}, \boldsymbol{\infty}) \setminus \{ {\bf 0}\}$ to be a continuity point of $\nu([{\bf 0}, \cdot ]^c)$ and note that
\begin{align*}
\tilde \nu_t((p,\infty] \times (q, \infty]) &= \int_{p}^\infty \int_{q}^\infty t P\left( \frac{Z^+}{a(t)} \in du, \, \frac{Z^-}{b(t)} \in dv\right) \\
&= t P\left(  \frac{Z^+}{a(t)} > p , \, \frac{Z^-}{b(t)} > q  \right) \\
&= t E\left[  P\left( \left.  \frac{Z^+}{a(t)} > p, \,  \frac{Z^-}{b(t)} > q   \right| \bold{W} \right) \right] \\
&= t E\left[  P\left( \left.  Z^+ > pa(t)  \right| \bold{W} \right) P\left( \left.  Z^- > qb(t)    \right| \bold{W} \right) \right].
\end{align*}
It follows that we need to show that 
$$\lim_{t \to \infty} t E\left[ P\left( \left.  Z^+ > pa(t)  \right| \bold{W} \right) P\left( \left.  Z^- > qb(t)    \right| \bold{W} \right) \right] = \nu((p,\infty] \times (q,\infty]).$$

To this end, define $e(t) =  \sqrt{\gamma a(t)\log a(t)}$ and $d(t) =  \sqrt{\eta b(t) \log b(t)}$ with $\gamma > 2q\beta $, $\eta > 2p\alpha$, and use them to define the events
$$A_t = \{ W^+ > pa(t) - e(t)\} \qquad \text{and} \qquad B_t = \{ W^- > qb(t) - d(t) \}.$$
Now note that
\begin{align}
&t E\left[  P\left( \left.  Z^+ > pa(t)  \right| \bold{W} \right) P\left( \left.  Z^- > qb(t)    \right| \bold{W} \right) \right] \notag \\
&= t E\left[  P\left( \left.  Z^+ > pa(t)  \right| \bold{W} \right) P\left( \left.  Z^- > qb(t)    \right| \bold{W} \right) 1(A_t \cap B_t) \right]  \label{eq:NonNegligibleSet} \\
&\hspace{5mm} + t E\left[  P\left( \left.  Z^+ > pa(t)  \right| \bold{W} \right) P\left( \left.  Z^- > qb(t)    \right| \bold{W} \right) 1(A_t^c \cup B_t^c) \right]. \label{eq:NegligibleSet}
\end{align}
To see that \eqref{eq:NegligibleSet} vanishes in the limit, use the bound $P(\Poi(\lambda) \geq p) \leq e^{-\lambda} (e\lambda/p)^p$ for $p > \lambda$, where $\text{Poi}(\lambda)$ is Poisson random variable with mean $\lambda$, to obtain that
\begin{align*}
&t E\left[  P\left( \left. Z^+ > pa(t)  \right| \bold{W} \right) P\left( \left.  Z^- > qb(t)     \right| \bold{W} \right) 1(A_t^c) \right] \\
&\leq t E\left[  P\left( \left. Z^+ > pa(t)  \right| \bold{W} \right) 1(A_t^c) \right] \\
&\leq t E\left[ \exp\left\{-W^+ + pa(t) \left(1 + \log (W^+) - \log(pa(t)) \right) \right\} \, 1(A_t^c) \right] \\
&\leq t  \exp\{ -(pa(t)-e(t)) + pa(t) (1 + \log (pa(t)-e(t)) - \log(pa(t))) \}  \\
&= t \exp\left\{ e(t)  + pa(t) \log \left( 1 - \frac{e(t)}{ pa(t)} \right) \right\} \\
&= t \exp\left(- \frac{e(t)^2}{2 pa(t)} + O\left( \frac{e(t)^3}{(pa(t))^2} \right)\right) = t a(t)^{-\frac{\gamma}{2p}} \left(1 + O\left( \frac{(\log a(t))^{3/2}}{a(t)^{1/2}} \right) \right),
\end{align*}
where in the third inequality we used the observation that $g(u) = -u + pa(t)\log u$ is concave with a unique maximizer at $u^* = pa(t)$. Similarly,
\begin{align*}
&t E\left[  P\left( \left. Z^+ > pa(t)   \right| \bold{W} \right) P\left( \left. Z^- > qb(t)    \right| \bold{W} \right) 1(B_t^c) \right] \\
&\leq t b(t)^{-\frac{\eta}{2q}} \left(1 + O\left( \frac{(\log b(t))^{3/2}}{b(t)^{1/2}} \right) \right). 
\end{align*}
Our choice of $\gamma,\eta$ guarantees that both terms converge to zero as $t \to \infty$, hence showing that \eqref{eq:NegligibleSet} does so as well. 

It remains to show that \eqref{eq:NonNegligibleSet} converges to $\nu((p,\infty] \times (q, \infty])$ as $t \to \infty$. To do this, we first note that \eqref{eq:NonNegligibleSet} is equal to 
$$tP(A_t \cap B_t) - tE\left[ \left(1 - P\left( \left.  Z^+ > pa(t)  \right| \bold{W} \right) P\left( \left.  Z^- > qb(t)    \right| \bold{W} \right) \right) 1(A_t \cap B_t) \right],$$
where 
\begin{align*}
&tE\left[ \left(1 - P\left( \left.  Z^+ > pa(t)  \right| \bold{W} \right) P\left( \left.  Z^- > qb(t)    \right| \bold{W} \right) \right) 1(A_t \cap B_t) \right]  \\
&\leq tE\left[ P\left( \left.  Z^+ \leq pa(t)  \right| \bold{W} \right)  1(A_t \cap B_t) \right] +  tE\left[ P\left( \left.  Z^- \leq qb(t)    \right| \bold{W} \right)  1(A_t \cap B_t) \right] \\
&\leq tE\left[ P\left( \left.  Z^+ \leq pa(t)  \right| \bold{W} \right)  1(\tilde A_t \cap B_t) \right] +  tE\left[ P\left( \left.  Z^- \leq qb(t)    \right| \bold{W} \right)  1(A_t \cap \tilde B_t) \right] \\
&\hspace{5mm} + t P( \tilde A_t^c \cap A_t \cap B_t) + t P(  A_t \cap B_t \cap \tilde B_t^c)
\end{align*}
with
$$\tilde A_t = \{ W^+ > pa(t) + e(t)\} \subseteq A_t \qquad \text{and} \qquad \tilde B_t = \{ W^- > qb(t) + d(t) \} \subseteq B_t.$$

Now note that the inequality $P(\text{Poi}(\lambda) \leq p) \leq e^{-\lambda}(e\lambda/p)^p$ for $0 \leq p < \lambda$ gives that
\begin{align*}
&tE\left[ P\left( \left.  Z^+ \leq pa(t)  \right| \bold{W} \right)  1(\tilde A_t \cap B_t) \right]  \\
&\leq tE\left[ \text{exp}\left\{ - W^+ + pa(t) \left( 1 + \log(W^+) - \log(pa(t)) \right) \right\}  1(\tilde A_t ) \right] \\
&\leq t \, \text{exp}\left\{ - (pa(t) + e(t)) + pa(t) \left( 1 + \log(pa(t) + e(t)) - \log(pa(t)) \right) \right\} \\
&= t \, \text{exp}\left\{  - e(t)  + pa(t) \log\left( 1 + \frac{e(t)}{pa(t)} \right)  \right\} \\
&= t \exp\left(- \frac{e(t)^2}{2 pa(t)} + O\left( \frac{e(t)^3}{(pa(t))^2} \right)\right) = t a(t)^{-\frac{\gamma}{2p}} \left(1 + O\left( \frac{(\log a(t))^{3/2}}{a(t)^{1/2}} \right) \right),
\end{align*}
where we used again the concavity of $g(u) = -u + pa(t) \log u$. Similarly,
$$ tE\left[ P\left( \left.  Z^- \leq qb(t)    \right| \bold{W} \right)  1(A_t \cap \tilde B_t) \right] \leq t b(t)^{-\frac{\eta}{2q}} \left(1 + O\left( \frac{(\log b(t))^{3/2}}{b(t)^{1/2}} \right) \right),$$
and our choice of $\gamma,\eta$ give again that
\begin{equation} \label{eq:LowerTail}
\lim_{t \to \infty} \left\{ tE\left[ P\left( \left.  Z^+ \leq pa(t)  \right| \bold{W} \right)  1(\tilde A_t \cap B_t) \right] +  tE\left[ P\left( \left.  Z^- \leq qb(t)    \right| \bold{W} \right)  1(A_t \cap \tilde B_t) \right] \right\} = 0.
\end{equation}

Next,  let $\nu_t(du, dv) = tP( W^+/a(t) \in du, \, W^-/b(t) \in dv)$ and note that for any $0 < \epsilon < p \wedge q$, we have that
\begin{align*}
&\limsup_{t \to \infty} \left\{ t P( \tilde A_t^c \cap A_t \cap B_t) + t P(  A_t \cap B_t \cap \tilde B_t^c) \right\} \\
&= \limsup_{t \to \infty} \left\{  \nu_t\left( (p - e(t)/a(t), p+e(t)/a(t)] \times (q-d(t)/b(t), \infty] \right) \right. \\
&\hspace{20mm} \left. + \nu_t\left( (p-e(t)/a(t), \infty] \times (q-d(t)/b(t), q+d(t)/b(t)] \right) \right\} \\
&\leq \limsup_{t \to \infty} \left\{  \nu_t\left( (p - \epsilon, p+\epsilon] \times (q-\epsilon, \infty] \right) + \nu_t\left( (p-\epsilon, \infty] \times (q-\epsilon, q+\epsilon] \right) \right\} \\
&= \nu\left( (p - \epsilon, p+\epsilon] \times (q-\epsilon, \infty] \right) + \nu\left( (p-\epsilon, \infty] \times (q-\epsilon, q+\epsilon] \right).
\end{align*}
Moreover, since $(p,q)$ is a continuity point of $\nu$, then
$$\lim_{\epsilon \downarrow 0} \left\{ \nu\left( (p - \epsilon, p+\epsilon] \times (q-\epsilon, \infty] \right) + \nu\left( (p-\epsilon, \infty] \times (q-\epsilon, q+\epsilon] \right) \right\} = 0.$$
It follows that
$$\lim_{t \to \infty} \left\{ t P( \tilde A_t^c \cap A_t \cap B_t) + t P(  A_t \cap B_t \cap \tilde B_t^c) \right\} = 0,$$
which combined with \eqref{eq:LowerTail} gives that
$$\lim_{t \to \infty}  tE\left[ \left(1 - P\left( \left.  Z^+ > pa(t)  \right| \bold{W} \right) P\left( \left.  Z^- > qb(t)    \right| \bold{W} \right) \right) 1(A_t \cap B_t) \right] = 0.$$

Finally, the continuity of $\nu$ at $(p,q)$ also yields that
\begin{align*}
\lim_{t\to \infty} tP(A_t \cap B_t) &= \lim_{t \to \infty} \nu_t\left( (p-e(t)/a(t), \infty] \times (q - d(t)/b(t), \infty] \right)= \nu\left( (p, \infty] \times (q, \infty] \right).
\end{align*}
\end{proof}

\subsection{Phase transition for the largest strongly connected component} \label{S.ProofsPhaseTransition}

The last part of the paper considers the connectivity properties of the graph, in particular, the size of the largest strongly connected component. As mentioned in Section~\ref{S.PhaseTransition}, our Theorem~\ref{giant} provides the directed version of Theorem~3.1 in \cite{bollobas2007phase}. However, our proof approach differs from the one used in \cite{bollobas2007phase} in the order in which we construct the different couplings involved. Specifically, in \cite{bollobas2007phase} the authors first couple the graph $G_n(\kappa(1+\varphi_n))$ with another graph $G_n(\kappa_m)$, where $\kappa_m$ is a piecewise constant kernel taking at most a finite number of different values and such that $\kappa_m \nearrow \kappa$ as $m \to \infty$. Then, they provide a coupling between the exploration of the component of a randomly chosen vertex in $G_n(\kappa_m)$ and that of a multi-type branching process, $\mathcal{T}_\mu(\kappa_m)$, whose offspring distribution is determined by $\kappa_m$. The phase transition result is then obtained by relating the survival probability of $\mathcal{T}_\mu(\kappa_m)$ with the survival probability of its limiting tree $\mathcal{T}_\mu(\kappa)$. Our proof leverages on the work done in \cite{bloznelis2012birth}, which applies to a related graph $G_{n'}(\kappa_m)$, to establish a lower bound for the size of the largest strongly connected component. For the upper bound, we give a new direct coupling between the exploration of the in-component and out-component of a randomly chosen vertex in $G_n(\kappa (1+\varphi_n))$ and a double tree $(\mathcal{T}_\mu^+(\kappa_m), \mathcal{T}_\mu^-(\kappa_m))$, where $\kappa_m \nearrow \kappa$ as $m \to \infty$. We then relate the survival probabilities of  $(\mathcal{T}_\mu^+(\kappa_m), \mathcal{T}_\mu^-(\kappa_m))$ with those of their limiting trees  $(\mathcal{T}_\mu^+(\kappa), \mathcal{T}_\mu^-(\kappa))$ as $m \to \infty$. 

Interestingly, trying to adapt the approach used in \cite{bollobas2007phase} to the directed case leads to a phenomenon that does not occur when analyzing undirected graphs. Namely, if we consider two coupled undirected graphs $G_n(\kappa (1+ \varphi_n))$ and $G_n(\kappa'(1+\varphi_n'))$ such that every edge in the first graph is also present in the second one but not the other way around (e.g., when $\kappa({\bf x}, {\bf y})(1 + \varphi_n({\bf x}, {\bf y})) \leq \kappa'({\bf x}, {\bf y})(1 + \varphi_n'({\bf x}, {\bf y}))$ for all ${\bf x}, {\bf y} \in \mathcal{S}$), then, the difference in the sizes of the components of a vertex present in both graphs can be bounded by the difference in their number of edges (see Lemma~9.4 in \cite{bollobas2007phase}). However, in the directed case, this is no longer true, as Figure~\ref{F.TwoGraphs} illustrates.  In other words, the existence of a (giant) strongly connected component can be determined by a single arc. For this reason, a coupling of the graphs $G_n(\kappa(1+\varphi_n))$ and $G_n(\kappa_m)$, such as the one used in \cite{bollobas2007phase}, does not provide an upper bound for the size of the strongly connected component in the directed case. This may be a notable observation considering the folklore that exists around the 
equivalence of undirected and directed networks. 
 
\begin{figure}[!tbp]
  \centering
    \includegraphics[scale=0.9]{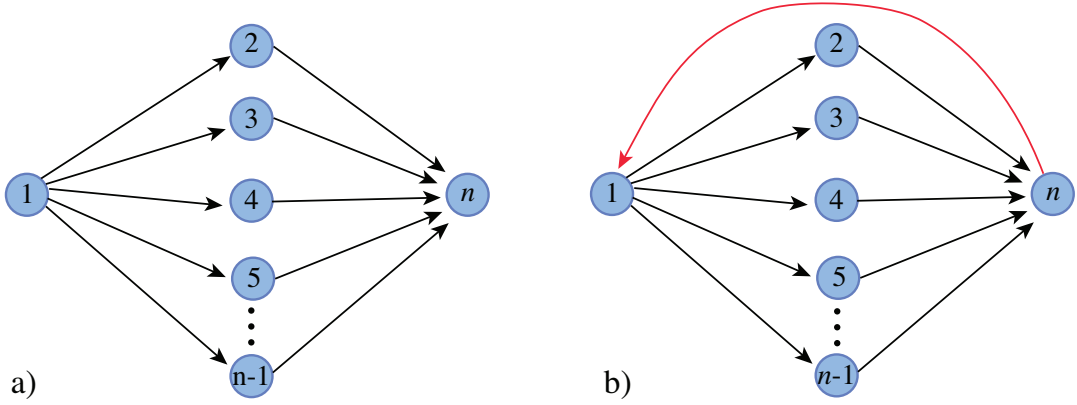}
    \caption{Directed graph with $n$ vertices. a) There is no strongly connected component. b) The same graph with one additional arc; the largest strongly connected component is {\em giant} of size $n$.} \label{F.TwoGraphs}
\end{figure}

With respect to how this section is organized, we have subdivided it into two subsections. In the first one we provide our coupling theorem between the exploration of the in-component and out-component of a randomly chosen vertex in $G_n(\kappa (1+\varphi_n))$ and the double tree  $(\mathcal{T}_\mu^+(\kappa_m), \mathcal{T}_\mu^-(\kappa_m))$. The second subsection gives the proof of Theorem~\ref{giant}, which establishes the phase transition for the size of the largest strongly connected component.

\subsubsection{Coupling with a double multi-type branching process} \label{SS.Coupling}

Starting with a randomly chosen vertex in $G_n(\kappa (1+\varphi_n))$, say vertex $i$, we will perform a double exploration process that we will couple with a double multi-type branching process $\{\hat {\bf Z}_t^{(n)}  : t \in \mathbb{N}_+ \}$ having  ``types'' $\{ 1, \dots, n\}$. Note that these ``types'' are actually the {\em identities} of the vertices in $V_n$, so to avoid confusion with the actual {\em types} of each of the vertices, i.e., $\{ {\bf X}_1, \dots, {\bf X}_n\}$, we will say that a vertex in the double tree has an {\em identity}, not a ``type''. The double tree is started at $\hat {\bf Z}_0^{(n)} = (\hat Z_{1,0}, \hat Z_{2,0}, \dots, \hat Z_{n,0}) $, and is such that for $t \geq 1$,  $ \hat {\bf Z}_t^{(n)} = (\hat Z_{1,t}^+, \hat Z_{2,t}^+, \dots, \hat Z_{n,t}^+, \hat Z_{1,t}^-, \hat Z_{2,t}^-, \dots, \hat Z_{n,t}^-) \in \mathbb{N}^{2n}$, where $\hat Z_{j,t}^+$ denotes the number of individuals of {\em identity} $j$ in the $t$th inbound generation of the double tree and $\hat Z_{j,t}^-$ denotes the number of individuals of {\em identity} $j$ in the $t$th outbound generation of the double tree. Moreover, the number of offspring that each node in the double tree has is independent of all other nodes in the double tree. The initial vector $\hat {\bf Z}_0^{(n)}$ is set to equal ${\bf e}_i$, where ${\bf e}_i$ is the unit vector that has a one in position $i$ and zeros elsewhere; note also that it does not have a $+/-$ superscript since it is at the center of the double tree. 

In order to define the offspring distribution of nodes in the double tree, we fix a kernel $\kappa_m$ on $\mathcal{S} \times \mathcal{S}$ satisfying
$$0 \leq \kappa_m({\bf x}, {\bf y}) \leq \kappa({\bf x}, {\bf y}) \qquad \text{for all } {\bf x}, {\bf y} \in \mathcal{S},$$
and such that 
$$\kappa_m({\bf x}, {\bf y}) = \sum_{i=1}^{M_m} \sum_{j=1}^{M_m} c_{ij}^{(m)} 1( {\bf x} \in \mathcal{J}_i^{(m)}, {\bf y} \in \mathcal{J}_j^{(m)}),$$
for some partition $\{ \mathcal{J}_i^{(m)} : 1 \leq i \leq M_m\}$ of $\mathcal{S}$ and some nonnegative constants $\{ c_{ij}^{(m)} : 1 \leq i,j \leq M_m\}$, $M_m < \infty$. Now let the number of offspring of {\em identity} $j$ that a node of {\em identity} $i$ in the inbound tree, respectively outbound tree, has, be Poisson distributed with mean $r_{ji}^{(m,n)}$, resp. $\tilde r_{ij}^{(m,n)}$, where:
$$r_{ji}^{(m,n)} = \frac{\kappa_m({\bf X}_j, {\bf X}_i) \mu(\mathcal{J}_{\theta(j)}^{(m)})}{n \mu_n(\mathcal{J}_{\theta(j)}^{(m)})} \qquad \text{and} \qquad \tilde r_{ij}^{(m,n)} = \frac{\kappa_m({\bf X}_i, {\bf X}_j) \mu(\mathcal{J}_{\theta(j)}^{(m)})}{n \mu_n(\mathcal{J}_{\theta(j)}^{(m)})}, $$
and $\theta(i) = j$ if and only if ${\bf X}_i \in \mathcal{J}_j^{(m)}$.   We denote $\mathcal{T}_\mu^+(\kappa_m; {\bf X}_i)$ and $\mathcal{T}_\mu^-(\kappa_m; {\bf X}_i)$ the inbound and outbound trees, respectively, whose root is vertex $i$. Note that the trees $\mathcal{T}_\mu^+(\kappa_m; {\bf X}_i)$ and $\mathcal{T}_\mu^-(\kappa_m; {\bf X}_i)$ are conditionally independent (given $\mathscr{F}$) by construction.

{\bf Note:}  We point out that in the double tree {\em identities} can appear multiple times, unlike in the graph where they appear only once. In either case, {\em identities} take values in the set $V_n = \{ 1, 2, \dots, n\}$.

\begin{remark}
An important observation that will be used later is that the double tree $ \hat {\bf Z}_t^{(n)} = (\hat Z_{1,t}^+, \dots, \hat Z_{n,t}^+, \hat Z_{1,t}^-,  \dots, \hat Z_{n,t}^-) \in \mathbb{N}^{2n}$ defined above, conditional on $\hat Z_0 = i \in V_n$, has the same law as the double tree $ \tilde {\bf Z}_t^{(m)} = (\tilde Z_{1,t}^+, \dots, \tilde Z_{M_m,t}^+, \tilde Z_{1,t}^-,  \dots, \tilde Z_{M_m,t}^-) \in \mathbb{N}^{2M_m}$, whose offspring distributions are Poisson with means
$$m_{ij}^+ := c_{ji}^{(m)} \mu(\mathcal{J}_j^{(m)}) \qquad \text{and} \qquad m_{ij}^- := c_{ij}^{(m)} \mu(\mathcal{J}_j^{(m)}), \qquad 1 \leq i,j \leq M_m.$$
Moreover, the latter is the same as $(\mathcal{T}_\mu^+(\kappa_m; {\bf x}), \mathcal{T}_\mu^-(\kappa_m; {\bf x}))$ for any ${\bf x} \in \mathcal{J}_i^{(m)}$. 
\end{remark}

Recall that $Y_{ij} = 1(\text{arc $(i,j)$ is present in } G_n(\kappa(1+\varphi_n))$ is a Bernoulli random variable with success probability
$$p_{ij}^{(n)} = \frac{\kappa({\bf X}_i, {\bf X}_j) (1 + \varphi_n({\bf X}_i, {\bf X}_j))}{n} \wedge 1, \quad 1 \leq i \neq j \leq n, \qquad p_{ii}^{(n)} = 0.$$
We will couple $Y_{ij}$ with a Poisson random variable $Z_{ij}$ having mean $r_{ij}^{(m,n)}$ on the inbound side, and with a Poisson random variable $\tilde Z_{ij}$ having mean $\tilde r_{ij}^{(m,n)}$ on the outbound side, using a sequence $\{ U_{ij}: 1 \leq i, j \leq n\}$ of i.i.d.~Uniform$(0,1)$ random variables.

The exploration of the graph and the construction of the double tree are done by choosing a vertex uniformly at random among those which have not been explored. Starting with vertex $i$, we fix the number of vertices to explore in the in-component of $i$, say $k_{in}$, and the number of vertices to explore in the out-component of $i$, say $k_{out}$.  A {\em step} in the exploration of the in-component (out-component) corresponds to identifying the inbound (outbound) neighbors of the vertex being explored. The exploration of the in-component continues until we have explored $k_{in}$ vertices or until there are no more vertices to reveal, after which we proceed to explore the out-component for $k_{out}$ steps or until there are no more vertices to reveal.  Moreover, we allow $k_{in}$ and $k_{out}$ to be stopping times with respect to the history of the exploration process. 

Vertices in the graph can have one of two labels: \{inactive, active\}. Active vertices are those that have been identified to be in the in-component, respectively out-component, of vertex $i$ but whose inbound, respectively outbound, neighbors have not been revealed. Inactive vertices are all other vertices that have been revealed through the exploration process but that are not active; again, there is an inbound inactive set and an outbound inactive set. Inactive vertices on the inbound side have revealed all its inbound neighbors, but not necessarily all their outbound ones; symmetrically, inactive nodes on the outbound side have revealed all their outbound neighbors but not necessarily all their inbound ones. 

In the double tree we will say that a node is ``active" if we have not yet sampled its offspring, and ``inactive" if we have.

{\bf Notation:} For $r = 0, 1, 2, \dots$, and assuming the chosen vertex is $i$, let
\begin{align*}
A_r^+ \, (A_r^-) &= \text{set of inbound (outbound) ``active" vertices after having explored the first $r$ } \\ &\hspace{5mm} \text{vertices in the in-component (out-component) of vertex $i$.} \\
I_r^+ \,  (I_r^-) &= \text{set of inbound (outbound) ``inactive" vertices after having explored the first $r$ } \\  &\hspace{5mm} \text{vertices in the in-component (out-component) of vertex $i$.} \\
T_r^+ \, (T_r^-) &= \text{{\em identity} of the vertex being explored in step $r$, $r \geq 1$, of the exploration of the } \\&\hspace{5mm} \text{in-component (out-component) of vertex $i$}. \\
\hat A_r^+\,  (\hat A_r^-) &= \text{set of  ``active" nodes in $\mathcal{T}_\mu^+(\kappa_m; {\bf X}_i)$ ($\mathcal{T}_n^-(\kappa_m; {\bf X}_i)$) after having sampled the offspring } \\ &\hspace{5mm} \text{of the first $r$ nodes in $\mathcal{T}^+_\mu(\kappa_m; {\bf X}_i)$ ($\mathcal{T}^-_\mu(\kappa_m; {\bf X}_i)$)}. \\
\hat I_r^+ \, (\hat I_r^-) &= \text{set of  {\em identities} belonging to ``inactive'' nodes in $\mathcal{T}_\mu^+(\kappa_m; {\bf X}_i)$ ($\mathcal{T}^-_\mu(\kappa_m; {\bf X}_i)$) after } \\ &\hspace{5mm} \text{having  sampled the offspring of the first $r$ nodes in $\mathcal{T}^+_\mu(\kappa_m; {\bf X}_i)$ ($\mathcal{T}^-_\mu(\kappa_m; {\bf X}_i)$)}. \\
\hat T_r^+ \, (\hat T_r^-) &= \text{{\em identity} of the node in $\mathcal{T}^+_\mu(\kappa_m; {\bf X}_i)$ ($\mathcal{T}^-_\mu(\kappa_m; {\bf X}_i)$) whose offspring are being sampled } \\ &\hspace{5mm} \text{in step $r$; $r \geq 1$}.
\end{align*}

%\begin{defn}
%For vectors ${\bf i}, {\bf j} \in \mathcal{U} = \{ \emptyset\} \cup \cup_{k=1}^\infty \mathbb{N}_+ $, we say that ${\bf i} \prec {\bf j}$ if $|{\bf i}| < |{\bf j}|$ or $|{\bf i}| = |{\bf j}| = k$ and $i_t < j_t$ for some $1 \leq t \leq k$. Here, $|{\bf i}|$ denotes the length of ${\bf i}$, i.e., $|{\bf i}| = k$ if ${\bf i} \in \mathbb{N}_+^k$. 
%\end{defn}

%{\bf Note:} Since the exploration of the in-component and the out-component of vertex $i$ forms two trees, each vertex revealed in either the in-component or out-component will receive a label according to its position within this tree, using the ordering $\prec$ defined above. 

\bigskip

{\em Exploration of the components of vertex $i$ in the graph:}

Fix $k_{in}$ and $k_{out}$.

\begin{itemize}
\item[1)]  For the exploration of the in-component:

Step 0:  Label vertex $i$ as ``active''  on the inbound side and set $A_0^+ = \{i\}$, $I_0^+  = \varnothing$. 

Step $r$, $1 \leq r \leq k_{in}$:  

Choose, uniformly at random, a vertex in $A_{r-1}^+$; let $T_r^+ = i$ denote its {\em identiy}. 
\begin{itemize} 
\item[a)] For $j = 1, 2, \dots, n$, $j \neq i$:
\begin{itemize}
\item[i.] Realize $Y_{ji} = 1(U_{ji} > 1- p_{ji}^{(n)})$. If $Y_{ji} = 0$ go to 1(a).
\item[ii.] If $Y_{ji} = 1$ and vertex $j \in I_{r-1}^+ \cup A_{r-1}^+$, do nothing. Go to 1(a). 
\item[iii.] If $Y_{ji} = 1$ and vertex $j$ had no label, label it ``active" on the inbound side. Go to 1(a).
\end{itemize}
\item[b)] Once all the new inbound neighbors of vertex $i$ have been identified and labeled ``active", label vertex $i$ as ``inactive" on the inbound side.
\item[c)] Define the sets $A_r^+ = A_{r-1}^+ \cup \{ \text{new ``active" vertices created in 1(a)(iii)}\} \setminus \{i \}$ and $I_r^+ = I_{r-1}^+ \cup \{ i\}$. This completes Step $r$ on the inbound side. 
\end{itemize}

\item[2)] For the exploration of the out-component:

Step 0: Label vertex $i$ as ``active'' on the outbound side and set $A_0^- = \{ i\}$, $I_0^- = \varnothing$. 

Step $r$, $1 \leq r \leq k_{out}$: 

Choose, uniformly at random, a vertex in $A_{r-1}^-$; let $T_r^-=i$ denote its {\em identity}. 
\begin{itemize}
\item[a)] For $j = 1, 2, \dots, n$, $j \neq i$, $j \notin I_{k_{in}}^+ \cup A_{k_{in}}^+$:
\begin{itemize}
\item[i.] Realize $Y_{ij} = 1(U_{ij} > 1- p_{ij}^{(n)})$. If $Y_{ij} = 0$ go to 2(a).
\item[ii.] If $Y_{ij} = 1$ and vertex $j \in I_{r-1}^- \cup A_{r-1}^-$, do nothing. Go to 2(a). 
\item[iii.] If $Y_{ij} = 1$ and vertex $j$ had no label, label it ``active" on the outbound side. Go to 2(a).
\end{itemize}
\item[b)] Once all the new outbound neighbors of vertex $i$ have been identified and labeled ``active", label vertex $i$ as ``inactive" on the outbound side.
\item[c)] Define the sets $A_{r}^- = A_{r-1}^- \cup \{ \text{new ``active" vertices created in 2(a)(iii)}\} \setminus \{ i \}$ and $I_r^- = I_{r-1}^- \cup \{ i\}$. This completes Step $r$ on the outbound side.
\end{itemize}
\end{itemize}

Note that by setting $k_{in} = \inf\{ r \geq 1: A_r^+ = \varnothing\}$ and $k_{out} = \inf\{ r \geq 1: A_r^- = \varnothing\}$ we can fully explore the in-component and out-component of vertex $i$.  We now explain how the coupled double tree is constructed. 

\bigskip

{\em Coupled construction of the double multi-type branching process:}

Let $g^{-1}(u)$ denote the pseudo inverse of function $g$, and let $G_{ji}$ and $\tilde G_{ij}$ be the distribution functions of Poisson random variables having means $r_{ji}^{(m,n)}$ and $\tilde r_{ij}^{(m,n)}$, respectively. On the double tree we use the index notation ${\bf i} = (i_1, \dots, i_r)$ to denote nodes in the $r$th generation (inbound/outbound) of the double tree. 
%For convenience, if node ${\bf i}$ has type ${\bf x}_j$, we simply say that it has type $j$. 
Let $T_{\bf i}$ denote the {\em identity} of node ${\bf i}$.

\bigskip

\begin{itemize}
\item[1)] Construction of the inbound tree:

Step 0: Set $\hat {\bf Z}_0^{(n)} = {\bf e}_i$. Let $\hat A_0^+ = \{ i\}$, $\hat I_0^+  = \varnothing$. 

%Step 0: Set $\hat T_0^+ = \hat T_0^- = i$, $\hat Z_{j,0} = 0$ for all $j \neq i$, and $\hat Z_{i,0} = 1$. Let $\hat A_0 = \{ i\}$, $\hat I_0^+ = \hat I_0^- = \varnothing$. 
Step $r$, $1 \leq r \leq k_{out}$:

Choose a node in ${\bf i} \in \hat A_{r-1}^+$, uniformly at random; set $\hat T_r^+ = T_{\bf i}$. 
\begin{itemize}
\item[I.] If this is the first time {\em identity} $T_{\bf i}$ appears in the inbound tree, do as follows: 
\begin{itemize}
\item[a)] For $j = 1, 2, \dots, n$,  $j \notin \{T_{\bf i}\} $:
\begin{itemize}
\item[i.] Realize $Z_{j,T_{\bf i}} = G_{j ,T_{\bf i}}^{-1}(U_{j,T_{\bf i}})$. If $Z_{j, T_{\bf i}} = 0$ go to 1(I)(a).
\item[ii.] If $Z_{j,T_{\bf i}} \geq 1$ label each of the newly created nodes as ``active" on the inbound side. Go to 1(I)(a).
\end{itemize} % end of (1)(I)(a)
\item[b)] For $j = T_{\bf i} $:
\begin{itemize}
\item[i.] Sample $Z_{j,T_{\bf i}}^*$ to be a Poisson random variable with mean $r_{j,T_{\bf i}}^{(m,n)}$, independently of everything else. If $Z_{j,T_{\bf i}}^* = 0$ go to 1(I)(c).
\item[ii.] If $Z_{j,T_{\bf i}}^* \geq 1$ label each of the newly created nodes as ``active" on the inbound side. Go to 1(I)(c).
\end{itemize} % end of (1)(I)(b)
\item[c)] Once all the inbound offspring of node ${\bf i}$ have been identified, label {\em identity} $T_{\bf i}$ as ``inactive" on the inbound side.
\item[d)] Define the sets $\hat A_r^+ = \hat A_{r-1}^+ \cup \{ \text{new ``active" nodes created in 1(I)(a)(ii) and 1(I)(b)(ii)}\} \setminus \{ {\bf i} \}$ and $\hat I_r^+ = \hat I_{r-1}^+ \cup \{ T_{\bf i} \}$. This completes Step $r$ on the inbound side.
\end{itemize} % end of (1)(I)

\item[II.] Else:
\begin{itemize}
\item[a)] For $j = 1,2, \dots, n$:
\begin{itemize}
\item[i.] Sample $Z_{j,T_{\bf i}}^*$ to be a Poisson random variable with mean $r_{j,T_{\bf i}}^{(m,n)}$, independently of everything else. If $Z_{j,T_{\bf i}}^* = 0$ go to 1(II)(a).
\item[ii.] If $Z_{j,T_{\bf i}}^* \geq 1$ label each of the newly created nodes as ``active" on the inbound side. Go to 1(II)(a).
\end{itemize}
\item[b)] Once all the inbound offspring of node ${\bf i}$ have been identified, label {\em identity} $T_{\bf i}$ as ``inactive" on the inbound side.
\item[c)] Define the sets $\hat A_r^+ = \hat A_{r-1}^+ \cup \{ \text{new ``active" nodes created in 1(II)(a)(ii)}\} \setminus \{{\bf i}\}$ and $\hat I_r^+ = \hat I_{r-1}^+ \cup \{ T_{\bf i} \}$. This completes Step $r$ on the inbound side.
\end{itemize}
\end{itemize} % end of (1)(II)

\item[2)] Construction of the outbound tree:

Step 0: Set $\hat A_0^- = \{ i\}$, $ \hat I_0^- = \varnothing$. 

Choose a node ${\bf i} \in \hat A_{r-1}^-$, uniformly at random; set $\hat T_r^- = T_{\bf i}$. 
\begin{itemize}
\item[I.] If this is the first time {\em identity} $T_{\bf i}$ appears in the outbound tree, do as follows: 
\begin{itemize}
\item[a)] For $j = 1, 2, \dots, n$, $j \notin \{ T_{\bf i} \}  \cup \{ T_{\bf j}: {\bf j} \in  \hat I_{k_{in}}^+ \cup \hat A_{k_{in}}^+\}$:
\begin{itemize}
\item[i.] Realize $\tilde Z_{T_{\bf i},j} = \tilde G_{T_{\bf i},j}^{-1}(U_{T_{\bf i},j})$. If $\tilde Z_{T_{\bf i},j} = 0$ go to 2(I)(a).
\item[ii.] If $\tilde Z_{T_{\bf i},j} \geq 1$ label each of the newly created nodes as ``active" on the outbound side. Go to 2(I)(a).
\end{itemize}
\item[b)] For $j \in \{ T_{\bf i} \} \cup \{ T_{\bf j}: {\bf j} \in \hat I_{k_{in}}^+ \cup \hat A_{k_{in}}^+\}$:
\begin{itemize}
\item[i.] Sample $\tilde Z_{T_{\bf i},j}^*$ to be a Poisson random variable with mean $\tilde r_{T_{\bf i},j}^{(m,n)}$, independently of everything else. If $\tilde Z_{T_{\bf i},j}^* = 0$ go to 2(I)(b).
\item[ii.] If $\tilde Z_{T_{\bf i},j}^* \geq 1$ label each of the newly created nodes as ``active" on the outbound side. Go to 2(I)(b).
\end{itemize}
\item[c)] Once all the outbound offspring of node ${\bf i}$ have been identified, label {\em identity} $T_{\bf i}$ as ``inactive" on the outbound side.
\item[d)] Define the sets $\hat A_r^- = \hat A_{r-1}^- \cup \{ \text{new ``active" nodes created in 2(I)(a)(ii) and 2(I)(b)(ii)}\} \setminus  \{ {\bf i} \}$ and $\hat I_r^+ = \hat I_{r-1}^+ \cup \{ T_{\bf i} \}$. This completes Step $r$ on the outbound side.
\end{itemize}
\end{itemize}

\item[II.] Else:
\begin{itemize}
\item[a)] For $j = 1,2, \dots, n$:
\begin{itemize}
\item[i.] Sample $\tilde Z_{T_{\bf i},j}^*$ to be a Poisson random variable with mean $\tilde r_{T_{\bf i},j}^{(m,n)}$, independently of everything else. If $\tilde Z_{T_{\bf i},j}^* = 0$ go to 2(II)(a).
\item[ii.] If $\tilde Z_{j,T_{\bf i}}^* \geq 1$ label each of the newly created nodes as ``active" on the outbound side. Go to 2(II)(a).
\end{itemize}
\item[b)] Once all the outbound offspring of node ${\bf i}$ have been identified, label {\em identity} $T_{\bf i}$ as ``inactive" on the outbound side.
\item[c)] Define the sets $\hat A_r^- = \hat A_{r-1}^- \cup \{ \text{new ``active" nodes created in 2(II)(a)(ii)}\} \setminus \{{\bf i}\}$ and $\hat I_r^+ = \hat I_{r-1}^+ \cup \{ T_{\bf i} \}$. This completes Step $r$ on the outbound side.
\end{itemize}
\end{itemize}

{\bf Note:} As long as the active sets in the graph and the double tree are the same, the chosen nodes in steps (1)(I) and (2)(I) are the same as the vertices chosen in steps (1) and (2) of the graph exploration process.

\begin{defn}
We say that the coupling of the graph and the double multi-type branching process holds up to Step $r$ on the inbound side if
$$A_t^+ = \{ T_{\bf j}: {\bf j} \in \hat A_t^+\} \qquad \text{and} \qquad |A_t^+| = |\hat A_t^+| \quad \text{for all } 0 \leq t \leq r,$$ 
and up to Step $r$ on the outbound side if
$$A_t^- = \{ T_{\bf j}: {\bf j} \in \hat A_t^-\} \qquad \text{and} \qquad |A_t^-| = |\hat A_t^-| \quad \text{for all } 0 \leq t \leq r.$$ 
Define the stopping time $\tau^+$ to be the step in the graph exploration process of vertex $T_0$ during which the coupling  breaks on the inbound side and $\tau^-$ to be the step during which it breaks on the outbound side. 
\end{defn}

\begin{remark} \label{R.BreakCoupling}
Note that $\tau^+ = r$ if either:
\begin{itemize}
\item[a.] If for any $j = 1, 2, \dots, n$, $j \notin \{T_r^+\} \cup A_{r-1}^+ \cup I_{r-1}^+$,  we have $Z_{j,T_r^+} \neq Y_{j, T_r^+}$ in step (1)(I)(a)(i),
\item[b.] If for any $j \in A_{r-1}^+ \cup I_{r-1}^+$ we have $Z_{j,T_r^+} \geq 1$ in step (1)(I)(a)(i),
\item[c.] If  $Z_{T_r^+,T_r^+}^* \geq 1$ in step (1)(I)(b)(i),
\end{itemize}
and $\tau^- = r$ if either:
\begin{itemize}
\item[d.] If for any $j = 1, 2, \dots, n$, $j \notin \{T_r^- \} \cup  I_{k_{in}}^+ \cup A_{k_{in}}^+ \cup A_{r-1}^- \cup I_{r-1}^-$, 
we have $\tilde Z_{T_r^-,j} \neq Y_{T_r^-,j}$ in step (2)(I)(a)(i),
\item[e.] If for any $j \in A_{r-1}^- \cup I_{r-1}^-$ we have $ \tilde Z_{T_r^-,j} \geq 1$ in step (2)(I)(a)(i),
\item[f.] If for any $j \in \{T_r^- \} \cup I_{k_{in}}^+ \cup A_{k_{in}}^+$, we have $\tilde Z_{T_r^-,j}^* \geq 1$ in step (2)(I)(b)(i).
\end{itemize}
\end{remark}

We are now ready to state our main coupling result, which provides an explicit upper bound for the probability that the coupling breaks before we can determine whether both the in-component and the out-component of the vertex being explored have at least $k$ vertices each or are fully explored. 

Throughout the reminder of the paper, we use the notation $\mathbb{P}_i(\cdot) = \mathbb{E}[ 1(\cdot ) | A_0 = \{i\} ]$ and $\mathbb{E}_i[ \cdot ] = \mathbb{E}[ \cdot | A_0 = \{ i\}]$; also, $\norm{{\bf x}}_1 = \sum_{i} |x_i|$ for any ${\bf x} \in \mathbb{R}^n$. Similarly to the definition of $\lambda_+({\bf x})$ and $\lambda_-({\bf x})$, define
$$\lambda_+^{(m)}({\bf x}) = \int_{\mathcal{S}} \kappa_m({\bf y}, {\bf x}) \mu(d{\bf y}) \qquad \text{and} \qquad \lambda_-^{(m)}({\bf x}) = \int_{\mathcal{S}} \kappa_m({\bf x}, {\bf y}) \mu(d{\bf y}),$$
$$\lambda_{m,n}^+({\bf x}) = \int_{\mathcal{S}} \kappa_m({\bf y}, {\bf x}) \mu_n(d{\bf y}) \qquad \text{and} \qquad \lambda_{m,n}^-({\bf x}) = \int_{\mathcal{S}} \kappa_m({\bf x}, {\bf y}) \mu_n(d{\bf y}),$$
and
$$\lambda_n^+({\bf x}) = \int_{\mathcal{S}} \kappa({\bf y}, {\bf x}) \mu_n(d{\bf y}) \qquad \text{and} \qquad \lambda_n^-({\bf x}) = \int_{\mathcal{S}} \kappa({\bf x}, {\bf y}) \mu_n(d{\bf y}).$$

\begin{theo}\label{T.Coupling}
Consider the exploration process described above along with its coupled double tree construction.  Define for any fixed $k \in \mathbb{N}_+$ the stopping times $\sigma_{k}^+ = \inf\{ t \geq 1: |A_t^+| + |I_t^+| \geq k \text{ or } A_t^+ = \varnothing \}$ and $\sigma_{k}^- = \{ t \geq 1: |A_t^-| + |I_t^-| \geq k \text{ or } A_t^- = \varnothing\}$.  For any $0 < \epsilon < 1/2$ and any $n,m \in \mathbb{N}_+$, 
$$\frac{1}{n} \sum_{i=1}^n \mathbb{P}_i \left( \{ \tau^+ \leq\sigma_{k}^+ \} \cup \{ \tau^- \leq \sigma_{k}^-\}  \right) \leq  H(n,m,k,\epsilon) ,$$
where  
\begin{align*}
H(n,m,k,\epsilon) &=  1(\Omega_{m,n}^c) + 4 \epsilon k^2 + 2\epsilon k^2 \left( 1 +  \sup_{{\bf x} \in \mathcal{S}} \lambda_+^{(m)}({\bf x}) \right) \\
&\hspace{5mm} + 1(\Omega_{m,n}) \sum_{r=1}^k \sum_{s=0}^{r-1} \binom{r-1}{s} 2^{r-1-s} \frac{1}{n} \left\{ \norm{ (\boldsymbol{\Gamma}^{(m,n)})^s {\bf g}^{(m,n)} }_1 + \norm{ (\boldsymbol{\tilde \Gamma}^{(m,n)})^s {\bf \tilde g}^{(m,n)} }_1 \right\},
\end{align*}
the matrices $\boldsymbol{\Gamma}^{(m,n)}, \boldsymbol{\tilde \Gamma}^{(m,n)} \in \mathbb{R}^{n\times n}$ are defined in Lemma~\ref{L.TypeDistributions}, and ${\bf g}^{(m,n)} = (g_1^{(m,n)}, \dots, g_n^{(m,n)})^T$, ${\bf \tilde g}^{(m,n)} = (\tilde g_1^{(m,n)}, \dots, \tilde g_n^{(m,n)})^T$ are defined according to
\begin{align*}
g_i^{(m,n)} &=   \min\left\{1, (1+5\epsilon) \lambda_n^+({\bf X}_i) -  \lambda_{m,n}^+({\bf X}_i) + (1+\epsilon) \sum_{j=1}^n  (p_{ji}^{(n)} + q_{ji}^{(n)}) 1(B_{ji}^c)  \right\} , \\
\tilde g_i^{(m,n)}  &=  \min\left\{ 1,  (1+5\epsilon) \lambda_n^-({\bf X}_i) -  \lambda_{m,n}^-({\bf X}_i)  + (1+\epsilon) \sum_{j=1}^n (p_{ij}^{(n)} + q_{ij}^{(n)}) 1(B_{ij}^c)\right\}, \\
\Omega_{m,n} &= \bigcap_{t=1}^{M_m}  \left\{  \left| \frac{\mu(\mathcal{J}_t^{(m)})}{\mu_n(\mathcal{J}_t^{(m)})} -1 \right| 1( \mu_n(\mathcal{J}_t^{(m)}) > 0) < \epsilon  \right\}, \\
B_{ij} &= \left\{ (1-\epsilon) q_{ij}^{(n)} \leq p_{ij}^{(n)} \leq (1+\epsilon) q_{ij}^{(n)}, \, q_{ij}^{(n)} \leq \epsilon \right\}.
\end{align*}
Moreover, there exists finite constants $H_{k}^+, H_{k}^-$, depending only on $k$, and $H_{1}^+ = H_{1}^- = 1$, such that
$$\lim_{\epsilon \downarrow 0} \limsup_{n \to \infty} H(n,m,k,\epsilon) \leq \hat H(m,k) \qquad \text{in probability},$$
where
\begin{align*}
\hat H(m,k) &= k^{-1} 1(k \geq 2)  + H_{k}^+ \int_{\mathcal{S}} (\lambda_+({\bf x}) -  \lambda_+^{(m)}({\bf x}) ) \mu(d{\bf x})  +  H_{k}^- \int_{\mathcal{S}} (\lambda_-({\bf x}) -  \lambda_-^{(m)}({\bf x}) ) \mu(d{\bf x}).
\end{align*}
\end{theo}

\bigskip

Before proving the theorem, we will state and prove several preliminary results. The first one below gives an upper bound for the number of offspring sampled in each side of  the double-tree $(\mathcal{T}_\mu^+(\kappa_m), \mathcal{T}_\mu^-(\kappa_m))$ up to step $\hat \sigma_k^+$ and step $\hat \sigma_k^-$, respectively.

\begin{lemma} \label{L.MeanTree}
Let $\hat \sigma_k^+ = \inf\{ t \geq 1: | \hat A_t^+| + | \hat I_t^+| \geq k \text{ or } \hat A_t^+ = \varnothing\}$ and $\hat \sigma_k^- = \inf\{ t \geq 1: | \hat A_t^-| + |\hat I_t^-| \geq k \text{ or } \hat A_t^- = \varnothing\}$. Then, 
$$\frac{1}{n} \sum_{i=1}^n \mathbb{E}_i\left[ \left| \hat I_{\hat \sigma_k^+}^+ \cup \hat A_{\hat \sigma_k^+}^+ \right| \right] \leq k +   k \sup_{{\bf x} \in \mathcal{S}} \lambda_+^{(m)}({\bf x}) \qquad \text{and} \qquad \frac{1}{n} \sum_{i=1}^n \mathbb{E}_i\left[ \left| \hat I_{\hat \sigma_k^-}^- \cup \hat A_{\hat \sigma_k^-}^- \right| \right] \leq k + k \sup_{{\bf x} \in \mathcal{S}} \lambda_-^{(m)}({\bf x}). $$
\end{lemma}

\begin{proof}
Define $\mathcal{G}_m^+$ to be the sigma-algebra containing all the information of the exploration process of the in-component of vertex $i$ up to the end of Step $m$ and including the {\em identity} of $T_{m+1}^+$. Note that 
\begin{align*}
&\mathbb{E}_i\left[ \left| \hat I_{\hat \sigma_k^+}^+ \cup \hat A_{\hat \sigma_k^+}^+ \right| \right] \\
&=  \mathbb{E}_i\left[ \left| \hat I_{\hat \sigma_k^+ -1}^+ \cup \hat A_{\hat \sigma_k^+ - 1}^+ \right|  + \sum_{j=1}^n Z_{j, \hat T_{\hat \sigma_k^+}^+} \right] \\
&\leq k-1 + \sum_{r=1}^k \mathbb{E}_i\left[ 1(\hat \sigma_k^+ = r) \sum_{j=1}^n Z_{j, \hat T_r^+} \right] \\
&= k -1+ \sum_{r=1}^k \mathbb{E}_i \left[ 1(\hat \sigma_k^+ > r-1) \mathbb{E} \left[ \left. 1\left( \sum_{j=1}^n  Z_{j, \hat T_r^+} \geq k - \left| \hat A_{r-1}^+ \cup \hat I_{r-1}^+  \right|   \right) \sum_{j=1}^n  Z_{j, \hat T_r^+}  \right| \mathcal{G}_{r-1}^+ \right] \right] .\end{align*}
Note that in the last equality the term that would correspond to $\{\hat A_r^+ = \emptyset \}$ in the description of the event $\{ \hat \sigma_k^+ = r\}$ vanishes since $\sum_{j=1}^n Z_{j,\hat T_r^+} = 0$ in that case. Now use the observation that $\sum_{j=1}^n  Z_{ji}$ is a Poisson random variable with mean $\sum_{j=1}^n r_{ji}^{(m,n)} = \lambda^{(m)}_+({\bf X}_i)$, and the identity $E[X 1(X \geq j) ] = \lambda P(X \geq j-1) \leq \lambda$ when $X$ is Poisson$(\lambda)$, to obtain that
\begin{align*}
&\sum_{r=1}^k \mathbb{E}_i\left[ 1(\hat \sigma_k^+ > r-1) \mathbb{E} \left[ \left. 1\left( \sum_{j=1}^n  Z_{j, \hat T_r^+} \geq k - \left| \hat A_{r-1}^+ \cup \hat I_{r-1}^+  \right|   \right) \sum_{j=1}^n  Z_{j, \hat T_r^+}  \right| \mathcal{G}_{r-1}^+ \right] \right] \\
&\leq \sum_{r=1}^k \mathbb{E}_i \left[ 1(\hat \sigma_k^+ > r-1) \lambda_+^{(m)}({\bf X}_{\hat T_r^+})  \right] \\
&\leq k \sup_{{\bf x} \in \mathcal{S}} \lambda_+^{(m)}({\bf x}).
\end{align*}

The proof for the outbound tree is essentially the same and is therefore omitted. 
\end{proof}

The next result is a technical lemma giving an explicit upper bound for the ratio of independent Poisson random variables.

\begin{lemma} \label{L.RatioPoissons}
Let $X, Y$ be independent Poisson random variables with means $\lambda$ and $\mu$, respectively. Let $a, b \in \mathbb{N}$. Then,
$$E\left[ \frac{a + X}{b + X + Y} \cdot 1(b+X+Y \geq 1) \right] \leq \frac{2a}{b+1} + \frac{\lambda}{\lambda+\mu} (1 - e^{-\lambda-\mu}).$$
\end{lemma}

\begin{proof}
Recall that $X$ given $X+Y = n$ is a Binomial$(n, \lambda/(\lambda+\mu))$. Hence, 
\begin{align*}
&E\left[ \frac{a + X}{b + X + Y} \cdot 1(b+X+Y \geq 1) \right] \\
&= E\left[ \frac{a}{b} \cdot 1(X+Y = 0, b \geq 1) \right] + E\left[ \frac{a + X}{b + X + Y} \cdot 1(X+Y \geq 1) \right] \\
&= \frac{a}{b} 1(b \geq 1) P(X+Y = 0)  + \sum_{n=1}^\infty \frac{E[ a + X | X+Y = n]}{b+n} P(X+Y = n).
\end{align*}
Now use the observation that $X$ given $X+Y = n$ is a binomial with parameters $(n, \lambda/(\mu+\lambda))$ to obtain that
\begin{align*}
&\sum_{n=1}^\infty \frac{E[ a + X | X+Y = n]}{b+n} P(X+Y = n) \\
&= \sum_{n=1}^\infty \frac{a + n\lambda/(\mu+\lambda)}{b+n} P(X+Y = n) \\
&= a \sum_{n=1}^\infty \frac{1}{b+n} P(X+Y = n) + \frac{\lambda}{\mu+\lambda} \sum_{n=1}^\infty \frac{ n}{b+n} P(X+Y = n) \\
&\leq \frac{a}{b+1} P(X+Y \geq 1) + \frac{\lambda}{\mu+\lambda} P(X + Y \geq 1) \\
&= \left( \frac{a}{b+1} + \frac{\lambda}{\lambda+\mu} \right) P(X+Y \geq 1) .
\end{align*}
Using the observation that $(a/b) 1(b \geq 1) \leq 2a/(b+1)$ gives that
\begin{align*}
E\left[ \frac{a + X}{b + X + Y} \cdot 1(b+X+Y \geq 1) \right] &\leq \frac{2a}{b+1} + \frac{\lambda}{\lambda+\mu} P(X + Y \geq 1),
\end{align*}
which completes the proof. 
\end{proof}

The following result constitutes a key step of the proof of Theorem~\ref{T.Coupling} by providing an upper estimate for the distribution of the {\em identities} of the active nodes $\hat T_r^+$ and $\hat T_r^-$.

\begin{lemma} \label{L.TypeDistributions}
Fix $i \in V_n$ and let ${\bf u}_r, {\bf v}_r \in \mathbb{R}^n$, $r \geq 1$, denote the row vectors defined according ${\bf u}_1 = {\bf v}_1 = {\bf e}_i$ and $u_{r,j} = \mathbb{P}_i(\hat A_{r-1}^+ \neq \varnothing, \hat T_r^+ = j)$, $v_{r,j} = \mathbb{P}_i(\hat A_{r-1}^- \neq \varnothing, \hat T_r^- = j)$ for each $j \in V_n$ and $r \geq 1$. Then, 
$${\bf u}_r \leq \sum_{s=0}^{r-1} \binom{r-1}{s} 2^{r-1-s} {\bf e}_i  (\boldsymbol{\Gamma}^{(m,n)})^s \qquad \text{and} \qquad {\bf v}_r \leq \sum_{s=0}^{r-1} \binom{r-1}{s} 2^{r-1-s} {\bf e}_i  (\boldsymbol{\tilde \Gamma}^{(m,n)})^s,$$
where ${\bf e}_i$ is the $i$th canonical  row vector in $\mathbb{R}^n$, and $\boldsymbol{\Gamma}^{(m,n)} = (\gamma_{ij}^{(m,n)})$ and $\boldsymbol{\tilde \Gamma}^{(m,n)} = (\tilde \gamma_{ij}^{(m,n)})$ are the matrices whose $(i,j)$th components are given by
$$\gamma_{ij} = \frac{r_{ji}^{(m,n)}}{\sum_{l=1}^n r_{li}^{(m,n)}} \left(1 - e^{-\sum_{l=1}^n r_{li}^{(m,n)}} \right)  \qquad \text{and} \qquad \tilde \gamma_{ij} = \frac{\tilde r_{ij}^{(m,n)}}{\sum_{l=1}^n \tilde r_{il}^{(m,n)}} \left( 1 - e^{-\sum_{l=1}^n \tilde r_{il}^{(m,n)} }\right).$$
\end{lemma}

\begin{proof}
Let ${\bf W}_t^+ = (W_{t,1}^+, \dots, W_{t,n}^+)$ denote the process that keeps track of the {\em identities} of the vertices in the active set $\hat A_t^+$ for $t \geq 0$. Then,
\begin{align*}
 \mathbb{P}_i( \hat A_{r-1}^+ \neq \varnothing, \hat T_r^+ = l ) &= \mathbb{E}_i\left[ 1\left( \| {\bf W}_{r-1}^+ \|_1 \geq 1 \right) \mathbb{P}(\hat T_r^+ = l | {\bf W}_{r-1}^+) \right] \\
 &= \mathbb{E}_i\left[ \frac{W_{r-1,l}^+}{\sum_{j=1}^n W_{r-1,j}^+} \cdot 1\left( \| {\bf W}_{r-1}^+ \|_1 \geq 1 \right) \right] \\
 &= \mathbb{E}_i \left[  \mathbb{E}\left[ \left. \frac{W_{r-1,l}^+}{\sum_{j=1}^n W_{r-1,j}^+} \cdot 1\left( \| {\bf W}_{r-1}^+ \|_1 \geq 1 \right) \right| {\bf W}_{r-2}^+ \right] 1( \| {\bf W}^+_{r-2} \|_1 \geq 1) \right].
\end{align*}
Moreover, provided $ \| {\bf W}_{r-2}^+ \|_1 \geq 1$, we have
\begin{align*}
&\mathbb{E}\left[ \left. \frac{W_{r-1,l}^+}{\sum_{j=1}^n W_{r-1,j}^+} \cdot 1\left( \| {\bf W}_{r-1}^+ \|_1 \geq 1 \right) \right| {\bf W}_{r-2}^+ \right] \\
&= \sum_{s=1}^n \mathbb{P}(\hat T_{r-1}^+ = s | {\bf W}_{r-2}^+) \,  \mathbb{E}\left[  \left. \frac{W_{r-1,l}^+}{\sum_{j=1}^n W_{r-1,j}^+} \cdot 1\left( \| {\bf W}_{r-1}^+ \|_1 \geq 1 \right)  \right| {\bf W}_{r-2}^+, \{\hat T_{r-1}^+ = s \}\right] \\
&= \sum_{1\leq s\leq n, s \neq l} \frac{W_{r-2,s}^+}{\sum_{t=1}^n W_{r-2,t}^+} \mathbb{E}\left[  \left. \frac{W_{r-2,l}^+ + Z_{ls}}{\sum_{j=1}^n (W_{r-2,j}^+ +  Z_{js} ) - 1 } \cdot 1\left( \sum_{j=1}^n (W_{r-2,j}^+ + Z_{js}) \geq 2 \right) \right| {\bf W}_{r-2}^+ \right]  \\
&\hspace{5mm} + \frac{W_{r-2,l}^+}{\sum_{t=1}^n W_{r-2,t}^+} \mathbb{E}\left[  \left. \frac{W_{r-2,l}^+ + Z_{ll}-1}{\sum_{j=1}^n (W_{r-2,j}^+ +  Z_{jl}) - 1 } \cdot 1\left( \sum_{j=1}^n (W_{r-2,j}^+ + Z_{jl}) \geq 2 \right) \right| {\bf W}_{r-2}^+\right] \\
&\leq \sum_{s=1}^n \frac{W_{r-2,s}^+}{\sum_{t=1}^n W_{r-2,t}^+} \mathbb{E}\left[  \left. \frac{W_{r-2,l}^+ + Z_{ls}}{\sum_{j=1}^n (W_{r-2,j}^+ +  Z_{js} ) - 1 } \cdot 1\left( \sum_{j=1}^n (W_{r-2,j}^+ + Z_{js}) \geq 2 \right) \right| {\bf W}_{r-2}^+ \right]  .
\end{align*}
Now use Lemma~\ref{L.RatioPoissons} with $a = W_{r-2,l}^+$, $b = \sum_{j=1}^n W_{r-2,j}^+ - 1$, $X = Z_{ls}$ and $Y = \sum_{j \neq l} Z_{js}$ to obtain that
\begin{align*}
&\sum_{s=1}^n \frac{W_{r-2,s}^+}{\sum_{t=1}^n W_{r-2,t}^+} \mathbb{E}\left[  \left. \frac{W_{r-2,l}^+ + Z_{ls}}{\sum_{j=1}^n (W_{r-2,j}^+ +  Z_{js} ) - 1 } \cdot 1\left( \sum_{j=1}^n (W_{r-2,j}^+ + Z_{js}) \geq 2 \right) \right| {\bf W}_{r-2}^+ \right] \\
&\leq \sum_{s=1}^n \frac{W_{r-2,s}^+}{\sum_{t=1}^n W_{r-2,t}^+} \left( \frac{2 W_{r-2,l}^+}{\sum_{j=1}^n W_{r-2,j}^+} + \frac{r_{ls}^{(m,n)}}{\sum_{j=1}^n r_{js}^{(m,n)}} (1 - e^{-\sum_{j=1}^n r_{js}^{(m,n)}} ) \right)  \\
&=: \frac{2 W_{r-2,l}^+}{\sum_{j=1}^n W_{r-2,j}^+} + \sum_{s=1}^n \frac{W_{r-2,s}^+}{\sum_{t=1}^n W_{r-2,t}^+} \cdot \gamma_{sl}^{(m,n)}, 
\end{align*}
where
$$\gamma_{sl}^{(m,n)} = \frac{r_{ls}^{(m,n)}}{\sum_{j=1}^n r_{js}^{(m,n)}} \left(1 - e^{-\sum_{j=1}^n r_{js}^{(m,n)}} \right) $$
and we use the convention that $(1 - e^{-0})/0 \equiv 1$. It follows that
\begin{align*}
\mathbb{P}_i(\hat A_{r-1}^+ \neq \varnothing, \hat T_r^+ = l) &\leq \mathbb{E}_i\left[ \left\{ \frac{2 W_{r-2,l}^+}{\sum_{j=1}^n W_{r-2,j}^+} + \sum_{s=1}^n \frac{W_{r-2,s}^+}{\sum_{t=1}^n W_{r-2,t}^+} \cdot \gamma_{ls}^{(m,n)} \right\} 1( \| {\bf W}^+_{r-2} \|_1 \geq 1) \right] \\
&= 2 \mathbb{P}_i(\hat A_{r-2}^+ \neq \varnothing, \hat T_{r-1}^+ = l) + \sum_{s=1}^n \mathbb{P}_i(\hat A_{r-2}^+ \neq \varnothing, \hat T_{r-1}^+ = s) \gamma_{sl}^{(m,n)}.
\end{align*}
In vector notation, 
$${\bf u}_r \leq {\bf u}_{r-1} ( 2{\bf I} + \boldsymbol{\Gamma}^{(m,n)} ) ,$$
where $\boldsymbol{\Gamma}^{(m,n)}$ is the matrix whose $(i,j)$th component is $\gamma_{ij}^{(m,n)}$, ${\bf I}$ is the identity matrix in $\mathbb{R}^{n \times n}$, and the inequality holds component-wise. Iterating $r-1$ times we obtain that
\begin{align*}
{\bf u}_r &\leq {\bf u}_1 (2{\bf I} + \boldsymbol{\Gamma}^{(m,n)} )^{r-1}  = {\bf e}_i (2{\bf I} + \boldsymbol{\Gamma}^{(m,n)} )^{r-1}  = \sum_{s=0}^{r-1} \binom{r-1}{s} 2^{r-1-s} {\bf e}_i ( \boldsymbol{\Gamma}^{(m,n)} )^{s} .
\end{align*}
The proof for ${\bf v}_r$ is essentially the same and is therefore omitted. 
\end{proof}

The following lemma gives an upper bound for the mean of a bounded function with respect to the measures obtained in Lemma~\ref{L.TypeDistributions}.

\begin{lemma} \label{L.BoundNorms}
Define the matrices $\boldsymbol{\Gamma}^{(m,n)}$ and $\boldsymbol{\tilde \Gamma}^{(m,n)}$ according to Lemma~\ref{L.TypeDistributions} and fix $s \in \mathbb{N}_+$. Then, for any positive constants $\{ K_i \}_{i\geq 1}$ and $\{ \tilde K_i \}_{i\geq 1}$, and any vector ${\bf h} = (h_1, \dots, h_n)^T$ such that $0 \leq h_i \leq 1$ for all $1 \leq i \leq n$, we have
$$\frac{1}{n} \norm{(\boldsymbol{\Gamma}^{(m,n)})^s {\bf h} }_1 1(\Omega_{m,n}) \leq  \sum_{t=1}^s (1+\epsilon)^t \mathcal{I}_{m,n}^-(K_t) \prod_{j=0}^{t-1} K_j + (1+\epsilon)^s \prod_{j=0}^s K_j \cdot \frac{1}{n} \norm{ {\bf h} }_1$$
and
$$\frac{1}{n} \norm{(\boldsymbol{\tilde \Gamma}^{(m,n)})^s {\bf h} }_1 1(\Omega_{m,n})  \leq  \sum_{t=1}^s (1+\epsilon)^t \mathcal{I}_{m,n}^+(\tilde K_t) \prod_{j=0}^{t-1} \tilde K_j + (1+\epsilon)^s \prod_{j=0}^s \tilde K_j \cdot \frac{1}{n} \norm{ {\bf h} }_1,$$
with $K_0 = \tilde K_0 = 1$,  and $\Omega_{m,n}$ the event defined in Theorem~\ref{T.Coupling}, 
$$\mathcal{I}_{m,n}^-(K) = \int_{\mathcal{S}} \lambda_{m,n}^-({\bf y}) 1(\lambda_{m,n}^-({\bf y}) > K) \mu_n(d{\bf y}),  \qquad  \mathcal{I}_{m,n}^+(K) = \int_{\mathcal{S}} \lambda_{m,n}^+({\bf y}) 1(\lambda_{m,n}^+({\bf y}) > K) \mu_n(d{\bf y}),$$
$$\lambda_{m,n}^+({\bf x}) = \int_{\mathcal{S}} \kappa_m({\bf y}, {\bf x}) \mu_n(d{\bf y}) \qquad \text{and} \qquad \lambda_{m,n}^-({\bf x}) = \int_{\mathcal{S}} \kappa_m({\bf x}, {\bf y}) \mu_n(d{\bf y}).$$
\end{lemma}

\begin{proof}
For notational convenience we may assume that ${\bf h} = (h_1, \dots, h_n)^T =  (h({\bf X}_i), \dots, h({\bf X}_n))^T$ for some function $h$ on $\mathcal{S}$, in which case we can write
\begin{align*}
(\boldsymbol{\Gamma}^{(m,n)} {\bf h})_i &= \sum_{j=1}^n \gamma_{ij}^{(m,n)} h_j \\
&= \sum_{j=1}^n \frac{\kappa_m({\bf X}_j, {\bf X}_i)  \mu(\mathcal{J}_{\theta(j)}^{(m)}) }{n \mu_n(\mathcal{J}_{\theta(j)}^{(m)})  \lambda_+^{(m)}({\bf X}_i)} (1 - e^{-\lambda_+^{(m)}({\bf X}_i)} )  h({\bf X}_j)   \\
&= \frac{(1 - e^{-\lambda_+^{(m)}({\bf X}_i)} )}{\lambda_+^{(m)}({\bf X}_i)} \int_{\mathcal{S}} \frac{\mu(\mathcal{J}_{\vartheta({\bf y})}^{(m)})}{\mu_n(\mathcal{J}_{\vartheta({\bf y})}^{(m)})} \kappa_m({\bf y}, {\bf X}_i) h({\bf y}) \mu_n(d{\bf y}) \\
&=: \Gamma^{(m,n)} h({\bf X}_i),
\end{align*}
where $\vartheta({\bf x}) = t$ if ${\bf x} \in \mathcal{J}_t^{(m)}$, and $\Gamma^{(m,n)}$ is a linear operator in $\mathcal{S} \times \mathcal{S}$. Hence, on the event $\Omega_{m,n}$, 
\begin{align*}
\frac{1}{n} \norm{(\boldsymbol{\Gamma}^{(m,n)})^s {\bf h} }_1 &= \int_{\mathcal{S}} (\Gamma^{(m,n)})^s h({\bf x}) \mu_n(d{\bf x}) \\
&= \int_{\mathcal{S}} \frac{(1 - e^{-\lambda_+^{(m)}({\bf x})} )}{\lambda_+^{(m)}({\bf x})} \int_{\mathcal{S}} \frac{\mu(\mathcal{J}_{\vartheta({\bf y})}^{(m)})}{\mu_n(\mathcal{J}_{\vartheta({\bf y})}^{(m)})} \kappa_m({\bf y}, {\bf x}) (\Gamma^{(m,n)})^{s-1} h({\bf y}) \mu_n(d{\bf y}) \mu_n(d{\bf x}) \\
&= \int_{\mathcal{S}} \frac{\mu(\mathcal{J}_{\vartheta({\bf y})}^{(m)})}{\mu_n(\mathcal{J}_{\vartheta({\bf y})}^{(m)})} (\Gamma^{(m,n)})^{s-1} h({\bf y})  \int_{\mathcal{S}}\frac{(1 - e^{-\lambda_+^{(m)}({\bf x})})}{\lambda_+^{(m)}({\bf x})}   \kappa_m({\bf y}, {\bf x}) \mu_n(d{\bf x})  \mu_n(d{\bf y}) \\
&\leq (1+\epsilon)  \int_{\mathcal{S}}  (\Gamma^{(m,n)})^{s-1} h({\bf y})  \int_{\mathcal{S}}  \kappa_m({\bf y}, {\bf x}) \mu_n(d{\bf x})  \mu_n(d{\bf y}) \\
&=  (1+\epsilon)  \int_{\mathcal{S}}  \left( (\Gamma^{(m,n)})^{s-1} h({\bf y})  \right) \lambda_{m,n}^-({\bf y})  \mu_n(d{\bf y}) ,
\end{align*}
where we used the observation that on the event $\Omega_{m,n}$ we have $\mu(\mathcal{J}_{\vartheta({\bf y})}^{(m)}) / \mu_n(\mathcal{J}_{\vartheta({\bf y})}^{(m)}) \leq 1+\epsilon$. Now note that since $\boldsymbol{\Gamma}^{(m,n)}$ is a substochastic matrix in $\mathbb{R}^n$, then so is $(\boldsymbol{\Gamma}^{(m,n)})^s$, and therefore, if $0 \leq h({\bf x}) \leq 1$ for all ${\bf x} \in \mathcal{S}$ we have 
$$(\Gamma^{(m,n)})^s h({\bf X}_i) = ((\boldsymbol{\Gamma}^{(m,n)})^s {\bf h} )_i \leq \sum_{j=1}^n ((\boldsymbol{\Gamma}^{(m,n)})^s)_{ij}  \leq 1, \qquad \text{ for all } 1 \leq i \leq n.$$
Therefore, for such $h$ and any $K_1 > 0$, we have that on the event $\Omega_{m,n}$, 
\begin{align*}
\frac{1}{n} \norm{(\boldsymbol{\Gamma}^{(m,n)})^s {\bf h} }_1 &\leq  (1+\epsilon) \int_{\mathcal{S}} \left((\Gamma^{(m,n)})^{s-1} h({\bf y}) \right) K_1 1(\lambda_{m,n}^-({\bf y}) \leq K_1) \mu_n(d{\bf y}) \\
&\hspace{5mm} + (1+\epsilon) \int_{\mathcal{S}} \left((\Gamma^{(m,n)})^{s-1} h({\bf y}) \right) \lambda_{m,n}^-({\bf y}) 1(\lambda_{m,n}^-({\bf y}) > K_1) \mu_n(d{\bf y}) \\
&\leq (1+\epsilon) \left\{  K_1  \int_{\mathcal{S}} (\Gamma^{(m,n)})^{s-1} h({\bf y})  \mu_n(d{\bf y}) +  \int_{\mathcal{S}} \lambda_{m,n}^-({\bf y}) 1(\lambda_{m,n}^-({\bf y}) > K_1) \mu_n(d{\bf y}) \right\} \\
&=: K_1 (1+\epsilon) \cdot \frac{1}{n} \norm{(\boldsymbol{\Gamma}^{(m,n)})^{s-1} {\bf h} }_1 + (1+\epsilon) \mathcal{I}_{m,n}^-(K_1). 
\end{align*}
Repeating the inequality for some $K_{2}$ gives
\begin{align*}
\frac{1}{n} \norm{(\boldsymbol{\Gamma}^{(m,n)})^s {\bf h} }_1 &\leq K_1 (1+\epsilon) \left( K_{2} (1+\epsilon) \cdot \frac{1}{n} \norm{(\boldsymbol{\Gamma}^{(m,n)})^{s-2} {\bf h} }_1  + (1+\epsilon)  \mathcal{I}_{m,n}^-(K_{2}) \right)  \\
&\hspace{5mm} + (1+\epsilon) \mathcal{I}_{m,n}^-(K_1) \\
&= (1+\epsilon)^2 K_1 K_{2} \cdot \frac{1}{n} \norm{(\boldsymbol{\Gamma}^{(m,n)})^{s-2} {\bf h} }_1 + (1+\epsilon)^2 K_1 \mathcal{I}_{m,n}^-(K_{2}) + (1+\epsilon) \mathcal{I}_{m,n}^-(K_1).
\end{align*}
In general, we obtain that
$$\frac{1}{n} \norm{(\boldsymbol{\Gamma}^{(m,n)})^s {\bf h} }_1 \leq \sum_{t=1}^{s} (1+\epsilon)^{t} \mathcal{I}_{m,n}^-(K_{t}) \prod_{j=0}^{t-1} K_{j} + (1+\epsilon)^s \prod_{j=0}^s K_j \cdot \frac{1}{n} \norm{ {\bf h}}_1 ,$$
with $K_0 = 1$. 

The proof for $\boldsymbol{\tilde \Gamma}^{(m,n)}$ is essentially the same and is therefore omitted. 
\end{proof}

The last two preliminary results allow us to compute the limit of the function $H(n,m,k,\epsilon)$ as  $n \to \infty$ and $\epsilon \downarrow 0$. 

\begin{lemma} \label{L.gLimits}
Define ${\bf g}^{(m,n)}$ and ${\bf \tilde g}^{(m,n)}$ according to Theorem~\ref{T.Coupling}. Then, the following limits hold in probability:
$$\limsup_{n \to \infty} \frac{1}{n} \norm{ {\bf g}^{(m,n)} }_1 \leq  \int_{\mathcal{S}}  \left( (1+5\epsilon) \lambda_+({\bf x})  -  \lambda_+^{(m)}({\bf x}) \right) \mu(d{\bf x})$$
and
$$\limsup_{n \to \infty} \frac{1}{n} \norm{ {\bf \tilde g}^{(m,n)} }_1 \leq \int_{\mathcal{S}}  \left( (1+5\epsilon) \lambda_-({\bf x})  -  \lambda_-^{(m)}({\bf x}) \right) \mu(d{\bf x}) .$$
\end{lemma}

\begin{proof}
Note that by ignoring the minimum with one in the definition of $g_i^{(m,n)}$, we obtain that
\begin{align*}
\frac{1}{n} \norm{ {\bf g}^{(m,n)} }_1 &\leq (1+5\epsilon) \frac{1}{n} \sum_{i=1}^n \lambda_n^+({\bf X}_i) -  \frac{1}{n} \sum_{i=1}^n \lambda_{m,n}^+({\bf X}_i) + (1 +\epsilon) \frac{1}{n} \sum_{i=1}^n \sum_{j=1}^n (p_{ji}^{(n)} + q_{ji}^{(n)}) 1(B_{ji}^c). 
\end{align*}
Now note that by Lemma~\ref{bound}, we have that $\lim_{n \to \infty} \frac{1}{n} E\left[  \sum_{i=1}^n \sum_{j=1}^n (p_{ji}^{(n)} + q_{ji}^{(n)}) 1(B_{ji}^c) \right] = 0$, 
and therefore, 
$$\frac{1}{n} \sum_{i=1}^n \sum_{j=1}^n (p_{ji}^{(n)} + q_{ji}^{(n)}) 1(B_{ji}^c) \xrightarrow{P} 0, \qquad n \to \infty.$$ 

Moreover, since  $\kappa_m$ is piecewise constant on $\mathcal{J}_t^{(m)} \times \mathcal{J}_s^{(m)}$ for $1 \leq t,s \leq M_m$, then, 
$$\lambda_{m,n}^+({\bf x}) = \sum_{t=1}^{M_m} 1({\bf x} \in \mathcal{J}_t^{(m)}) \int_{\mathcal{S}} \kappa_m({\bf y}, {\bf x}) \mu_n(d{\bf x}) =  \sum_{t=1}^{M_m} 1({\bf x} \in \mathcal{J}_t^{(m)}) \alpha^{(m,n)}_t,$$
where, by Assumption~\ref{reg}(a), 
$$\alpha^{(m,n)}_t := \sum_{i=1}^{M_m} c_{i,t}^{(m)} \mu_n(\mathcal{J}_i^{(m)}) \xrightarrow{P} \alpha_t^{(m)} :=  \sum_{i=1}^{M_m} c_{i,t}^{(m)} \mu(\mathcal{J}_i^{(m)}) \quad  n \to \infty.$$
Therefore, 
\begin{equation} \label{eq:LimitLambdaMN}
\frac{1}{n} \sum_{i=1}^n \lambda_{m,n}^+({\bf X}_i) = \int_{\mathcal{S}} \sum_{t=1}^{M_m} 1({\bf x} \in \mathcal{J}_t^{(m)}) \alpha^{(m,n)}_t \mu_n(d{\bf x})  \xrightarrow{P} \int_{\mathcal{S}} \lambda_+^{(m)}({\bf x}) \mu(d{\bf x}), \qquad n \to \infty.
\end{equation}

It only remains to show that $n^{-1} \sum_{i=1}^n \lambda_n^+({\bf X}_i) \xrightarrow{P} \int_{\mathcal{S}} \lambda_+({\bf x}) \mu(d{\bf x})$. To this end, let  $({\bf X}^{(n)}, {\bf Y}^{(n)}, {\bf X}, {\bf Y})$ be  a random vector in $\mathcal{S}^4$ such that ${\bf X}^{(n)}$ and ${\bf Y}^{(n)}$ are conditionally i.i.d.~(given $\mathscr{F}$) with common distribution $\mu_n$, ${\bf X}$ and ${\bf Y}$ are i.i.d.~with common distribution $\mu$, and ${\bf X}^{(n)} \to {\bf X}$, ${\bf Y}^{(n)} \to {\bf Y}$ a.s.~as $n \to \infty$; note that such a coupling exist by Assumption~\ref{reg}(a). Note that since $\kappa$ is continuous on $\mathcal{S}^2$, then $\kappa({\bf X}^{(n)}, {\bf Y}^{(n)}) \to \kappa({\bf X}, {\bf Y})$ a.s. Also, by Assumption~\ref{reg}(d), we have convergence of the means, which implies that $\kappa({\bf X}^{(n)}, {\bf Y}^{(n)}) \xrightarrow{L_1} \kappa({\bf X}, {\bf Y})$ as $n \to \infty$. 

Next, define $\mathcal{G}_{\bf X} = \sigma( {\bf X}, \{ {\bf X}^{(n)}: n \geq 1\} )$ and note that
\begin{align*}
\lambda_n^+({\bf X}^{(n)}) &= \mathbb{E}\left[\left. \kappa({\bf Y}^{(n)}, {\bf X}^{(n)}) \right| \mathcal{G}_{\bf X} \right] \qquad \text{and} \qquad \lambda_+({\bf X}) = \mathbb{E} \left[\left. \kappa({\bf Y}, {\bf X}) \right| \mathcal{G}_{\bf X} \right], 
%\lambda_n^-({\bf X}^{(n)}) &= \mathbb{E}\left[\left. \kappa({\bf X}^{(n)}, {\bf Y}^{(n)}) \right| \mathcal{G}_{\bf X} \right] \qquad \text{and} \qquad \lambda_-({\bf X}) = \mathbb{E} \left[\left. \kappa({\bf X}, {\bf Y}) \right| \mathcal{G}_{\bf X} \right].
\end{align*}
It follows that
\begin{align*}
E\left[ \left| \lambda_n^+({\bf X}^{(n)})  - \lambda_+({\bf X}) \right|   \right] &\leq E\left[  \mathbb{E}\left[\left. \left| \kappa({\bf Y}^{(n)}, {\bf X}^{(n)}) - \kappa({\bf Y}, {\bf X})  \right| \right| \mathcal{G}_{\bf X} \right] \right]   \\
&=  E\left[  \left| \kappa({\bf Y}^{(n)}, {\bf X}^{(n)}) - \kappa({\bf Y}, {\bf X})  \right|   \right] \to 0,
\end{align*}
as $n \to \infty$, which in turn implies that
$$\frac{1}{n} \sum_{i=1}^n \lambda_n^+({\bf X}_i) \xrightarrow{P} E[ \kappa({\bf Y}, {\bf X})] =  \int_{\mathcal{S}} \lambda_+({\bf x}) \mu(d{\bf x}), \qquad n \to \infty.$$

The analysis of ${\bf \tilde g}^{(m,n)}$ is essentially the same and is therefore omitted. 
\end{proof}

\begin{lemma} \label{L.HLimit}
For $H(n,m,k,\epsilon)$ defined as in Theorem~\ref{T.Coupling}, and any $m,k \in \mathbb{N}_+$, 
$$\lim_{\epsilon \downarrow 0} \limsup_{n \to \infty} H(n,m,k,\epsilon) \leq \hat H(m,k) \quad \text{in probability} ,$$
where
\begin{align*}
\hat H(m,k) &=  k^{-1} 1(k\geq 2)  + H_{k}^+ \int_{\mathcal{S}} (\lambda_+({\bf x}) -  \lambda_+^{(m)}({\bf x}) ) \mu(d{\bf x})  +  H_{k}^- \int_{\mathcal{S}} (\lambda_-({\bf x}) -  \lambda_-^{(m)}({\bf x}) ) \mu(d{\bf x}) ,
\end{align*}
with $H_{k}^+, H_{k}^-$ finite constants depending only on $k$ and $H_{1}^+ = H_{1}^- = 1$. 
\end{lemma}

\begin{proof}
We start by choosing constants $\{ K_{i}^+ \}_{i \geq 1}$ and $\{ K_{i}^- \}_{i\geq 1}$ such that
$$\int_{\mathcal{S}} \lambda_+({\bf x}) 1(\lambda_+({\bf x}) \geq K_{1}^+) \mu(d{\bf x}) < \frac{1}{k^23^k} , \qquad \int_{\mathcal{S}} \lambda_-({\bf x}) 1(\lambda_-({\bf x}) \geq K_{1}^-) \mu(d{\bf x}) < \frac{1}{k^23^k}$$
and
$$\int_{\mathcal{S}} \lambda_+({\bf x}) 1(\lambda_+({\bf x}) \geq K_{i}^+) \mu(d{\bf x}) < \frac{1}{k^2 3^k \prod_{j=1}^{i-1} K_{j}^+} , \quad \int_{\mathcal{S}} \lambda_-({\bf x}) 1(\lambda_-({\bf x}) \geq K_{i}^-) \mu(d{\bf x}) < \frac{1}{k^2 3^k \prod_{j=1}^{i-1} K_j^-}$$
for $i \geq 2$. Set $K_0^+ = K_0^- = 1$ and note that the constants $\{ K_i^+, K_i^-\}_{i \geq 0}$ depend only on $k$. 

Next, use Lemma~\ref{L.BoundNorms} to obtain that for any $s \in \mathbb{N}_+$, 
$$\frac{1}{n} \norm{(\boldsymbol{\Gamma}^{(m,n)})^s {\bf g}^{(m,n)} }_1 1(\Omega_{m,n}) \leq  \sum_{t=1}^s (1+\epsilon)^t \mathcal{I}_{m,n}^-(K_t^-) \prod_{j=0}^{t-1} K_j^- + (1+\epsilon)^s \prod_{j=0}^s K_j^- \cdot \frac{1}{n} \norm{ {\bf g}^{(m,n)} }_1$$
and
$$\frac{1}{n} \norm{(\boldsymbol{\tilde \Gamma}^{(m,n)})^s {\bf \tilde g}^{(m,n)} }_1 1(\Omega_{m,n}) \leq  \sum_{t=1}^s (1+\epsilon)^t \mathcal{I}_{m,n}^+(K_t^+) \prod_{j=0}^{t-1} K_j^+ + (1+\epsilon)^s \prod_{j=0}^s K_j^+ \cdot \frac{1}{n} \norm{ {\bf \tilde g}^{(m,n)} }_1,$$
where $\mathcal{I}_{m,n}^+(K)$ and  $\mathcal{I}_{m,n}^-(K)$ are defined in Lemma~\ref{L.BoundNorms}. 

Now note that the same arguments leading to \eqref{eq:LimitLambdaMN}, also give that
$$\limsup_{n \to \infty} \mathcal{I}_{m,n}^+(K) = \int_{\mathcal{S}} \lambda_{m,n}^+({\bf x}) 1(\lambda_{m,n}^+({\bf x}) > K) \mu_n(d{\bf y}) \leq \int_{\mathcal{S}} \lambda_+^{(m)}({\bf x}) 1(\lambda_+^{(m)}({\bf x}) \geq K) \mu(d{\bf y}) $$
and
$$\limsup_{n \to \infty} \mathcal{I}_{m,n}^-(K) = \int_{\mathcal{S}} \lambda_{m,n}^-({\bf x}) 1(\lambda_{m,n}^-({\bf x}) > K) \mu_n(d{\bf y}) \leq \int_{\mathcal{S}} \lambda_-^{(m)}({\bf x}) 1(\lambda_-^{(m)}({\bf x}) \geq K) \mu(d{\bf y}) $$
 in probability. 

It follows from our choice of $\{ K_i^+, K_i^-\}_{i \geq 0}$ and from Lemma~\ref{L.gLimits}, that for any $s \in \mathbb{N}$, 
\begin{align*}
&\limsup_{n \to \infty} \frac{1}{n} \norm{(\boldsymbol{\Gamma}^{(m,n)})^s {\bf g}^{(m,n)} }_1 1(\Omega_{m,n})  \\
&\leq 1(s \geq 1) \sum_{t=1}^s (1+\epsilon)^t  \mathcal{I}^-(K_t^-) \prod_{j=0}^{t-1} K_j^-  + (1+\epsilon)^s \prod_{j=0}^s K_j^- \limsup_{n \to \infty} \frac{1}{n} \norm{ {\bf g}^{(m,n)} }_1 \\
&\leq \frac{s (1+\epsilon)^s}{k^2 3^k} + (1+\epsilon)^s \prod_{j=0}^s K_j^- \int_{\mathcal{S}}  \left( (1+5\epsilon) \lambda_+({\bf x})  -  \lambda_+^{(m)}({\bf x}) \right) \mu(d{\bf x})
\end{align*}
and
\begin{align*}
&\limsup_{n \to \infty} \frac{1}{n} \norm{(\boldsymbol{\tilde \Gamma}^{(m,n)})^s {\bf \tilde g}^{(m,n)} }_1 1(\Omega_{m,n}) \\
&\leq \frac{s(1+\epsilon)^s}{k^2 3^k} + (1+\epsilon)^s \prod_{j=0}^s K_j^+ \int_{\mathcal{S}}  \left( (1+5\epsilon) \lambda_-({\bf x})  -  \lambda_-^{(m)}({\bf x}) \right) \mu(d{\bf x}).
\end{align*}

Plugging these estimates into the expression for $H(n,m,k,\epsilon)$ given in Theorem~\ref{T.Coupling} gives 
\begin{align*}
&\lim_{\epsilon \downarrow 0} \limsup_{n \to \infty} \, H(n,m,k,\epsilon)  \\
&\leq \lim_{\epsilon \downarrow 0} \left\{ \lim_{n \to \infty} 1(\Omega_{m,n}^c) +  4\epsilon k^2  + 2\epsilon k^2 \left( 1 +   \sup_{{\bf x} \in \mathcal{S}} \lambda_+^{(m)}({\bf x}) \right)  \right. \\
&\hspace{5mm} + (1+\epsilon)^k \sum_{r=1}^k \sum_{s=0}^{r-1} \binom{r-1}{s} 2^{r-1-s} \left(  \frac{ s  }{k^2 3^k} + \prod_{j=0}^s K_j^- \int_{\mathcal{S}}  \left( (1+5\epsilon) \lambda_+({\bf x})  - \lambda_+^{(m)}({\bf x}) \right) \mu(d{\bf x})  \right) \\
&\hspace{5mm} \left. + (1+\epsilon)^k \sum_{r=1}^k \sum_{s=0}^{r-1} \binom{r-1}{s} 2^{r-1-s} \left(  \frac{s}{k^2 3^k} + \prod_{j=0}^s K_j^+ \int_{\mathcal{S}}  \left( (1+5\epsilon) \lambda_-({\bf x})  - \lambda_-^{(m)}({\bf x}) \right) \mu(d{\bf x}) \right) \right\} \\
&= \frac{2}{k^2 3^k} \sum_{r=1}^k \sum_{s=0}^{r-1} \binom{r-1}{s} 2^{r-1-s} s + \sum_{r=1}^k \sum_{s=0}^{r-1} \binom{r-1}{s} 2^{r-1-s}  \prod_{j=0}^s K_j^- \int_{\mathcal{S}}  \left(  \lambda_+({\bf x})  -  \lambda_+^{(m)}({\bf x}) \right) \mu(d{\bf x}) \\
&\hspace{5mm} + \sum_{r=1}^k \sum_{s=0}^{r-1} \binom{r-1}{s} 2^{r-1-s} \prod_{j=0}^s K_j^+ \int_{\mathcal{S}}  \left(  \lambda_-({\bf x})  -  \lambda_-^{(m)}({\bf x}) \right) \mu(d{\bf x})  ,
\end{align*}
where we used the observation that 
$$1(\Omega_{m,n}) \xrightarrow{P} 0, \qquad n \to \infty$$
for any $\epsilon >0$ by Assumption~\ref{reg}(a). 

Finally, define 
$$H_{k}^+ := \sum_{r=1}^k \sum_{s=0}^{r-1} \binom{r-1}{s} 2^{r-1-s}  \prod_{j=0}^s K_j^- \qquad \text{and} \qquad H_{k}^- := \sum_{r=1}^k \sum_{s=0}^{r-1} \binom{r-1}{s} 2^{r-1-s}  \prod_{j=0}^s K_j^+,$$
and note that
\begin{align*}
\frac{2}{k^23^k} \sum_{r=1}^k \sum_{s=0}^{r-1} \binom{r-1}{s} 2^{r-1-s} s &= \frac{2}{k^2 3^k} \sum_{r=1}^k 3^{r-1}\sum_{s=0}^{r-1} \binom{r-1}{s} (2/3)^{r-1-s} (1/3)^s s \\
&= \frac{2}{k^2 3^k} \sum_{r=1}^k 3^{r-2} (r-1) = \frac{2}{k^23^k} \cdot \frac{1 + (2k-3) 3^{k-1} }{4} \\
&< k^{-1} 1(k\geq 2). 
\end{align*}
This completes the proof. 
\end{proof}

\bigskip

We are now ready to give the proof of Theorem~\ref{T.Coupling}. 

\begin{proof}[Proof of Theorem~\ref{T.Coupling}]
To start, note that
\begin{align*}
&\mathbb{P}_i \left( \{ \tau^+ \leq\sigma_k^+ \} \cup \{ \tau^- \leq \sigma_k^-\}  \right)  \\
&= \left\{  \mathbb{P}_i \left(  \tau^+ \leq\sigma_k^+  \right) + \mathbb{P}_i \left( \{ \tau^+ > \sigma_k^+ \} \cap \{ \tau^- \leq \sigma_k^-\}  \right) \right\} 1(\Omega_{m,n}) + 1(\Omega_{m,n}^c),
\end{align*}
where the event $\Omega_{m,n}$ is defined in the statement of the theorem. To analyze the two probabilities, define $\mathcal{G}_m^+$ to be the sigma-algebra containing all the information of the exploration process of the in-component of vertex $i$ up to the end of Step $m$ and including the {\em identity} of $T_{m+1}^+$, and let $\mathcal{G}_m^-$ be the sigma-algebra containing all the information of the exploration process of the in-component of vertex $i$ up to Step $\sigma_k^+$, and of its out-component up to the end of Step $m$, including the {\em identity} of $T_{m+1}^-$; note that $\mathcal{G}_m^+ \subseteq \mathcal{G}_r^-$ for all $0 \leq m \leq \sigma_k^+$ and any $r \geq 0$. Next, for any $r \geq 1$ define the events
\begin{align*}
E_{r}^+ &= \{ |I_{r}^+| + |A_{r}^+| < k \}, \\
E_r^- &= \{ |I_{r}^-| + |A_{r}^-| < k \} ,\\ 
C_i^+(r) &= \left\{ \max_{j \in V_n, j \notin \{i \} \cup A_{r-1}^+ \cup I_{r-1}^+ }  |Z_{ji} -Y_{ji}| + \max_{j \in A_{r-1}^+ \cup I_{r-1}^+} Z_{ji} + Z_{ii}^* = 0 \right\},  \\
C_i^-(r) &= \left\{ \max_{j \in V_n, j \notin \{i\} \cup I_{\sigma_k^+}^+ \cup A_{\sigma_k^+}^+ \cup A_{r-1}^- \cup I_{r-1}^-} | \tilde Z_{ij} -Y_{ij}| + \max_{j \in A_{r-1}^- \cup I_{r-1}^-} \tilde Z_{ij} + \max_{j \in \{i\} \cup I_{\sigma_k^+}^+ \cup A_{\sigma_k^+}^+} \tilde Z_{ji}^* =0 \right\}.  
\end{align*}
Now, use Remark~\ref{R.BreakCoupling} to obtain that on the event $\Omega_{m,n}$, 
\begin{align*}
& \mathbb{P}_i \left(  \tau^+ \leq\sigma_k^+  \right) + \mathbb{P}_i \left( \{ \tau^+ > \sigma_k^+ \} \cap \{ \tau^- \leq \sigma_k^-\}  \right)  \\
&= \sum_{r=1}^k \left\{  \mathbb{P}_i \left(  r = \tau^+ \leq \sigma_k^+   \right)  + \mathbb{P}_i \left( \{ \tau^+  > \sigma_k^+ \} \cap \{ r = \tau^-  \leq \sigma_k^- \}  \right) \right\} \\
&\leq \sum_{r=1}^k \mathbb{P}_i \left( \tau^+ > r-1,  A_{r-1}^+ \neq \varnothing, E_{r-1}^+ \cap (C_{T_r^+}^+(r))^c   \right) \\
&\hspace{5mm} + \sum_{r=1}^k  \mathbb{P}_i \left( \tau^+  > \sigma_k^+, \, \tau^- > r-1,  A_{r-1}^- \neq \varnothing, E_{r-1}^- \cap (C_{T_r^-}^-(r))^c  \right) \\
&= \sum_{r=1}^k \mathbb{E}_i \left[ 1( \tau^+ > r-1, A_{r-1}^+ \neq \varnothing, E_{r-1}) \mathbb{P}_i \left( \left.  (C_{T_r^+}^+(r))^c \right| \mathcal{G}_{r-1}^+ \right)  \right] \\
&\hspace{5mm} + \sum_{r=1}^k  \mathbb{E}_i \left[ 1( \tau^+  > \sigma_k^+, \, \tau^- > r-1,  A_{r-1}^- \neq \varnothing, E_{r-1}) \mathbb{P}_i \left( \left. (C_{T_r^-}^-(r))^c \right| \mathcal{G}_{r-1}^- \right) \right]. 
\end{align*}

To analyze the two conditional probabilities in the last expressions, note that the union bound and the independence of the $\{U_{ij}: 1\leq i,j \leq n\}$ from everything else give
\begin{align}
\mathbb{P}\left( \left. (C_{T_r^+}^+(r))^c  \right| \mathcal{G}_{r-1}^+ \right) 
&\leq \sum_{j\in V_n, j \notin \{T_r^+\} \cup A_{r-1}^+ \cup I_{r-1}^+ } \mathbb{P}( |Z_{j,T_r^+} - Y_{j,T_r^+}| > 0 | T_r^+)  \label{eq:EdgeDiscrepanciesIn}  \\
&\hspace{5mm} + \sum_{j \in \{T_r^+\} \cup A_{r-1}^+ \cup I_{r-1}^+}  \mathbb{P}(Z_{j,T_r^+} \geq 1 | T_r^+) , \label{eq:SelfLoopsIn}
\end{align}
and
\begin{align}
\mathbb{P}\left( \left. (C_{T_r^-}^-(r))^c  \right| \mathcal{G}_{r-1}^- \right) 
&\leq \sum_{j\in V_n, j \notin \{T_r^-\} \cup I_{\sigma_k^+}^+ \cup A_{\sigma_k^+}^+ \cup A_{r-1}^- \cup I_{r-1}^- } \mathbb{P}( |\tilde Z_{T_r^-,j} - Y_{T_r^-,j}| > 0 | T_r^-)  \label{eq:EdgeDiscrepanciesOut}  \\
&\hspace{5mm} +  \sum_{j\in \{T_r^-\} \cup I_{\sigma_k^+}^+ \cup A_{\sigma_k^+}^+ \cup A_{r-1}^- \cup I_{r-1}^-} \mathbb{P}(\tilde Z_{T_r^-,j} \geq 1 | T_r^-) . \label{eq:SelfLoopsOut}
\end{align}

To analyze \eqref{eq:EdgeDiscrepanciesIn} note that on the event $B_{ji} $ we have that $(1-\epsilon) q_{ji}^{(n)} \leq p_{ji}^{(n)} < (1+\epsilon) q_{ji}^{(n)} \leq (1+\epsilon)\epsilon <1$, which implies that on the event $B_{ji}$ we have
\begin{align*}
\mathbb{P}( |Y_{ji} - Z_{ji}| > 0) &= (p_{ji}^{(n)} - r_{ji}^{(m,n)}) 1( p_{ji}^{(n)} > r_{ji}^{(m,n)}) \\
&\hspace{5mm} + (e^{-r_{ji}^{(m,n)}} - 1 + p_{ji}^{(n)}) 1(1-e^{-r_{ji}^{(m,n)}} < p_{ji}^{(n)} \leq r_{ji}^{(m,n)}) \\
&\hspace{5mm} + (1 - p_{ji}^{(n)} - e^{-r_{ji}^{(m,n)}} ) 1( p_{ji}^{(n)} < 1 - e^{-r_{ji}^{(m,n)}}) + 1 - e^{-r_{ji}^{(m,n)}} - e^{-r_{ji}^{(m,n)}} r_{ji}^{(m,n)} \\
&\leq (p_{ji}^{(n)} - r_{ji}^{(m,n)}) 1( p_{ji}^{(n)} > r_{ji}^{(m,n)}) + (p_{ji}^{(n)} - r_{ji}^{(m,n)} ) 1(1-e^{-r_{ji}^{(m,n)}} < p_{ji}^{(n)} \leq r_{ji}^{(m,n)}) \\
&\hspace{5mm} + ( r_{ji}^{(m,n)} - p_{ji}^{(n)}  ) 1( p_{ji}^{(n)} < 1 - e^{-r_{ji}^{(m,n)}}) + (r_{ji}^{(m,n)})^2 \\
&= | p_{ji}^{(n)} - r_{ji}^{(m,n)} | + (r_{ji}^{(m,n)})^2,
\end{align*}
where we have used the inequalities $e^{-x} -1 \leq -x + x^2/2$, $1-e^{-x} \leq x$, and $1 - e^{-x} - e^{-x} x \leq x^2/2$ for $x \geq 0$. It follows that if we let $q_{ji}^{(m,n)} = \kappa_m({\bf X}_j, {\bf X}_i)/n$, then, on the event $\Omega_{m,n}$, where we have $(1-\epsilon) q_{ji}^{(m,n)} \leq r_{ji}^{(m,n)} \leq (1+\epsilon) q_{ji}^{(m,n)}$, we have
\begin{align*}
&\mathbb{P}( |Y_{ji} - Z_{ji}| > 0) 1(B_{ji}) \\
&\leq \left( | p_{ji}^{(n)} - r_{ji}^{(m,n)} | + (r_{ji}^{(m,n)})^2  \right) 1(B_{ji} ) \\
&\leq \left(  | p_{ji}^{(n)} - q_{ji}^{(n)} |  + q_{ji}^{(n)} - q_{ji}^{(m,n)} + | q_{ji}^{(m,n)} - r_{ji}^{(m,n)} |  + (1 + \epsilon)^2 (q_{ji}^{(m,n)})^2 \right) 1(B_{ji})  \\
&\leq \epsilon q_{ji}^{(n)} +   q_{ji}^{(n)} - q_{ji}^{(m,n)} + \epsilon q_{ji}^{(m,n)}  + (1 + \epsilon)^2 \epsilon q_{ji}^{(m,n)} \\
&\leq  (1+5\epsilon) q_{ji}^{(n)} - q_{ji}^{(m,n)}.
\end{align*}

On the other hand, note that on the event $\Omega_{m,n}$ we have
\begin{align*}
\mathbb{P}( |Y_{ji} - Z_{ji}| > 0) 1(B_{ji}^c) &\leq \mathbb{P}( Y_{ji} + Z_{ji} > 0) 1(B_{ji}^c) \\
&\leq   \min\left\{1, p_{ji}^{(n)} + r_{ji}^{(m,n)} \right\} 1(B_{ji}^c) \\
&\leq (1 +\epsilon) (p_{ji}^{(n)} + q_{ji}^{(n)}) 1(B_{ji}^c) =: \mathcal{B}_n(j,i) .
\end{align*}
Hence, on the event $\Omega_{m,n}$,  \eqref{eq:EdgeDiscrepanciesIn} is bounded from above by
\begin{align*}
&  \sum_{j\in V_n} \left\{   (1+5\epsilon) q_{j,T_r^+}^{(n)} - q_{j,T_r^+}^{(m,n)}   \right\} + \sum_{j\in V_n, j \notin \{T_r^+\} \cup A_{r-1}^+ \cup I_{r-1}^+ } \mathcal{B}_n(j,T_r^+) \\
&\leq  (1+5\epsilon) \lambda_n^+({\bf X}_{T_r^+}) - \lambda_{m,n}^+({\bf X}_{T_r^+})  +   \sum_{j\in V_n, j \notin \{T_r^+\} \cup A_{r-1}^+ \cup I_{r-1}^+ } \mathcal{B}_n(j,T_r^+)  ,
\end{align*}
where we have used the observation that $\sum_{j=1}^n q_{ji}^{(n)} = \lambda_n^+({\bf X}_i)$ and $\sum_{j=1}^n q_{ji}^{(m,n)} = \int_{\mathcal{S}} \kappa_m({\bf y}, {\bf X}_i) \mu_n(d{\bf y}) = \lambda_{m,n}^+({\bf X}_i)$.

To analyze \eqref{eq:SelfLoopsIn}, note that on the event $\Omega_{m,n}$, 
\begin{align*}
\mathbb{P}(Z_{ji} \geq 1)  &= 1 - e^{-r_{ji}^{(m,n)}}  \leq r_{ji}^{(m,n)} \leq  (1+\epsilon) q_{ji}^{(m,n)} 1(B_{ji}) + (1+\epsilon) q_{ji}^{(m,n)} 1(B_{ji}^c) \leq 2\epsilon + \mathcal{B}_n(j,i) .
\end{align*}
We have thus obtained that, on the event $\Omega_{m,n}$, 
\begin{align*}
\mathbb{P}\left( \left. (C_{T_r^+}^+)^c  \right| \mathcal{G}_{r-1}^+ \right) &\leq  \mathcal{E}_{m,n}^+({\bf X}_{T_r^+})  +  \sum_{j\in V_n, j \notin \{T_r^+\} \cup A_{r-1}^+ \cup I_{r-1}^+ } \mathcal{B}_n(j,T_r^+) \\
&\hspace{5mm} + \sum_{j \in \{T_r^+\} \cup A_{r-1}^+ \cup I_{r-1}^+ }  \left( 2\epsilon +  \mathcal{B}_n(j, T_r^+) \right) \\ 
&\leq  \mathcal{E}_{m,n}^+({\bf X}_{T_r^+}) + 2\epsilon  \left| \{T_r^+\} \cup A_{r-1}^+ \cup I_{r-1}^+ \right| + \sum_{j \in V_n} \mathcal{B}_n(j,T_r^+) .
\end{align*}

The same arguments yield that, on the event $\Omega_{m,n}$, \eqref{eq:EdgeDiscrepanciesOut} is bounded by
$$ \mathcal{E}_{m,n}^-({\bf X}_{T_r^-}) + \sum_{j \in V_n, j \notin \{ T_r^-\} \cup I_{\sigma_k^+}^+ \cup A_{\sigma_k^+}^+ \cup A_{r-1}^- \cup I_{r-1}^-} \mathcal{B}_n(T_r^-,j),$$
and \eqref{eq:SelfLoopsOut} is bounded by
$$\sum_{j \in \{ T_r^-\} \cup I_{\sigma_k^+}^+ \cup A_{\sigma_k^+}^+ \cup A_{r-1}^- \cup I_{r-1}^-} \left( 2\epsilon + \mathcal{B}_n(T_r^-,j) \right).$$
Hence, on the event $\Omega_{m,n}$, 
$$\mathbb{P}\left( \left. (C_{T_r^-}^-)^c  \right| \mathcal{G}_{r-1}^- \right) \leq  \mathcal{E}_{m,n}^-({\bf X}_{T_r^-}) + 2\epsilon  \left| \{ T_r^-\} \cup I_{\sigma_k^+}^+ \cup A_{\sigma_k^+}^+ \cup A_{r-1}^- \cup I_{r-1}^- \right| + \sum_{j \in V_n} \mathcal{B}_n(T_r^-,j).$$

It follows that on the event $\Omega_{m,n}$ we have
\begin{align*}
& \mathbb{P}_i \left(  \tau^+ \leq\sigma_k^+  \right) + \mathbb{P}_i \left( \{ \tau^+ > \sigma_k^+ \} \cap \{ \tau^- \leq \sigma_k^-\}  \right)  \\
&\leq \sum_{r=1}^k \mathbb{E}_i \left[ 1( \tau^+ > r-1,  A_{r-1}^+ \neq \varnothing, E_{r-1}) \min \left\{ 1, \,   \mathcal{E}_{m,n}^+({\bf X}_{T_r^+}) \right. \right. \\
&\hspace{35mm} \left. \left. + 2\epsilon \left| \{T_r^+\} \cup A_{r-1}^+ \cup I_{r-1}^+ \right|  + \sum_{j \in V_n} \mathcal{B}_n(j,T_r^+)   \right\}   \right] \\
&\hspace{5mm} + \sum_{r=1}^k  \mathbb{E}_i \left[ 1( \tau^+  > \sigma_k^+, \, \tau^- > r-1,  A_{r-1}^- \neq \varnothing, E_{r-1}) \min \left\{ 1, \,  \mathcal{E}_{m,n}^-({\bf X}_{T_r^-}) \right. \right. \\
&\hspace{35mm} \left. \left.  + 2\epsilon  \left| \{ T_r^-\} \cup I_{\sigma_k^+}^+ \cup A_{\sigma_k^+}^+ \cup A_{r-1}^- \cup I_{r-1}^- \right| + \sum_{j \in V_n} \mathcal{B}_n(T_r^-,j)  \right\} \right] \\
&\leq \sum_{r=1}^k \mathbb{E}_i \left[ 1( \hat A_{r-1}^+ \neq \varnothing) \min \left\{  1, \,  \mathcal{E}_{m,n}^+({\bf X}_{\hat T_r^+}) + 2 \epsilon k + \sum_{j \in V_n} \mathcal{B}_n(j,\hat T_r^+)   \right\}   \right] \\
&\hspace{5mm} + \sum_{r=1}^k  \mathbb{E}_i \left[ 1(  \hat A_{r-1}^- \neq \varnothing) \min \left\{ 1, \,  \mathcal{E}_{m,n}^-({\bf X}_{\hat T_r^-})   + 2\epsilon \left( k + \left|  \hat I_{\hat \sigma_k^+}^+ \cup \hat A_{\hat \sigma_k^+}^+ \right|    \right)     + \sum_{j \in V_n} \mathcal{B}_n(\hat T_r^-,j)  \right\} \right] ,
\end{align*}
where $\hat T_{r}^+$ and $\hat T_{r}^-$ are the {\em identities} of the $r$th ``active" nodes to be explored in the inbound and outbound multi-type branching processes, respectively, and $\hat \sigma_k^\pm = \inf\{ t \geq 1: | \hat A_t^\pm| + | \hat I_t^\pm| \geq k \text{ or } \hat A_t^\pm = \varnothing\}$. Furthermore, using the inequality $\min\{ 1, x+y\} \leq x + \min\{1, y\}$ we obtain
\begin{align*}
&\left( \mathbb{P}_i \left(  \tau^+ \leq\sigma_k^+  \right) + \mathbb{P}_i \left( \{ \tau^+ > \sigma_k^+ \} \cap \{ \tau^- \leq \sigma_k^-\}  \right) \right) 1(\Omega_{m,n}) \\
&\leq \left( 4\epsilon k^2 + 2\epsilon k \mathbb{E}_i\left[ \left| \hat I^+_{\hat \sigma_k^+} \cup \hat A_{\hat \sigma_k^+}^+ \right| \right]  + \sum_{r=1}^k \mathbb{E}_i \left[ 1( \hat A_{r-1}^+ \neq \varnothing) g_{\hat T_r^+}   \right]  + \sum_{r=1}^k  \mathbb{E}_i \left[ 1(  \hat A_{r-1}^- \neq \varnothing) \tilde g_{\hat T_r^-} \right]  \right) 1(\Omega_{m,n}),
\end{align*}
where 
$$g^{(m,n)}_l = \min\left\{ 1, \, \mathcal{E}_{m,n}^+({\bf X}_j) + \sum_{j =1}^n \mathcal{B}_n(j,l) \right\} \quad \text{and} \quad \tilde g^{(m,n)}_l = \min\left\{ 1, \, \mathcal{E}_{m,n}^-({\bf X}_j) + \sum_{j =1}^n \mathcal{B}_n(l,j) \right\}.$$

Next, use Lemma~\ref{L.TypeDistributions} to obtain that for $r \geq 1$, 
\begin{align*}
\mathbb{E}_i \left[ 1\left(\hat A_{r-1}^+ \neq \varnothing \right) g_{\hat T_r^+} \right]  &= \sum_{l \in V_n} \mathbb{P}_i(\hat A_{r-1}^+ \neq \varnothing, \hat T_r^+ = l)  g_l^{(m,n)} \\
&\leq \sum_{l \in V_n} \sum_{s=0}^{r-1} \binom{r-1}{s} 2^{r-1-s} ({\bf e}_i (\boldsymbol{\Gamma}^{(m,n)})^s)_l g_l^{(m,n)} \\
&= \sum_{s=0}^{r-1} \binom{r-1}{s} 2^{r-1-s} {\bf e}_i (\boldsymbol{\Gamma}^{(m,n)})^s {\bf g}^{(m,n)}  ,
\end{align*}
where ${\bf g}^{(m,n)} = (g^{(m,n)}_1, \dots, g^{(m,n)}_n)'$ and the matrix $\boldsymbol{\Gamma}^{(m,n)}$ is defined in Lemma~\ref{L.TypeDistributions}. Similarly, 
$$\mathbb{E}_i \left[ 1\left(\hat A_{r-1}^- \neq \varnothing \right) \tilde g_{\hat T_r^-} \right] \leq \sum_{s=0}^{r-1} \binom{r-1}{s} 2^{r-1-s} {\bf e}_i (\boldsymbol{\tilde \Gamma}^{(m,n)})^s {\bf \tilde g}^{(m,n)},$$
with ${\bf \tilde g}^{(m,n)} = (\tilde g^{(m,n)}_1, \dots, \tilde g^{(m,n)}_n)'$ and $\boldsymbol{\tilde \Gamma}^{(m,n)}$ as defined in Lemma~\ref{L.TypeDistributions}.

Averaging over all $1 \leq i \leq n$ and using Lemma \ref{L.MeanTree} to bound $n^{-1} \sum_{i=1}^n \mathbb{E}_i\left[ \left| \hat I^+_{\hat \sigma_k^+} \cup \hat A_{\hat \sigma_k^+}^+ \right| \right]$, we obtain
\begin{align*}
&\frac{1}{n} \sum_{i=1}^n \left( \mathbb{P}_i \left(  \tau^+ \leq\sigma_k^+  \right) + \mathbb{P}_i \left( \{ \tau^+ > \sigma_k^+ \} \cap \{ \tau^- \leq \sigma_k^-\}  \right) \right) 1(\Omega_{m,n}) \\
&\leq 4 \epsilon  k^2 + 2\epsilon  k^2 \left( 1 +  \sup_{{\bf x} \in \mathcal{S}} \lambda_+^{(m)}({\bf x}) \right) \\
&\hspace{5mm} +1(\Omega_{m,n})  \sum_{r=1}^k \sum_{s=0}^{r-1} \binom{r-1}{s} 2^{r-1-s} \left\{ \frac{1}{n}  \norm{ ( \boldsymbol{ \Gamma}^{(m,n)})^s {\bf g}^{(m,n)} }_1 + \frac{1}{n} \norm{   ( \boldsymbol{\tilde \Gamma}^{(m,n)})^s  {\bf \tilde g}^{(m,n)} }_1 \right\} \\
&=: H(n,m,k,\epsilon) - 1(\Omega_{m,n}^c).
\end{align*}

The upper bound for the limit of $H(n,m,k,\epsilon)$ as $n \to \infty$ is given in Lemma~\ref{L.HLimit}. This completes the proof. 
\end{proof}

As a last proof in this section, we use Theorem~\ref{T.Coupling} to prove Theorem~\ref{degree}, the result establishing the limiting distribution of the degrees in $G_n(\kappa(1+\varphi_n))$.  The latter can also be proven directly using similar arguments as those used in the proof of Theorem~\ref{T.Coupling}, but we choose to do it this way to avoid repetition. 

\begin{proof}[Proof of Theorem \ref{degree}]
Let 
$$D_{n,i}^+ = \sum_{j \neq i} Y_{ji} \qquad \text{and} \qquad  D_{n,i}^- = \sum_{j\neq i} Y_{ij}$$
and define
$$Z_{n,i}^+ = \sum_{j =1}^n Z_{ji} \qquad \text{and} \qquad Z_{n,i}^- = \sum_{j=1}^n \tilde Z_{ij},$$
where $Z_{ji}$ is Poisson with mean $r_{ji}^{(m,n)}$ and $\tilde Z_{ij}$ is Poisson with mean $\tilde r_{ij}^{(m,n)}$. Then, 
\begin{align*}
\left( D_{n,\xi}^+, D_{n,\xi}^- \right) &= \left( D_{n,\xi}^+ - Z_{n,\xi}^+, \, D_{n,\xi}^- - Z_{n,\xi}^- \right) + \left( Z_{n,\xi}^+, \, Z_{n,\xi}^- \right),
\end{align*}
where since $\sum_{j=1}^n r_{ji}^{(m,n)} = \lambda_+^{(m)}({\bf X}_i)$ and $\sum_{j=1}^n \tilde r_{ij}^{(m,n)} = \lambda_-^{(m)}({\bf X}_i)$, we obtain that
\begin{align*}
\mathbb{P}\left( Z_{n,\xi}^+ = k, Z_{n,\xi}^- = l \right) &= \frac{1}{n} \sum_{i=1}^n \frac{e^{-\lambda_+^{(m)}({\bf X}_i)} (\lambda_+^{(m)}({\bf X}_i))^k}{k!} \cdot \frac{e^{-\lambda_-^{(m)}({\bf X}_i)} (\lambda_-^{(m)}({\bf X}_i))^l}{l!} \\
&\xrightarrow{P} \int_{\mathcal{S}} \frac{e^{-\lambda_+^{(m)}({\bf x})} (\lambda_+^{(m)}({\bf x}))^k}{k!} \cdot \frac{e^{-\lambda_-^{(m)}({\bf x})} (\lambda_-^{(m)}({\bf x}))^l}{l!} \, \mu(d{\bf x})
\end{align*}
for any $k,l \geq 0$, as $n \to \infty$ (by the bounded convergence theorem). Moreover, by Theorem~\ref{T.Coupling},
\begin{align*}
\mathbb{P}\left( |D_{n,\xi}^+ - Z_{n,\xi}^+| + |D_{n,\xi}^- - Z_{n,\xi}^-| > 0 \right) &= \frac{1}{n} \sum_{i=1}^n \mathbb{P}_i\left( \{ \tau^+ \leq \sigma_1^+\} \cup \{ \tau^- \leq \sigma_1^-\} \right) \leq H(n,m,1,\epsilon)
\end{align*}
for any $0 < \epsilon < 1/2$. Therefore, for $(Z^+_{(m)}, Z^-_{(m)})$ constructed on the same probability space as $\left( Z_{n,\xi}^+, Z_{n,\xi}^- \right)$, with $Z^+_{(m)}$ and $Z^-_{(m)}$ conditionally independent (given {\bf X}) Poisson random variables with parameters $\lambda_+^{(m)}({\bf X})$ and $\lambda_-^{(m)}({\bf X})$, and ${\bf X}$ distributed according to $\mu$, we obtain that
\begin{align*}
&\lim_{\epsilon \downarrow 0} \limsup_{n \to \infty} \mathbb{P}\left( |D_{n,\xi}^+ - Z^+_{(m)}| + |D_{n,\xi}^- - Z^-_{(m)}| > 0 \right) \\
&\leq \lim_{\epsilon \downarrow 0} \limsup_{n \to \infty} H(n,m,1,\epsilon) \\
&\leq  \int_{\mathcal{S}} (\lambda_+({\bf x}) - \lambda_+^{(m)}({\bf x})) \mu(d{\bf x}) + \int_{\mathcal{S}} (\lambda_-({\bf x}) - \lambda_-^{(m)}({\bf x})) \mu(d{\bf x}),
\end{align*}
which converges to zero as $m \nearrow \infty$ (by the monotone convergence theorem). Taking the limit as $m \nearrow \infty$ and noting that $(Z^+_{(m)}, Z^-_{(m)}) \to (Z^+, Z^-)$ a.s., where $(Z^+, Z^-)$ are conditionally independent (given ${\bf X}$) Poisson random variables with parameters $\lambda_+({\bf X})$ and $\lambda_-({\bf X})$, gives the weak convergence statement of the theorem. 

To obtain the convergence of the expectations note that
$$E[ D_{n,\xi}^+] = E[ D_{n,\xi}^-] =  \frac{1}{n} E\left[ \sum_{i=1}^n \sum_{j=1}^n p_{ji}^{(n)} \right] \rightarrow \iint_{\mathcal{S}^2} \kappa({\bf x}, {\bf y}) \mu(d{\bf x}) \mu(d{\bf y})$$
as $n \to \infty$ by Assumption~\ref{reg}(d). Now note that
$$\iint_{\mathcal{S}^2} \kappa({\bf x}, {\bf y}) \mu(d{\bf x}) \mu(d{\bf y}) = E[\lambda_+({\bf X})] = E[\lambda_-({\bf X})] = E[Z^+] = E[Z^-].$$
This completes the proof. 
\end{proof}

\subsubsection{Size of the Largest Strongly Connected Component}

This last section of the paper contains the proof of Theorem~\ref{giant}, the phase transition for the existence of a giant strongly connected component.  As mentioned earlier, the idea is to use Theorem~\ref{T.Coupling} to couple the exploration of the graph $G_n(\kappa(1+\varphi_n))$ starting from a given vertex with a double tree $(\mathcal{T}_\mu^+(\kappa_m), \mathcal{T}_\mu^-(\kappa_m))$ for a kernel $\kappa_m$ that takes at most a finite number of different values. 

Recall from Section~\ref{S.PhaseTransition} that $(\mathcal{T}_\mu^+(\kappa; {\bf x}), \mathcal{T}_\mu^-(\kappa; {\bf x}))$ denotes the double multi-type Galton-Watson process having root of type $\bold{x} \in \mathcal{S}$, and whose offspring distributions are given by \eqref{eq:OffspringDistr}. Let $\rho_{+}^{\geq k}(\kappa;\bold{x})$ (respectively,  $\rho_{-}^{\geq k}(\kappa;\bold{x})$) be the probability that the total population of $\mathcal{T}_\mu^+(\kappa; {\bf x})$ (respectively, $\mathcal{T}_\mu^-(\kappa; {\bf x})$) is at least $k$. Define also $\rho_{+}(\kappa;\bold{x})$ (respectively, $\rho_{-}(\kappa;\bold{x}))$ to be its survival probability, i.e., the probability that its total population is infinite. The averaged joint survival probability is defined as 
 $$\rho(\kappa)=\int_{\mathcal{S}}\rho_{-}(\kappa;\bold{x})\rho_{+}(\kappa; \bold{x})\mu(d\bold{x}).$$
Similarly, for any $k\in \mathbb{N}_+$, we define $\rho^{\geq k}(\kappa)=\int_{\mathcal{S}}\rho_-^{\geq k}(\kappa;\bold{x})\rho_+^{\geq k}(\kappa;\bold{x}) \mu(d\bold{x}).$ 
  
In addition, we will require from here on that the kernel $\kappa_m$ be regular finitary (see Definition~\ref{D.Finitary}) and quasi-irreducible (see Definition~\ref{D.Irreducible}).   The following lemma is taken from \cite{bollobas2007phase} and it provides the existence of a sequence of partitions $\{\mathscr{J}_m\}_{m \geq 1}$ of $\mathcal{S}$ over which we can define a sequence of regular finitary kernels. 

\begin{lemma}[Lemma 7.1 in  \cite{bollobas2007phase}] \label{L.Partition}
there exists a sequence of partitions $\{ \mathscr{J}_m: m \geq 1\}$ of $\mathcal{S}$, with  $\mathscr{J}_m = \{ \mathcal{J}_1^{(m)}, \dots, \mathcal{J}^{(m)}_{M_m}\}$, such that
\begin{enumerate} \renewcommand{\labelenumi}{\roman{enumi})}
\item each $\mathcal{J}_{i}^{(m)}$ is measurable and $\mu(\partial \mathcal{J}_{i}^{(m)})=0$,
\item for each $m$, $\mathscr{J}_{m+1}$ refines $\mathscr{J}_m$, i.e., each $\mathcal{J}_{i}^{(m)} = \bigcup_{j\in I_{i}^{(m)}}\mathcal{J}_{j}^{(m+1)}$ for some index set $I_{i}^{(m)}$,
\item for a.e. $\bold{x}\in\mathcal{S}$, diam$(\mathcal{J}_{\vartheta(\bold{x})}^{(m)})\rightarrow 0$ as $m\rightarrow\infty$, where $\vartheta(\bold{x})=j$ if and only if $\bold{x}\in \mathcal{J}_{j}^{(m)}$.
\end{enumerate}
\end{lemma}

Before we construct the sequence of quasi-irreducible regular finitary kernels that we need, we define for notational convenience the following relation.

\begin{defn} \label{D.Accessible}
Let $\tilde \kappa$ be a kernel on $\mathcal{S} \times \mathcal{S}$ and let $\mathscr{J} = \{ \mathcal{J}_1, \dots, \mathcal{J}_M\}$ be a finite partition of $\mathcal{S}$. Then, we say that set $A \subseteq \mathcal{S}$ is {\em inbound-accessible} (respectively, {\em outbound-accessible}) from ${\bf x} \in \mathcal{S}$ with respect to $(\tilde \kappa, \mathscr{J})$, denoted ${\bf x} \to A$ (respectively, ${\bf x} \leftarrow A$),  if there exists  $\{ u_1, \dots, u_k\} \subseteq \{1, \dots, M\}$ such that: 
\begin{itemize}
\item[i)] $\tilde \kappa({\bf x}, {\bf y}) > 0$ for all ${\bf y} \in \mathcal{J}_{u_1}$, 
\item[ii)] $\tilde \kappa > 0$ on $\mathcal{J}_{u_i} \times \mathcal{J}_{u_{i+1}}$ (respectively, $\tilde \kappa > 0$ on $\mathcal{J}_{u_{i+1}} \times \mathcal{J}_{u_{i}}$) for all $1 \leq i < k$, 
\item[iii)] $\mu(\mathcal{J}_{u_i}) > 0$ for all $1 \leq i \leq k$,  and 
\item[iv)] $\mathcal{J}_{u_k} \subseteq A$. 
\end{itemize}
\end{defn}

\begin{remark} \label{R.RemainAccessible}
Note that if we take $\mathscr{J}_m = \{ \mathcal{J}_1^{(m)}, \dots, \mathcal{J}_{M_m}^{(m)}\}$ as constructed in Lemma~\ref{L.Partition}, and we let $\tilde \kappa_m$ satisfy $\tilde \kappa_m \leq \tilde \kappa_{m+1}$ a.e., then if ${\bf x} \to A$ (${\bf x} \leftarrow A$) with respect to $(\tilde \kappa_{m_0}, \mathscr{J}_{m_0})$ for some $m_0 \geq 1$, then ${\bf x} \to A$ (${\bf x} \leftarrow A$) with respect to $(\tilde \kappa_{m}, \mathscr{J}_{m})$ for any $m \geq m_0$, since each $\mathcal{J}_{u_i}^{(m)}$ in part (iii) of Definition~\ref{D.Accessible} must contain at least one subset $\mathcal{J}_{t}^{(m+1)} \subseteq \mathcal{J}_{u_i}^{(m)}$ with $\mu(\mathcal{J}_{t}^{(m+1)}) > 0$. 
\end{remark}

We now give a result that states that we can always find a sequence of quasi-irreducible regular finitary kernels which converges monotonically to $\kappa$ and can be used to approximate from below $\kappa(1+\varphi_n)$. Its proof follows that of Lemma~7.3 in \cite{bollobas2007phase}, with some variations due to the directed nature of our kernels.

\begin{lemma}\label{approx_reducible}
For any continuous kernel $\kappa$ and any $\varphi_n$ satisfying Assumption~\ref{reg}, there exists a sequence $\{\tilde \kappa_m\}_{m \geq 1}$ of regular finitary kernels on $\mathcal{S} \times \mathcal{S}$, measurable with respect to $\mathscr{F}$, with the following properties.
\begin{enumerate}
\item $\tilde \kappa_m(\bold{x},\bold{y}) \nearrow \kappa(\bold{x},\bold{y})$ $\mathbb{P}$-a.s. as $m\rightarrow\infty$ for a.e. $(\bold{x},\bold{y})\in \mathcal{S} \times \mathcal{S}$
\item $\tilde \kappa_m(\bold{x},\bold{y})\leq \inf_{n \geq m} \kappa(\bold{x},\bold{y})(1+ \varphi_n({\bf x}, {\bf y}))$ for every $(\bold{x},\bold{y})\in \mathcal{S} \times \mathcal{S}$.
\item If $\kappa$ is quasi-irreducible, then so is $\kappa_m$ for all large $m$.
\end{enumerate}
\end{lemma}

\begin{proof}
We may assume that $\kappa >0$ on a set of positive measure, as otherwise we may take $\kappa_m \equiv 0$ for every $m$ and there is nothing to prove. 
We will construct the sequence $\{\kappa_m: m \geq 1\}$ in two stages. First, we construct a sequence $\{\tilde \kappa_m: m \geq 1\}$ where each $\tilde \kappa_m$ is regular finitary and satisfies conditions (a) and (b); then we use this sequence to obtain $\{ \kappa_m: m \geq 1\}$ satisfying (c).  

To this end, construct the sequence of partitions $\{\mathscr{J}_m \}_{m \geq 1}$ according to Lemma~\ref{L.Partition} and define
$$\tilde{\kappa}_m(\bold{x},\bold{y}) := \inf\left\{  \kappa({\bf x}', {\bf y}') \wedge \inf_{n \geq m} \kappa ({\bf x}', {\bf y}') (1 + \varphi_n({\bf x}', {\bf y}')) :  \bold{x}'\in \mathcal{J}_{\vartheta(\bold{x})}^{(m)}, \, \bold{y}'\in \mathcal{J}_{\vartheta(\bold{y})}^{(m)}\right\}.$$ 
Note that the properties of $\{ \mathscr{J}_m: m \geq 1\}$, and the assumption on $\varphi_n$ imply that 
$$\tilde{\kappa}_m(\bold{x},\bold{y})\nearrow\kappa(\bold{x},\bold{y}) \quad \mathbb{P}\text{-a.s.} \quad \text{as }m\rightarrow\infty,\quad \text{for a.e. } (\bold{x},\bold{y})\in \mathcal{S} \times \mathcal{S}.$$
Moreover, for $n\geq m$ we have that
$$\tilde \kappa_m({\bf x}, {\bf y}) \leq \kappa({\bf x}, {\bf y})(1 + \varphi_n({\bf x}, {\bf y})) \qquad \text{for all } ({\bf x}, {\bf y}) \in \mathcal{S} \times \mathcal{S}. $$
Hence, $\kappa_m = \tilde \kappa_m$ satisfies conditions (a) and (b) in the statement of the lemma. 

To prove (c) assume from now on that $\kappa$ is quasi-irreducible.  In fact, without loss of generality we may assume that $\kappa$ is irreducible, since it suffices to construct $\kappa_m$ to be quasi-irreducible on the restriction $\mathcal{S}'\times\mathcal{S}'$ where $\kappa$ is irreducible and then set it to be zero outside of $\mathcal{S}' \times \mathcal{S}'$. 

The first step of the proof ensures the existence of a directed cycle $\mathcal{C} \subseteq \mathcal{S}$ for some $m_1 \geq 1$. The second step uses $\mathcal{C}$ to construct a set on which $\tilde \kappa_m$ is irreducible. To establish the existence of $\mathcal{C}$, note that if $\tilde \kappa_m = 0$ a.e.~for all $m \geq 1$, it would imply that $\kappa = 0$~a.e., which would contradict the irreducibility of $\kappa$. Therefore, there must exist some $m_0 \geq 1$ and indexes $1 \leq r, s, t \leq M_{m_0}$ \ such that $\tilde \kappa_{m_0} > 0$ on $(\mathcal{J}_t^{(m_0)} \times \mathcal{J}_r^{(m_0)})$ and on $(\mathcal{J}_r^{(m_0)} \times \mathcal{J}_s^{(m_0)})$, with $\mu(\mathcal{J}_t^{(m_0)}) \mu(\mathcal{J}_r^{(m_0)}) \mu(\mathcal{J}_s^{(m_0)}) > 0$. 

{\bf Claim:} for any set $A \subseteq \mathcal{S}$  for which there exists a set $D\subseteq \mathcal{S}$ such that $\mu(D)  > 0$ and $\tilde \kappa_m > 0$ on $D \times A$ (respectively, $A \times D$), the sequence of sets $\{B_m(A) \}_{m \geq 1}$ (respectively, $\{ \tilde B_m(A) \}_{m \geq 1}$) defined according to $B_m(A) = \{ {\bf x} \in \mathcal{S}: {\bf x} \to A \text{ w.r.t. } (\tilde \kappa_m, \mathscr{J}_m)  \}$ (respectively, $\tilde B_m(A) = \{{\bf x} \in \mathcal{S}:  {\bf x} \leftarrow A \text{ w.r.t. } (\tilde \kappa_m, \mathscr{J}_m) \}$) satisfy: 1) $B_m(A) \subseteq B_{m+1}(A)$ (respectively, $\tilde B_m(A) \subseteq \tilde B_{m+1}(A)$), and  2) $\mu\left( \bigcup_{m=1}^\infty B_m(A) \right) = 1$ $\left( \text{respectively, } \mu\left( \bigcup_{m=1}^\infty \tilde B_m(A) \right) = 1 \right)$.

To prove the claim note that Remark~\ref{R.RemainAccessible} implies (1). To see that (2) holds, let $B(A) = \bigcup_{m=1}^\infty B_m(A)$ and note that from the definition of $B(A)$  we have $\kappa = 0$ a.e.~on $B(A)^c \times B(A)$, and the irreducibility of $\kappa$ implies that either $\mu(B(A)^c) = 0$ or $\mu(B(A)) = 0$; since $\mu(B(A)) \geq \mu(D) > 0$, it must be that $\mu(B(A)^c) = 0$, which implies that $\mu(B(A)) = 1$. The symmetric arguments yield the claim for $\{ \tilde B_m(A)\}$.

Now apply the inbound part of the claim to $A  = \mathcal{J}_r^{(m_0)}$ and $D = \mathcal{J}_t^{(m_0)}$ to obtain that there exists $m_1 \geq m_0$ such that $\mu(B_{m_1}(\mathcal{J}_r^{(m_0)}) \cap \mathcal{J}_s^{(m_0)}) > 0$,  which in turn implies there exists a set $\mathcal{J}_{s'}^{(m_1)} \subseteq \mathcal{J}_s^{(m_0)}$ such that $\mu(\mathcal{J}_{s'}^{(m_1)}) > 0$ and ${\bf x} \to \mathcal{J}_r^{(m_0)}$ for all ${\bf x} \in \mathcal{J}_{s'}^{(m_1)}$.  
In other words, there exist sets $\{ \mathcal{J}_{u_0}^{(m_1)}, \dots, \mathcal{J}_{u_k}^{(m_1)} \}$ satisfying $\mu(\mathcal{J}_{u_i}^{(m_1)}) > 0$ for all $0 \leq i \leq k$, $\mathcal{J}_{u_0}^{(m_1)} =  \mathcal{J}_{s'}^{(m_1)}$, $\mathcal{J}_{u_k}^{(m_1)} \subseteq \mathcal{J}_r^{(m_0)}$, and $\tilde \kappa_{m_1} > 0$ on $\mathcal{J}_{u_i}^{(m_1)} \times \mathcal{J}_{u_{i+1}}^{(m_1)}$ for all $0 \leq i < k$.  Since $0 < \tilde \kappa_{m_0} \leq \tilde \kappa_{m_1}$ on $\mathcal{J}_{u_k}^{(m_1)} \times \mathcal{J}_{u_0}^{(m_1)}$ by construction, we have that the set $\mathcal{C} = \bigcup_{i=0}^k \mathcal{J}_{u_i}^{(m_1)}$ defines a directed cycle. 

Next, construct the sequences $\{ B_m(\mathcal{C}) \}_{m \geq 1}$ and $\{ \tilde B_m(\mathcal{C}) \}_{m \geq 1}$ according to the claim, and define 
$$\kappa_m({\bf x}, {\bf y}) = \tilde \kappa_m({\bf x}, {\bf y}) 1({\bf x} \in (B_m(\mathcal{C}) \cap \tilde B_m(\mathcal{C})), \, {\bf y} \in (B_m(\mathcal{C}) \cap \tilde B_m(\mathcal{C}))).$$
Note that $\kappa_m \nearrow \kappa$ $\mathbb{P}$-a.s.~as $m \to \infty$ since $\tilde \kappa_m \nearrow \kappa$ $\mathbb{P}$-a.s.~and 
$$\mu\left( \bigcup_{m = 1}^\infty (B_m(\mathcal{C}) \cap \tilde B_m(\mathcal{C})) \right) \geq  1 -  \mu\left( \bigcap_{m = 1}^\infty B_m(\mathcal{C})^c \right) - \mu\left( \bigcap_{m=1}^\infty \tilde B_m(\mathcal{C})^c \right) = 1.$$

It remains to show that $\kappa_m$ restricted to $(B_m(\mathcal{C}) \cap \tilde B_m(\mathcal{C})) \times (B_m(\mathcal{C}) \cap \tilde B_m(\mathcal{C}))$ is irreducible. To see this, let $A \subseteq (B_m(\mathcal{C}) \cap \tilde B_m(\mathcal{C}))$ and suppose $\kappa_m = 0$ on $A \times (A^c \cap B_m(\mathcal{C}) \cap \tilde B_m(\mathcal{C}))$. Note that since $\tilde \kappa_{m_1} > 0$ on each $\mathcal{J}_{u_i}^{(m_1)} \times \mathcal{J}_{u_{i+1}}^{(m_1)}$, then it must be that either $\mathcal{C} \subseteq A$ or $\mathcal{C} \subseteq A^c$. Suppose that it is the former, and note that for any ${\bf x} \in A^c \cap B_m(\mathcal{C}) \cap \tilde B_m(\mathcal{C})$ there exist indexes $\{v_1, \dots, v_l\}$ and $\{ w_1, \dots, w_j\}$ such that 
$$\tilde \kappa_{m_1} > 0 \text{ on } \mathcal{J}^{(m_1)}_{v_i} \times \mathcal{J}^{(m_1)}_{v_{i+1}}, \, 0 \leq i \leq l, \,  \mu(\mathcal{J}_{v_i}^{(m_1)}) > 0, \, 1 \leq i \leq l, \, \mathcal{J}^{(m_1)}_{v_l} \subseteq \mathcal{C},$$
and
$$ \tilde \kappa_{m_1} > 0 \text{ on } \mathcal{J}^{(m_1)}_{w_{i+1}} \times \mathcal{J}^{(m_1)}_{w_i}, \, 0 \leq i \leq j, \, \mu(\mathcal{J}_{w_i}^{(m_1)}) > 0, \, 1 \leq i \leq j, \, \mathcal{J}^{(m_1)}_{w_j} \subseteq \mathcal{C},$$
where $\mathcal{J}^{(m_1)}_{v_0} = \mathcal{J}^{(m_1)}_{w_0} = \mathcal{J}^{(m_1)}_{\vartheta({\bf x})}$. Moreover, $\mu(\mathcal{J}^{(m_1)}_{\vartheta({\bf x})}) > 0$ would imply that  $\mathcal{J}_{v_i}^{(m_1)} \subseteq B_m(\mathcal{C}) \cap \tilde B_m(\mathcal{C})$  for all $1 \leq i \leq l$ and $\mathcal{J}_{w_h}^{(m_1)} \subseteq B_m(\mathcal{C}) \cap \tilde B_m(\mathcal{C})$ for all $1 \leq h \leq j$, since they would all lie on a directed cycle of positive measure, but this contradicts our assumption that $\tilde \kappa_{m_1} = 0$ on $A \times A^c \cap B_m(\mathcal{C}) \cap \tilde B_m(\mathcal{C})$. Hence, it must be that $\mu(\mathcal{J}^{(m_1)}_{\vartheta({\bf x})})  = 0$ for all ${\bf x} \in A^c \cap B_m(\mathcal{C}) \cap \tilde B_m(\mathcal{C})$, and therefore, $\mu(A^c \cap B_m(\mathcal{C}) \cap \tilde B_m(\mathcal{C})) = 0$. The same argument gives that if $\mathcal{C} \subseteq A^c \cap B_m(\mathcal{C}) \cap \tilde B_m(\mathcal{C})$ then $\mu(A) = 0$. We conclude that $\kappa_m$ restricted to $(B_m(\mathcal{C}) \cap \tilde B_m(\mathcal{C})) \times (B_m(\mathcal{C}) \cap \tilde B_m(\mathcal{C}))$ is irreducible. This completes the proof. 
\end{proof}

The following lemma establishes the relationships between $\rho(\kappa_m)$, $\rho^{\geq k}(\kappa_m)$, $\rho^{\geq k}(\kappa)$, and $\rho(\kappa)$.

\begin{lemma}\label{L.MonotoneKernels}
Let $\{\kappa_m\}_{m \geq 1}$ be a sequence of kernels on $(\mathcal{S},\mu)$ increasing a.e. to $\kappa$. Then, the following limits hold:
\begin{enumerate}
\item  $\rho^{\geq k}(\kappa;\bold{x})\searrow\rho(\kappa;\bold{x})$ a.e. $\bold{x}$ and  $\rho^{\geq k}(\kappa)\searrow\rho(\kappa)$ as $k \to \infty$.

\item For every $k\geq 1$, $\rho^{\geq k}(\kappa_m;\bold{x})\nearrow\rho^{\geq k}(\kappa;\bold{x})$ for a.e. $\bold{x}$ and $\rho^{\geq k}(\kappa_{m})\nearrow \rho^{\geq k}(\kappa)$ as $m\rightarrow\infty$.

\item $\rho(\kappa_m;\bold{x})\nearrow\rho(\kappa;\bold{x})$ for a.e. $\bold{x}$ and $\rho(\kappa_{m})\nearrow \rho(\kappa)$ as $m\rightarrow\infty$.
\end{enumerate}
 \end{lemma}

\begin{proof}
By Lemma 9.5~in \cite{bollobas2007phase}, we have that $\rho^{\geq k}_{-}(\kappa;\bold{x})\searrow\rho_{-}(\kappa;\bold{x})$ and $\rho^{\geq k}_{+}(\kappa;\bold{x})\searrow\rho_{+}(\kappa;\bold{x})$  as $k\rightarrow\infty$ for a.e.~$\bold{x}$. Then, by the monotone convergence theorem, we have
\begin{align*}
\lim_{k\rightarrow\infty}\rho^{\geq k}(\kappa)&=\lim_{k\rightarrow\infty}\int_{\mathcal{S}}\rho_{-}^{\geq k}(\kappa;\bold{s})\rho_{+}^{\geq k}(\kappa;\bold{s}) \mu(d\bold{s})\\
&=\int_{\mathcal{S}}\lim_{k\rightarrow\infty}\rho_{-}^{\geq k}(\kappa;\bold{s})\rho_{+}^{\geq k}(\kappa;\bold{s})\mu(d\bold{s})\\
&=\int_{\mathcal{S}}\rho_{-}(\kappa;\bold{s})\rho_{+}(\kappa;\bold{s})\mu(d\bold{s})=\rho(\kappa),
\end{align*}
which establishes (a). 

By Theorem~6.5(i) in \cite{bollobas2007phase} we have that for any fixed $k \geq 1$, $\rho^{\geq k}_-(\kappa_m;\bold{x}) \nearrow \rho^{\geq k}_-(\kappa;\bold{x})$ and $\rho^{\geq k}_+(\kappa_m;\bold{x}) \nearrow \rho^{\geq k}_+(\kappa;\bold{x})$ as $m\rightarrow\infty$ for a.e.~$\bold{x}$, which together with monotone convergence as above implies (b). 

Part (c) follows from part (a) applied to the kernel $\kappa_m$, followed by part (b), to obtain that
$$\lim_{m \to \infty} \rho(\kappa_m; {\bf x}) = \lim_{m \to \infty} \lim_{k \to \infty} \rho^{\geq k}(\kappa_m; {\bf x}) = \lim_{k \to \infty} \lim_{m \to \infty} \rho^{\geq k}(\kappa_m; {\bf x}) = \lim_{k \to \infty} \rho^{\geq k}(\kappa;{\bf x}) = \rho(\kappa;{\bf x})$$
for a.e.~${\bf x}$. Then use monotone convergence as above. 
\end{proof}

\bigskip

Recall the definition of the operators $T_\kappa^+$ and $T_\kappa^-$ given in Section~\ref{S.PhaseTransition}, as well as of their spectral radii $r(T_\kappa^+)$ and $r(T_\kappa^-)$. The strict positivity of $\rho(\kappa)$, which ensures the existence of a giant strongly connected component, is characterized below.  As a preliminary result, we establish the phase transition for regular finitary, quasi-irreducible kernels first.

\begin{prop} \label{P.FinitaryPhaseTransition}
Suppose that $\tilde \kappa$ is a regular finitary, quasi-irreducible, kernel on the type-space $\mathcal{S}$ with respect to measure $\mu$. Then, $r(T_{\tilde \kappa}^+) = r(T_{\tilde \kappa}^-)$ and we have that $\rho(\tilde \kappa) > 0$ if and only if $r(T_{\tilde \kappa}^+) > 1$.  Moreover, there exist nonnegative, non-zero eigenfunctions $f_+$ and $f_-$, such that $T_{\tilde\kappa}^+ f^+ = r(T_{\tilde \kappa}^+) f_+$ and $T_{\tilde\kappa}^- f^- = r(T_{\tilde\kappa}^-) f_-$, and they are the only (up to multiplicative constants and sets of measure zero) nonnegative, non-zero eigenfunctions of $T_{\tilde \kappa}^+$ and $T_{\tilde \kappa}^-$, respectively. 
\end{prop}

\begin{proof} 
Since $\tilde \kappa$ is quasi-irreducible, there exists $\mathcal{S}^* \subseteq \mathcal{S}$ such that $\tilde \kappa$ restricted to $\mathcal{S}^*$ is irreducible and $\mu(\mathcal{S}^*) > 0$. Also, since $\tilde \kappa$ is regular finitary, there exists a finite partition $\{ \mathcal{J}_i: 1 \leq i \leq M\}$ such that $\tilde \kappa$ is constant on $\mathcal{J}_i \times \mathcal{J}_j$. Next, define
$$\mathcal{S}' = \bigcup_{i=1}^M \{ \mathcal{J}_i \cap \mathcal{S}^*: \mu(\mathcal{J}_i \cap \mathcal{S}^*) > 0 \},$$
and define the kernel $\kappa'({\bf x}, {\bf y}) = \mu(\mathcal{S}') \tilde \kappa({\bf x}, {\bf y})$ for ${\bf x}, {\bf y} \in \mathcal{S}'$. Note that $\kappa'$ is regular finitary and irreducible on $\mathcal{S}'$ and $\mu(\mathcal{S}') = \mu(\mathcal{S}^*)$. Moreover, if we let $\mu'(A) = \mu(A)/\mu(\mathcal{S}')$ for $A \subseteq \mathcal{S}'$, and let $\{ \mathcal{J}_i' : 1 \leq i \leq M' \}$ denote the partition of $\mathcal{S}'$ such that $\kappa'$ is constant on $\mathcal{J}_i' \times \mathcal{J}_j'$, then $\mu'( \mathcal{J}_i') > 0$ for all $1 \leq i \leq M'$. 

Next, consider the double tree $(\mathcal{T}_{\mu'}^+(\kappa'), \mathcal{T}_{\mu'}^-(\kappa'))$ on the type-space $\mathcal{S}'$ with respect to measure $\mu'$. Note that each of these trees can be thought of as a multi-type branching process with $M'$ types (one associated to each of the $\mathcal{J}_i'$) each having positive probability. We will show that:
\begin{enumerate}
\item the survival probability $\rho(\tilde \kappa) = \mu(\mathcal{S}') \rho'(\kappa')$, where
$$\rho'(\kappa') = \int_{\mathcal{S}'} \rho_+'(\kappa'; {\bf x}) \rho_-'(\kappa'; {\bf x}) \mu'(d{\bf x}),$$
and $\rho_+'(\kappa'; {\bf x})$, $\rho_-'(\kappa'; {\bf x})$ are the survival probabilities of the trees $\mathcal{T}_{\mu'}^+(\kappa'; {\bf x})$ and $\mathcal{T}_{\mu'}^-(\kappa'; {\bf x})$, respectively; and

\item the spectral radii of the operators $T_{\tilde \kappa}^\pm$ on $\mathcal{S}$ and $T_{\kappa'}^\pm$ on $\mathcal{S}'$ are the same. 
\end{enumerate}

To prove (a), note that since types ${\bf x} \in (\mathcal{S}^*)^c$ are isolated (since $\tilde \kappa({\bf x}, {\bf y}) = 0$ for ${\bf x} \in (\mathcal{S}')^c$ or ${\bf y} \in (\mathcal{S}')^c$) and $\mathcal{S}^* \cap (\mathcal{S}')^c$ has measure zero, then they do not contribute to the survival probabilities of $\mathcal{T}_\mu^+(\tilde \kappa)$ and $\mathcal{T}_\mu^-(\tilde \kappa)$, which implies that
\begin{align*}
\rho(\tilde \kappa) &= \int_{\mathcal{S}} \rho_-(\tilde \kappa; {\bf x}) \rho_+(\tilde \kappa; {\bf x}) \mu(d{\bf x}) = \mu(\mathcal{S}') \int_{\mathcal{S}'} \rho_-(\tilde \kappa; {\bf x}) \rho_+(\tilde \kappa; {\bf x}) \mu'(d{\bf x}) .
\end{align*}
Now note that the trees $\mathcal{T}^\pm_\mu(\tilde \kappa)$ and $\mathcal{T}^\pm_{\mu'} (\kappa')$ have the same law when their roots belong to $\mathcal{S}'$ since the number of offspring of type ${\bf y} \in \mathcal{S}'$ that an individual of type ${\bf x} \in \mathcal{S}'$ on the tree  $\mathcal{T}^+_{\mu'}(\kappa')$ has, is Poisson distributed with mean
$$\int_{\mathcal{S}'} \kappa'({\bf y}, {\bf x}) \mu'(d{\bf x}) = \int_{\mathcal{S}'} \mu(\mathcal{S}') \tilde \kappa({\bf y}, {\bf x}) \mu(d{\bf x}) / \mu(\mathcal{S}') = \int_{\mathcal{S}} \tilde \kappa({\bf y}, {\bf x}) \mu(d{\bf x}),$$
which is equal to the corresponding distribution in $\mathcal{T}_\mu(\tilde \kappa)$. The same argument yields the result for $\mathcal{T}^-_\mu(\tilde \kappa)$ and $\mathcal{T}^-_{\mu'} (\kappa')$. Hence, we have that $\rho_\pm(\tilde \kappa; {\bf x}) = \rho_\pm(\kappa'; {\bf x})$ for ${\bf x} \in \mathcal{S}'$, and therefore, 
$$\rho(\tilde \kappa) = \mu(\mathcal{S}') \rho'(\kappa').$$ 

To establish (b), note that if $f_\pm'$ is the nonnegative eigenfunction associated to $r(T_{\kappa'}^\pm)$ on $\mathcal{S}'$, then $f_\pm({\bf x}) = f_\pm'({\bf x}) 1({\bf x} \in \mathcal{S}')$ satisfies
$$(T_{\tilde \kappa}^+ f_+)({\bf x}) = \int_{\mathcal{S}} \tilde \kappa({\bf y}, {\bf x}) f_+({\bf y}) \mu(d{\bf y}) = \int_{\mathcal{S}'} \kappa'({\bf y}, {\bf x}) f_+'({\bf y}) \mu'(d{\bf y})  = r(T_{\kappa'}^+) f'_+({\bf x}) =  r(T_{\kappa'}^+) f_+({\bf x})$$
for ${\bf x} \in \mathcal{S}'$, while for ${\bf x} \in (\mathcal{S}')^c$ we have $(T_{\tilde \kappa}^+ f_+)({\bf x})  = 0$ since $\tilde \kappa({\bf y}, {\bf x}) = 0$ for all ${\bf y} \in \mathcal{S}$. Therefore, $r(T_{\kappa'}^+)$ is an eigenvalue of $T_{\tilde \kappa}^+$, which implies that $r(T_{\kappa'}^+) \leq r(T_{\tilde \kappa}^+)$; similarly, $r(T_{\kappa'}^-)$ is an eigenvalue of $T_{\tilde \kappa}^-$ and $r(T_{\kappa'}^-) \leq r(T_{\tilde \kappa}^-)$.  For the opposite inequality, suppose $f_\pm$ is a nonnegative eigenvector associated to $r(T_{\tilde \kappa}^\pm)$ and set $f_\pm'$ to be its restriction to $\mathcal{S}'$. Then note that for ${\bf x} \in \mathcal{S}'$,
$$(T_{\kappa'}^+ f_+')({\bf x}) = \int_{\mathcal{S}'} \kappa'({\bf y}, {\bf x}) f_+'({\bf y}) \mu'(d{\bf y}) = \int_{\mathcal{S}} \tilde \kappa({\bf y}, {\bf x}) f_+({\bf y}) \mu(d{\bf y}) = r(T_{\tilde \kappa}^+) f_+({\bf x}) =  r(T_{\tilde \kappa}^+) f_+'({\bf x}),$$
and therefore, $r(T_{\tilde \kappa}^+)$ is an eigenvalue of $T_{\kappa'}^+$ and therefore $r(T_{\tilde \kappa}^+) \leq r(T_{ \kappa'}^+)$. Similarly, $r(T_{\tilde \kappa}^-) \leq r(T_{ \kappa'}^-)$. We conclude that
$$r(T_{\tilde \kappa}^\pm) = r(T_{ \kappa'}^\pm).$$

To see that $r(T_{\kappa'}^+) = r(T_{\kappa'}^-)$ we first point out that $\mathcal{T}_{\mu'}^+(\kappa')$ and $\mathcal{T}_{\mu'}^-(\kappa')$ can be thought of as irreducible multi-type Galton-Watson processes with a finite number of types and mean progeny matrices ${\bf M}^+ = (m_{ij}^+)$ and ${\bf M}^- = (m_{ij}^-)$, respectively, where $m_{ij}^+ = c_{ji} \mu'(\mathcal{J}_j')$, $m_{ij}^- = c_{ij} \mu'(\mathcal{J}_j')$, and $\kappa'({\bf x}, {\bf y}) = \sum_{i=1}^{M'} \sum_{j=1}^{M'} c_{ij} 1({\bf x} \in \mathcal{J}_i', {\bf y} \in \mathcal{J}_j')$.  Moreover, the operators $T_{\kappa'}^+$ and $T_{\kappa'}^-$ satisfy
$$T_{\kappa'}^\pm f = {\bf M}^\pm {\bf v} \qquad \text{for ${\bf v} = (v_1, \dots, v_{M'})^T \in \mathbb{R}^{M'}$ and $f({\bf x}) = v_i 1({\bf x} \in \mathcal{J}_i')$, ${\bf x} \in \mathcal{S}'$}.$$
That ${\bf M}^+$ and ${\bf M}^-$ have the same spectral radius follows from noting that ${\bf M}^+ = {\bf C} {\bf D}$ and ${\bf M}^- = {\bf C}^T {\bf D} = ({\bf D} {\bf C})^T$ for ${\bf D} = \text{diag}(\mu'(\mathcal{J}_1'), \dots, \mu'(\mathcal{J}_{M'}'))$ and ${\bf C} = (c_{ij})$, which implies that the eigenvalues of ${\bf M}^-$ are the complex conjugates of those of ${\bf D} {\bf C}$, which in turn are the same as those of ${\bf C} {\bf D}$.

The if and only if statement for the survival probabilities now follows from Theorem~8 in \cite{athreya1972} (see also Theorems 2.1 and 2.2 in Chapter 2 of \cite{mode1971multitype}), which states that 
$$\rho_\pm'(\kappa'; {\bf x}) > 0 \text{ for all } {\bf x} \in \mathcal{S}' \qquad \text{if and only if} \qquad  r({\bf M}^\pm)  > 1,$$
where $r({\bf M}^\pm) = r(T_{\kappa'}^\pm)$ is the spectral radius of ${\bf M}^\pm$.

The existence of the eigenfunctions $f_+$ and $f_-$ on $\mathcal{S}$ follows from the Perron-Frobenius theorem (see Theorem~1.5 in \cite{seneta2006non}), which guarantees the existence of strictly positive eigenfunctions $f_+'$ and $f_-'$ on $\mathcal{S}'$ such that $T_{\kappa'}^\pm f_\pm' = r(T_{\kappa'}^\pm) f_\pm'$, by setting $f_\pm({\bf x}) = f_\pm'({\bf x}) 1({\bf x} \in \mathcal{S}')$. Moreover, $f_+'$ and $f_-'$ are the only (up to multiplicative constants) nonnegative, non-zero eigenfunctions of the operators $T_{\kappa'}^+$ and $T_{\kappa'}^-$, respectively. To see that the nonnegative eigenfunctions $f_+$ and $f_-$ are also unique (up to multiplicative constants and sets of measure zero) note that any other nonnegative eigenfunction $g_+$ of $T_{\tilde\kappa}^+$ associated to a positive eigenvalue $\lambda$  would have to satisfy 
$$(T_{\tilde \kappa}^+ g_+)({\bf x}) = \int_{\mathcal{S}} \tilde \kappa({\bf y}, {\bf x}) g_+({\bf y}) \mu(d{\bf y}) = 0 \qquad \text{for } {\bf x} \in (\mathcal{S}^*)^c,$$
since $\tilde \kappa({\bf x}, {\bf y}) = 0$ for ${\bf x} \in (\mathcal{S}^*)^c$, and 
$$(T_{\tilde \kappa}^+ g_+)({\bf x}) = \int_{\mathcal{S}'}  \kappa'({\bf y}, {\bf x})  g_+({\bf y}) \mu'(d{\bf y}) = \lambda g_+({\bf x}) \qquad \text{for } {\bf x} \in \mathcal{S}'$$
which would imply $\lambda$ is a positive eigenvalue of $T_{\kappa'}^+$ with a nonnegative, non-zero, eigenfunction. The uniqueness of $f_+'$ then gives that $g_+({\bf x}) = \alpha f_+'({\bf x})$ for ${\bf x} \in \mathcal{S}'$ for some constant $\alpha > 0$. Finally, since $\mu(\mathcal{S}^* \cap (\mathcal{S}')^c) = 0$, we conclude that $g_+({\bf x}) = \alpha f_+({\bf x})$ a.e. The same arguments give that any other nonnegative eigenfunction $g_-$ of $T_{\tilde \kappa}^-$ would have to satisfy $g_-({\bf x}) = \beta f_-({\bf x})$ a.e. This completes the proof. 
\end{proof}

We now use the regular finitary and quasi-irreducible case to establish the result for general irreducible kernels. As pointed out in Remark~\ref{R.NotIFF}, the result does not provide a full if and only if condition for the strict positivity of $\rho(\kappa)$, since when the operators $T_\kappa^+$ and $T_\kappa^-$ are unbounded we cannot guarantee the continuity of the spectral radii of the sequence of operators $T_{\kappa_m}^+$ and $T_{\kappa_m}^-$.

\begin{lemma} \label{L.PhaseTransition}
Suppose that $\kappa$ is irreducible on the type-space $\mathcal{S}$ with respect to measure $\mu$. Then, if $\rho(\kappa) > 0$ then $r(T_\kappa^+) > 1$ and $r(T_\kappa^-) > 1$. Moreover, if there exists a regular finitary quasi-irreducible kernel $\tilde \kappa$ such that $\tilde \kappa \leq \kappa$ a.e.~and $r(T_{\tilde \kappa}^+)  > 1$ (equivalently, $r(T_{\tilde \kappa}^-) > 1$), then $\rho(\kappa) > 0$. 
\end{lemma}

\begin{proof}
Suppose first that $\rho(\kappa) > 0$. Now use Lemma~\ref{approx_reducible} and Lemma~\ref{L.MonotoneKernels} to obtain that $\rho(\kappa_m) > 0$ for some quasi-irreducible, regular finitary, kernel $\kappa_m$ such that $\kappa_m({\bf x}, {\bf y}) \leq \kappa({\bf x}, {\bf y})$ for all ${\bf x}, {\bf y} \in \mathcal{S}$. By Proposition~\ref{P.FinitaryPhaseTransition} we have that the spectral radii of the operators $T_{\kappa_m}^+$ and $T_{\kappa_m}^-$ satisfy $r(T_{\kappa_m}^+) = r(T_{\kappa_m}^-)> 1$. By monotonicity of the spectral radius, we conclude that $r(T_{\kappa}^+) \geq r(T_{\kappa_m}^+) > 1$ and $r(T_{\kappa}^-) \geq r(T_{\kappa_m}^-) > 1$. 

For the converse, note that if $\tilde \kappa \leq \kappa$ a.e.~and $r(T_{\tilde \kappa}^+)  > 1$, then by Proposition~\ref{P.FinitaryPhaseTransition} we have that $\rho(\tilde \kappa) > 0$. Since $\rho(\tilde \kappa) \leq \rho(\kappa)$, the result follows. 
\end{proof}

The last preliminary result before proving Theorem~\ref{giant} provides the key estimates obtained through Theorem~\ref{T.Coupling}, since it relates the indicator random variables for each vertex $i$ to have in-component and out-component of size at least $k$ with the corresponding probabilities in the double-tree $(\mathcal{T}_\mu^+(\kappa_m; {\bf X}_i), \mathcal{T}_\mu^-(\kappa_m; {\bf X}_i))$.

\begin{prop} \label{P.BernoullisCoupled}
For any $k \geq 1 $ and $i \in V_n$, define $\chi_{n,i}^{\geq k}$ to be the indicator function of the event that vertex $i$ has in-component and out-component both of size at least $k$. Then, for any $0 < \epsilon < 1/2$, we have 
\begin{align*}
\left| \frac{1}{n} \sum_{i=1}^n \mathbb{E}\left[ \chi_{n,i}^{\geq k} \right]  - \frac{1}{n} \sum_{i=1}^n \rho_+^{\geq k}(\kappa_m; {\bf X}_i) \rho_-^{\geq k}(\kappa_m; {\bf X}_i)  \right| &\leq  H(n,m,k,\epsilon),  \\
\frac{1}{n^2} \sum_{i=1}^{n} \sum_{j \neq i} \mathbb{E}\left[ \left( \chi_{n,i}^{\geq k} -  \mathbb{E}\left[ \chi_{n,i}^{\geq k} \right]  \right) \left( \chi_{n,j}^{\geq k} -  \mathbb{E}\left[ \chi_{n,j}^{\geq k} \right]  \right) \right] &\leq K(n,m,k) + 3 H(n,m,k,\epsilon),
\end{align*}
where
\begin{align*}
K(n,m,k) &:= \frac{4(k+1) \log n}{n} \sup_{{\bf x}, {\bf y} \in \mathcal{S}} \kappa_m({\bf x}, {\bf y})  + \frac{k}{\log n} \left( 2 + \sup_{{\bf x} \in \mathcal{S}} \lambda_+^{(m)}({\bf x}) + \sup_{{\bf x} \in \mathcal{S}} \lambda_-^{(m)}({\bf x}) \right) ,
\end{align*}
and $H(n,m,k,\epsilon)$ is defined in Theorem~\ref{T.Coupling}. 
\end{prop}

\begin{proof}
To derive the first bound construct a coupling between the graph exploration processes of the in-component and out-component of vertex $i$ and the double tree $(\mathcal{T}_\mu^+(\kappa_m; {\bf X}_i), \mathcal{T}_\mu^-(\kappa_m; {\bf X}_i))$, as described in Section~\ref{SS.Coupling}. Define $\tau^+$ and $\tau^-$ to be the steps in the construction when the coupling breaks on the inbound, respectively outbound, sides, and let $\sigma_{k}^+ = \inf\{ t \geq 1: |A_t^+| + |I_t^+| \geq k  \text{ or } A_t^+ = \varnothing \}$ and $\sigma_k^- = \inf\{ t \geq 1: |A_t^-| + |I_t^-| \geq k \text{ or }   A_t^- = \varnothing \}$. Note that at time $\sigma_k^+ \vee \sigma_k^-$ it is possible to determine whether both the in-component and out-component of vertex $i$ have at least $k$ vertices or not. To simplify the notation, let $ \rho^{\geq k}(\kappa_m; {\bf x}) =  \rho_+^{\geq k}(\kappa_m; {\bf x}) \rho_-^{\geq k}(\kappa_m; {\bf x})$.
\begin{align*}
\frac{1}{n} \sum_{i=1}^n \mathbb{E}\left[ \chi_{n,i}^{\geq k} \right]  &= \frac{1}{n} \sum_{i=1}^n \mathbb{P}\left(\chi_{n,i}^{\geq k} = 1 \right) \\
&\leq \frac{1}{n} \sum_{i=1}^n \mathbb{P}_i \left( \chi_{n,i}^{\geq k} = 1, \, \tau^+ \geq \sigma_k^+, \tau^- \geq \sigma_k^- \right) + \frac{1}{n} \sum_{i=1}^n \mathbb{P}_i \left(\{ \tau^+ < \sigma_k^+\} \cup \{ \tau^- < \sigma_k^- \} \right) \\
 &\leq \frac{1}{n} \sum_{i=1}^n \mathbb{P} \left( \text{both $\mathcal{T}_\mu^+(\kappa_m; {\bf X}_i)$ and $\mathcal{T}_\mu^-(\kappa_m; {\bf X}_i)$ have at least $k$ nodes}  \right) + H(n,m,k,\epsilon) \\
 &= \frac{1}{n} \sum_{i=1}^n \rho^{\geq k}(\kappa_m; {\bf X}_i) +  H(n,m,k,\epsilon),
\end{align*}
where we used Theorem~\ref{T.Coupling} to obtain that $n^{-1}  \sum_{i=1}^n \mathbb{P}_i \left(\{ \tau^+ < \sigma_k^+\} \cup \{ \tau^- < \sigma_k^- \} \right) \leq H(n,m,k,\epsilon)$. 

For the second inequality, first note that  
\begin{align*}
&\frac{1}{n^2} \sum_{i=1}^n \sum_{j \neq i} \mathbb{E}\left[ \left( \chi_{n,i}^{\geq k} -  \mathbb{E}\left[ \chi_{n,i}^{\geq k} \right]  \right) \left( \chi_{n,j}^{\geq k} -  \mathbb{E}\left[ \chi_{n,j}^{\geq k} \right]  \right) \right] \\
&= \frac{1}{n^2} \sum_{i=1}^n \sum_{j \neq i} \mathbb{E}\left[ \chi_{n,i}^{\geq k}  \chi_{n,j}^{\geq k} \right] - \frac{1}{n^2} \sum_{i=1}^n \sum_{j\neq i}  \mathbb{E}\left[ \chi_{n,i}^{\geq k} \right]  \mathbb{E}\left[ \chi_{n,j}^{\geq k} \right]  .
\end{align*}
To estimate $\mathbb{E}\left[ \chi_{n,i}^{\geq k}  \chi_{n,j}^{\geq k} \right] $ we will assume that we first explore the inbound and outbound neighborhood of vertex $i$ up to the time both its in-component and out-component have at least $k$ vertices or there are no more vertices to explore, i.e., we will explore the in-component of vertex $i$ up to time $\sigma_{k,i}^+$ and its out-component up to time $\sigma_{k,i}^-$. Note that we have added the subscript $i$, relative to the notation introduced in Section~\ref{SS.Coupling}, to emphasize that the exploration starts at vertex $i$. Next, define $\mathcal{F}_{k,i}$ to be the sigma-algebra generated by the exploration of the in-component and out-component of vertex $i$, as described in Section~\ref{SS.Coupling}, up to Step $\sigma_{k,i}^+$ on the inbound side and up to Step $\sigma_{k,i}^-$ on the outbound side. Define $\mathcal{N}_i^{(k)} =  I_{\sigma_{k,i}^-}^-  \cup  I_{\sigma_{k,i}^+}^+  \cup A_{\sigma_{k,i}^+}^+ \cup A_{\sigma_{k,i}^-}^-$  to be the set of vertices discovered during that exploration. Now explore the in-component and out-component of vertex $j$, as described in Section~\ref{SS.Coupling}, up to Step $\sigma_{k,j}^+$ on the inbound side and up to Step $\sigma_{k,j}^-$ on the outbound side; let $\mathcal{N}_j^{(k)} $ be the corresponding set of vertices discovered during the exploration of vertex $j$. 

Define $C_{ij} = \left\{ \mathcal{N}_i^{(k)} \cap \mathcal{N}_j^{(k)} = \varnothing \right\}$ and note that,
\begin{align*}
\mathbb{E}\left[ \chi_{n,i}^{\geq k}  \chi_{n,j}^{\geq k} \right] &\leq \mathbb{E}\left[ \chi_{n,i}^{\geq k}  \chi_{n,j}^{\geq k}  1(C_{ij})\right] + \mathbb{E}\left[ 1(C_{ij}^c) \right]  = \mathbb{E}\left[ \chi_{n,i}^{\geq k} \mathbb{E}\left[ \left. \chi_{n,j}^{\geq k}   1(C_{ij}) \right| \mathcal{F}_{k,i} \right]   \right] + \mathbb{P}(C_{ij}^c) .
\end{align*}
To analyze the conditional expectation, observe that
$$\mathbb{E}\left[ \left. \chi_{n,j}^{\geq k}   1(C_{ij}) \right| \mathcal{F}_{k,i} \right] = \mathbb{E}\left[ \left. \chi_{n,j}^{\geq k}    \right| \mathcal{F}_{k,i}, C_{ij}  \right] \mathbb{P}(C_{ij} | \mathcal{F}_{k,i} ),$$
where, due to the independence among the arcs, we have that conditionally on $\mathcal{F}_{k,i}$ and $C_{ij}$, the random variable $\chi_{n,j}^{\geq k}$ has the same distribution as the indicator function of the event that vertex $j$ has in-component and out-component both of size at least $k$ on the graph $G_{n}(\kappa_{n,i})$, with
$$\kappa_{n,i}({\bf X}_s, {\bf X}_t) = \kappa({\bf X}_s, {\bf X}_t) (1 + \varphi_n({\bf X}_s, {\bf X}_t)) 1(s \notin \mathcal{N}_i^{(k)}, t \notin \mathcal{N}_i^{(k)}).$$
Now note that since $\kappa_{n,i} \leq \kappa (1+\varphi_n)$ for any realization of $\mathcal{N}_i^{(k)} \subseteq V_n$, we have
$$\mathbb{E}\left[ \left. \chi_{n,j}^{\geq k}    \right| \mathcal{F}_{k,i}, C_{ij}  \right] \leq  \mathbb{E}\left[ \chi_{n,j}^{\geq k} \right],$$
from where it follows that
\begin{align*}
\frac{1}{n^2} \sum_{i=1}^n \sum_{j \neq i} \mathbb{E}\left[ \chi_{n,i}^{\geq k}  \chi_{n,j}^{\geq k} \right] &\leq \frac{1}{n^2} \sum_{i=1}^n \sum_{j \neq i} \left( \mathbb{E}\left[ \chi_{n,i}^{\geq k}  \right] \mathbb{E}\left[ \chi_{n,j}^{\geq k}  \right] + \mathbb{P}(C_{ij}^c) \right),
\end{align*}
which in turn implies that
\begin{align*}
\frac{1}{n^2} \sum_{i=1}^n \sum_{j \neq i} \mathbb{E}\left[ \left( \chi_{n,i}^{\geq k} -  \mathbb{E}\left[ \chi_{n,i}^{\geq k} \right]  \right) \left( \chi_{n,j}^{\geq k} -  \mathbb{E}\left[ \chi_{n,j}^{\geq k} \right]  \right) \right] \leq \frac{1}{n^2} \sum_{i=1}^n \sum_{j \neq i} \mathbb{P}(C_{ij}^c). 
\end{align*}

Similarly to what was done on the graph, define $ \hat{\mathcal{N}}_i^{(k)} = \hat I_{\hat\sigma_{k,i}^-}^-  \cup  \hat I_{\hat \sigma_{k,i}^+}^+  \cup \hat A_{\hat \sigma_{k,i}^+}^+ \cup \hat A_{\hat \sigma_{k,i}^-}^-$ to be the set of {\em identities} that appear during the construction of the double tree $(\mathcal{T}_\mu^+(\kappa_m;\bold{X}_i), \mathcal{T}_\mu^-(\kappa_m;\bold{X}_i))$ up to Step $\hat{\sigma}_{k,i}^+$ on the inbound side, and up to Step $\hat{\sigma}_{k,i}^-$ on the outbound side. Let $\hat{C}_{ij}= \left\{ \hat{\mathcal{N}}_i^{(k)}\cap\hat{\mathcal{N}}_j^{(k)}=\varnothing \right\}$. We then have
\begin{align*}
\mathbb{P}(C_{ij}^c) &\leq \mathbb{P}(C_{ij}^c,\tau_i^+>\sigma_{k,i}^+,\tau_i^->\sigma_{k,i}^-,\tau_j^+>\sigma_{k,j}^+,\tau_j^->\sigma_{k,j}^-) 1(\Omega_{m,n})  + 1(\Omega_{m,n}^c) \\
&\hspace{5mm} + \mathbb{P}( \{ \tau_i^+ \leq \sigma_{k,i}^+ \} \cup \{\tau_i^-\leq \sigma_{k,i}^-\} )+\mathbb{P}( \{ \tau_j^+ \leq \sigma_{k,j}^+ \} \cup \{ \tau_j^- \leq \sigma_{k,j}^-\} )  \\
&\leq 1(\Omega_{m,n}^c) +  \mathbb{P}(\hat{C}_{ij}^c, \, |\hat{\mathcal{N}}_i^{(k) }  | \leq \log{n}) 1(\Omega_{m,n}) +\mathbb{P}(|\hat{\mathcal{N}}_i^{(k) }  | >\log{n}) \\
&\hspace{5mm} + \mathbb{P}_i( \{ \tau^+ \leq \sigma_{k}^+ \} \cup \{\tau^-\leq \sigma_{k}^-\} )+\mathbb{P}_j( \{ \tau^+ \leq \sigma_{k}^+ \} \cup \{ \tau^- \leq \sigma_{k}^-\} ) ,
\end{align*}
where the event $\Omega_{m,n}$ is defined in Theorem~\ref{T.Coupling}. 

To bound the first probability on the right-hand side, define $\hat{\mathcal{F}}_{k,i}$ to be the sigma-algebra generated by the construction of the double tree whose root has {\em identity} $i$, up to Step $\hat{\sigma}_{k,i}^+$ on the inbound side and up to Step $\hat{\sigma}_{k,i}^-$ on the outbound side. 
Now note that
\begin{align*}
\hat{C}_{ij} &= \{ j \notin \hat{\mathcal{N}}_i^{(k)} \} \cap \left(\bigcap_{r=1}^{\hat{\sigma}_{k,j}^-}\bigcap_{t\in \hat{\mathcal{N}}_i^{(k)}} \{\tilde{Z}_{\hat T_{r,j}^-,t}=0 \} \right) \cap \left(\bigcap_{r=1}^{\hat{\sigma}_{k,j}^+} \bigcap_{t\in \hat{\mathcal{N}}_i^{(k)}} \{ Z_{t,\hat T_{r,j}^+}=0 \} \right)
\end{align*}
where $\hat T_{r,j}^+$ and $\hat T_{r,j}^-$ are the $r$th active {\em identities} to have their offspring sampled in the double tree whose root is $j$. Moreover, if we define $B_{s} = \bigcap_{t\in \hat{\mathcal{N}}_i^{(k)}} \{ Z_{ts}=0 \}$ and $\tilde B_{s} = \bigcap_{t\in \hat{\mathcal{N}}_i^{(k)}} \{ \tilde Z_{st}=0 \}$, then
$$\{ j \notin \hat{\mathcal{N}}_i^{(k)} \} = B_{j} \cap \tilde B_{j} \qquad \text{and} \qquad \hat C_{ij} = B_{j} \cap \tilde B_{j} \cap \left( \bigcap_{r=1}^{\hat{\sigma}_{k,j}^-} \tilde B_{\hat T_{r,j}^-} \right) \cap \left( \bigcap_{r=1}^{\hat{\sigma}_{k,j}^+} B_{\hat T_{r,j}^+} \right) ,$$
and therefore, since $\hat \sigma_{k,j}^+, \hat \sigma_{k,j}^+ \leq k$, the union bound gives
\begin{align}
\mathbb{P}( \hat C_{ij}^c | \hat{\mathcal{F}}_i^{(k)})  &\leq   \mathbb{P}\left( \left.  B_{j}^c  \cup  \left( \bigcup_{r=1}^{\hat{\sigma}_{k,j}^+} B_{\hat T_{r,j}^+}^c\right)  \right|  \hat{\mathcal{F}}_i^{(k)} \right) +  \mathbb{P}\left( \left.   \tilde B_{j}^c \cup \left( \bigcup_{r=1}^{\hat{\sigma}_{k,j}^-} \tilde B_{\hat T_{r,j}^-}^c \right)   \right|  \hat{\mathcal{F}}_i^{(k)} \right) \notag  \\
&=  \mathbb{E}\left[ \left.  1(B_{j}^c) + \sum_{r=1}^{\hat{\sigma}_{k,j}^+}  1\left( B_j \cap \bigcap_{s=1}^{r-1} B_{\hat T_{s,j}^+} \cap B_{\hat T_{r,j}^+}^c \right)   \right|  \hat{\mathcal{F}}_i^{(k)} \right] \notag \\
&\hspace{5mm} + \mathbb{E}\left[ \left.  1(\tilde B_{j}^c) + \sum_{r=1}^{\hat{\sigma}_{k,j}^-}  1\left( \tilde B_j \cap \bigcap_{s=1}^{r-1} \tilde B_{\hat T_{s,j}^-} \cap \tilde B_{\hat T_{r,j}^-}^c \right)   \right|  \hat{\mathcal{F}}_i^{(k)} \right]  \notag \\ 
&\leq \mathbb{E}\left[ \left.  1(B_{j}^c) + \sum_{r=1}^k  1\left( \hat A_{r-1,j}^+ \neq \varnothing, \, \bigcap_{s=1}^r \{ \hat T_{s,j}^+ \notin \hat{\mathcal{N}}_i^{(k)} \},  B_{\hat T_{r,j}^+}^c \right)   \right|  \hat{\mathcal{F}}_i^{(k)} \right] \label{eq:InboundNoI} \\
&\hspace{5mm} +  \mathbb{E}\left[ \left.  1(\tilde B_{j}^c) + \sum_{r=1}^k  1\left( \hat A_{r-1,j}^- \neq \varnothing, \, \bigcap_{s=1}^r \{ \hat T_{s,j}^- \notin \hat{\mathcal{N}}_i^{(k)}\},  \tilde B_{\hat T_{r,j}^-}^c \right)   \right|  \hat{\mathcal{F}}_i^{(k)} \right] , \label{eq:OutboundNoI}
\end{align}
where $\hat A_{r,j}^+$ and $\hat A_{r,j}^-$ are the $r$th inbound and outbound active sets in the construction of the double tree started at $j$, and $\Omega_{m,n}$ is defined in Theorem~\ref{T.Coupling}. Now note that the event $\bigcap_{s=1}^r \{ \hat T_{s,j}^+ \notin \hat{\mathcal{N}}_i^{(k)}\}$ gives that none of the $\{ U_{s,\hat T_{r,j}^+}: 1 \leq s \leq n\}$  have been used in the construction of the double tree started at $i$, hence
$$ \mathbb{P}\left( \left. \hat A_{r-1,j}^+ \neq \varnothing, \, \bigcap_{s=1}^r \{ \hat T_{s,j}^+ \notin \hat{\mathcal{N}}_i^{(k)}\},   B_{\hat T_{r,j}^+}^c \right| \hat{\mathcal{F}}_i^{(k)} \right)  \leq \mathbb{E}\left[ 1( \hat A_{r-1,j}^+ \neq \varnothing) Q(\hat{\mathcal{N}}_i^{(k)}, \hat T_{r,j}^+) \right],$$
where
\begin{align*}
Q(V, s) &= \mathbb{P}\left( \bigcup_{t \in V} \{ Z_{ts} \geq 1 \}  \right) \leq  \sum_{t \in V} P(Z_{ts} \geq 1) =  \sum_{t \in V}  (1 - e^{-r_{ts}^{(m,n)}} ) \\
&\leq \sum_{t \in V}  r_{ts}^{(m,n)} \leq \frac{R_n}{n} \sum_{t \in V} \kappa_m({\bf X}_t, {\bf X}_s) \leq \frac{R_n}{n} |V| \sup_{{\bf x}, {\bf y} \in \mathcal{S}} \kappa_m({\bf x}, {\bf y}),
\end{align*}
and $R_n = \max_{1 \leq t \leq M_m} 1( \mu_n(\mathcal{J}_t^{(m)}) > 0) \mu(\mathcal{J}_t^{(m)})/\mu_n(\mathcal{J}_t^{(m)})$.  Since $\mathbb{P}(B_j^c | \hat{\mathcal{F}}_i^{(k)}) \leq Q(\hat{\mathcal{N}}_i^{(k)}, j)$, we obtain that \eqref{eq:InboundNoI} is bounded from above by
$$Q(\hat{\mathcal{N}}_i^{(k)}, j) + \sum_{r=1}^k \mathbb{E}\left[ \left. 1(\hat A_{r-1,j}^+ \neq \varnothing) Q(\hat{\mathcal{N}}_i^{(k)}, \hat T_{r,j}^+) \right| \hat{\mathcal{F}}_i^{(k)}\right] \leq \frac{R_n(k+1)}{n} |\hat{\mathcal{N}}_i^{(k)}| \sup_{{\bf x}, {\bf y} \in \mathcal{S}} \kappa_m({\bf x}, {\bf y}).$$
Similarly, \eqref{eq:OutboundNoI} is bounded from above by
$$\frac{R_n(k+1)}{n} |\hat{\mathcal{N}}_i^{(k)}| \sup_{{\bf x}, {\bf y} \in \mathcal{S}} \kappa_m({\bf x}, {\bf y}).$$

It follows that
$$\mathbb{P}( \hat C_{ij}^c | \hat{\mathcal{F}}_i^{(k)}) \leq \frac{2R_n(k+1)}{n} |\hat{\mathcal{N}}_i^{(k)}| \sup_{{\bf x}, {\bf y} \in \mathcal{S}} \kappa_m({\bf x}, {\bf y}),$$
which in turn implies that for any $i,j \in V_n$, 
\begin{align*}
\mathbb{P}( \hat C_{ij}^c, |\hat{\mathcal{N}}_i^{(k)}| < \log n) 1(\Omega_{m,n})  &= \mathbb{E}\left[ \mathbb{P}( \hat C_{ij}^c | \hat{\mathcal{F}}_i^{(k)}) 1( |\hat{\mathcal{N}}_i^{(k)}| < \log n) \right] 1(\Omega_{m,n}) \\
&\leq \frac{4(k+1) \log n}{n} \sup_{{\bf x}, {\bf y} \in \mathcal{S}} \kappa_m({\bf x}, {\bf y}),
\end{align*}
and we have used the observation that on $\Omega_{m,n}$ we have $R_n \leq 1+ \epsilon \leq 2$.  

Using this estimate we obtain that
\begin{align*}
\frac{1}{n^2}\sum_{i=1}^n\sum_{j\neq i}\mathbb{P}(C_{ij}^c)&\leq  \frac{1}{n^2} \sum_{i=1}^n\sum_{j\neq i} \left\{ 1(\Omega_{m,n}^c) + \mathbb{P}(\hat C_{ij}^c, \, |\hat{\mathcal{N}}_i^{(k)}| \leq \log n) 1(\Omega_{m,n}) + \mathbb{P}(|\hat{\mathcal{N}}_i^{(k)}| > \log n ) \right\} \\
&\hspace{5mm}  + \frac{2}{n}\sum_{i=1}^n \mathbb{P}_i( \{ \tau^+ \leq \sigma_{k}^+ \} \cup \{\tau^-\leq \sigma_{k}^-\})  \\
&\leq 1(\Omega_{m,n}^c) +  \frac{4(k+1) \log n}{n} \sup_{{\bf x}, {\bf y} \in \mathcal{S}} \kappa_m({\bf x}, {\bf y}) + \frac{1}{n}\sum_{i=1}^n \mathbb{P}(|\hat{\mathcal{N}}_i^{(k) }  | >\log{n}) \\
&\hspace{5mm} + \frac{2}{n} \sum_{i=1}^n \mathbb{P}_i( \{ \tau^+ \leq \sigma_{k}^+ \} \cup \{\tau^-\leq \sigma_{k}^-\}).
\end{align*}
To complete the proof, apply Theorem~\ref{T.Coupling} to obtain
$$1(\Omega_{m,n}^c) + \frac{2}{n} \sum_{i=1}^n \mathbb{P}_i( \{ \tau^+ \leq \sigma_{k}^+ \} \cup \{\tau^-\leq \sigma_{k}^-\}) \leq 3 H(n,m,k,\epsilon),$$
and Markov's inequality followed by Lemma~\ref{L.MeanTree} to get
$$\frac{1}{n}\sum_{i=1}^n \mathbb{P}(|\hat{\mathcal{N}}_i^{(k) }  | >\log{n}) \leq \frac{1}{n \log n}\sum_{i=1}^n \mathbb{E}\left[ |\hat{\mathcal{N}}_i^{(k) }  | \right] \leq \frac{k}{\log n} \left( 2 + \sup_{{\bf x} \in \mathcal{S}} \lambda_+^{(m)}({\bf x}) + \sup_{{\bf x} \in \mathcal{S}} \lambda_-^{(m)}({\bf x}) \right).$$
\end{proof}

We are now ready to prove Theorem \ref{giant}, the phase transition for the existence of a giant strongly connected component in $G_n(\kappa(1+\varphi_n))$. 

\begin{proof}[Proof of Theorem \ref{giant}]
By Lemma \ref{approx_reducible}, there exists a sequence of kernels $\{\kappa_m:m\geq 1\}$ defined on $\mathcal{S}\times\mathcal{S}$ such that $\kappa_m$ is quasi-irreducible, regular finitary, and such that for any $n\geq m$, we have
$$\kappa_m(\bold{x},\bold{y})\leq \kappa(\bold{x},\bold{y})(1+\varphi_n(\bold{x},\bold{y})) \qquad \text{for all } \bold{x,y}\in\mathcal{S}.$$

{\em Proof of the lower bound:} We will start by proving a lower bound for the largest strongly connected component of $G_n(\kappa,\varphi_n)$. To this end, note that we can construct a coupling between $G_n(\kappa (1+\varphi_n))$ and $G_n(\kappa_m)$ such that every arc in $G_n(\kappa_m)$ is also in $G_n(\kappa(1+\varphi_n))$ $\mathbb{P}$-a.s. It follows that
$$\mathcal{C}_1(G_n(\kappa(1+\varphi_n))) \geq \mathcal{C}_1(G_n(\kappa_m)) \quad\mathbb{P}\text{-a.s.}$$
The idea is now to apply Theorem 1 in \cite{bloznelis2012birth} to $G_n(\kappa_m)$, however, that theorem requires 
that the kernel $\kappa_m$ be irreducible, whereas $\kappa_m$ is only quasi-irreducible. To address this issue, we construct a third graph as follows. Let $\mathcal{S}^*$ be the restriction of $\mathcal{S}$ where $\kappa_m$ is irreducible and set
$$\mathcal{S}' =\bigcup_{i=1}^{M_m}\left\{\mathcal{J}_i^{(m)} \cap  \mathcal{S}^*: \mu(\mathcal{J}_i^{(m)})>0 \right\}.$$
To avoid trivial cases, assume from now on that $\mu(\mathcal{S}') > 0$. 

Now let $V_{n'} = \{ 1 \leq i \leq n: {\bf X}_i \in \mathcal{S}'\}$ denote the set of vertices in $G_n(\kappa_m)$ that have types in $\mathcal{S}'$ and let $n'$ denote its cardinality. Note that $n'$ is random, but measurable with respect to $\mathscr{F}$. Next, fix $0 < \delta < 1$ and define the kernel $\kappa'({\bf x}, {\bf y}) = (1-\delta) \mu(\mathcal{S}') \kappa_m({\bf x}, {\bf y})$ and the graph $G_{n'}(\kappa')$ whose arc probabilities are given by
$$p_{ij}^{(n')} = \frac{(1-\delta) \mu(\mathcal{S}') \kappa_m({\bf X}_i, {\bf X}_j)}{n'} \wedge 1, \qquad i,j \in V_{n'}, \, i\neq j.$$
Note that $G_{n'}(\kappa')$ is a graph on the type space $\mathcal{S}'$ whose types are distributed according to measure $\mu_{n}'( A) := \mu_n(A)/\mu_n(\mathcal{S}')$ for any $A \subseteq \mathcal{S}'$. 
Moreover, $\kappa'$ is irreducible on $\mathcal{S}'$ with each of its induced types, i.e., the sets $\mathcal{J}_i^{(m)} \cap \mathcal{S}'$, having strictly positive measure. Now note that since $n \mu_n(\mathcal{S}')= n'$ and $\mu_n(\mathcal{S}') \xrightarrow{P} \mu(\mathcal{S}')$ as $n \to \infty$, then
$$p_{ij}^{(n')} = \frac{(1-\delta) \mu(\mathcal{S}')  \kappa_m({\bf X}_i, {\bf X}_j)}{n \mu_n(\mathcal{S}')} \wedge 1 \leq \frac{\kappa_m({\bf X}_i, {\bf X}_j)}{n} \wedge 1, \qquad i,j \in V_n', \, i \neq j,$$
for all sufficiently large $n$. Therefore, there exists a coupling such that every arc in $G_{n'}(\kappa')$ is also in $G_n(\kappa_m)$, and therefore, for all sufficiently large $n$, 
$$\mathcal{C}_1(G_n(\kappa_m)) \geq \mathcal{C}_1(G_{n'}( \kappa')) \quad\mathbb{P}\text{-a.s.}$$
Now use Theorem~1 in \cite{bloznelis2012birth} to obtain that for every $\epsilon > 0$
$$P\left( \left| \frac{\mathcal{C}_1(G_{n'}(\kappa'))}{n'} - \rho'(\kappa') \right| > \epsilon \right) \to 0  \qquad n \to \infty,$$
where
$$\rho'(\kappa') = \int_{\mathcal{S}'} \rho'_+(\kappa'; {\bf x}) \rho_-'(\kappa'; {\bf x}) \mu'(d{\bf x}),$$
and $\rho'_+(\kappa'; {\bf x}), \rho_-'(\kappa'; {\bf x})$ are the survival probabilities of the trees $\mathcal{T}^+_{\mu'} (\kappa')$ and $\mathcal{T}^-_{\mu'} (\kappa')$, respectively, defined on the type space $\mathcal{S}'$ with respect to the measure $\mu'(A) = \mu(A)/\mu(\mathcal{S}')$ for $A \subseteq \mathcal{S}'$.

By the arguments in the proof of Proposition~\ref{P.FinitaryPhaseTransition}, we have that $\rho((1-\delta)\kappa_m) = \mu(\mathcal{S}') \rho'(\kappa')$, where 
$$\rho((1-\delta) \kappa) = \int_{\mathcal{S}} \rho_+((1-\delta) \kappa_m; {\bf x}) \rho_-((1-\delta) \kappa; {\bf x}) \mu(d{\bf x}),$$
and $\rho_+((1-\delta) \kappa_m; {\bf x})$, $\rho_-((1-\delta) \kappa_m; {\bf x})$ are the survival probabilities of the trees $\mathcal{T}_\mu^+((1-\delta)\kappa_m)$ and $\mathcal{T}_\mu^-((1-\delta) \kappa_m)$, defined on the type space $\mathcal{S}$.

Hence, 
\begin{align*}
\frac{\mathcal{C}_1(G_n(\kappa (1+\varphi_n)))}{n} \geq \frac{\mathcal{C}_1(G_n((1-\delta)\kappa_m))}{n} \geq \frac{\mathcal{C}_1(G_{n'}(\kappa'))}{n'} \cdot \frac{n'}{n} \xrightarrow{P} \rho'(\kappa') \mu(\mathcal{S}') = \rho((1-\delta) \kappa_m),
\end{align*}
as $n \to \infty$. Now use Lemma~\ref{L.MonotoneKernels} to obtain that
$$\lim_{m \to \infty} \lim_{\delta \downarrow 0} \rho((1-\delta) \kappa_m) =  \lim_{\delta \downarrow 0} \lim_{m \to \infty} \rho((1-\delta) \kappa_m) = \rho(\kappa), $$ 
from where we conclude that for any $\epsilon > 0$, 
$$P\left( \frac{\mathcal{C}_1(G_n(\kappa, \varphi_n))}{n} - \rho(\kappa) < -\epsilon \right) \to 0 \qquad n \to \infty.$$

\bigskip

{\em Proof of the upper bound:}  For any $k,m \in \mathbb{N}_+$ let $\rho^{\geq k}_+(\kappa_m; {\bf x})$  ($\rho^{\geq k}_-(\kappa_m; {\bf x})$) denote the probability that the tree $\mathcal{T}_\mu^+(\kappa_m;{\bf x})$ ($\mathcal{T}_\mu^-(\kappa_m; {\bf x})$) has a population of at least $k$ nodes. Define for $k \geq 1$ the set
$$N_n^{\geq k} = | \{ i \in V_n: \text{ both the in-component and out-component of $i$ have at least $k$ vertices} \}|,$$
and note that  
$$\mathcal{C}_1(G_n(\kappa (1+\varphi_n))) \leq N_n^{\geq k} \qquad \text{for any } k \geq 1.$$
It follows that
\begin{align}
\frac{\mathcal{C}_1(G_n(\kappa (1+\varphi_n)))}{n} - \rho(\kappa) &\leq \frac{N_n^{\geq k} }{n} - \frac{1}{n} \sum_{i=1}^n \mathbb{E}\left[ \chi_{n,i}^{\geq k} \right] \notag \\ 
&\hspace{5mm} +  \frac{1}{n} \sum_{i=1}^n \mathbb{E}\left[ \chi_{n,i}^{\geq k} \right]  - \frac{1}{n} \sum_{i=1}^n \rho_+^{\geq k}(\kappa_m;{\bf X}_i) \rho_-^{\geq k}(\kappa_m; {\bf X}_i) \label{eq:MeanApprox} \\
&\hspace{5mm} + \frac{1}{n} \sum_{i=1}^n \rho_+^{\geq k}(\kappa_m;{\bf X}_i) \rho_-^{\geq k}(\kappa_m; {\bf X}_i) - \rho^{\geq k}(\kappa_m) + \rho^{\geq k}(\kappa_m)  - \rho(\kappa) . \notag
\end{align}
Moreover, by Proposition~\ref{P.BernoullisCoupled} we have that for any $0 < \epsilon < 1/2$,  \eqref{eq:MeanApprox}  is bounded by $H(n,m,k,\epsilon)$, where $H(n,m,k,\epsilon)$ is defined in Theorem~\ref{T.Coupling} and satisfies
$$\lim_{\epsilon \downarrow 0} \limsup_{n \to \infty} H(n,m,k,\epsilon) \leq  \hat H(m,k) \qquad \text{in probability},$$
for some other function $\hat H(m,k)$ also defined in Theorem~\ref{T.Coupling} and satisfying $\lim_{m\to\infty} \hat H(m,k) = k^{-1} 1(k \geq 2)$. Also, by the bounded convergence theorem we have that for any $m,k \in \mathbb{N}_+$,
$$ \frac{1}{n} \sum_{i=1}^n \rho_+^{\geq k}(\kappa_m;{\bf X}_i) \rho_-^{\geq k}(\kappa_m; {\bf X}_i)  =  \int_{\mathcal{S}}  \rho_+^{\geq k}(\kappa_m;{\bf x}) \rho_-^{\geq k}(\kappa_m; {\bf x}) \mu_n(d{\bf x}) \xrightarrow{P}  \rho^{\geq k}(\kappa_m)  \qquad  n \to \infty,$$
and by Lemma~\ref{L.MonotoneKernels} we have
$$\lim_{k \to \infty} \lim_{m \to \infty} \rho^{\geq k}(\kappa_m)  - \rho(\kappa) = 0. $$
Therefore, for any $0 < \delta <1$ we can choose $m,k \in \mathbb{N}_+$ such that
$$\hat H(m,k) + \rho^{\geq k}(\kappa_m) - \rho(\kappa) < \delta/2,$$
and for such $\delta, m, k$, and any $0 < \epsilon < 1/2$, we have
\begin{align*}
\mathbb{P}\left( \frac{\mathcal{C}_1(G_n(\kappa, \varphi_n))}{n} - \rho(\kappa) > \delta \right) 1(\Omega_{m,n}) &\leq \mathbb{P}\left( \frac{N_n^{\geq k} }{n} - \frac{1}{n} \sum_{i=1}^n \mathbb{E}\left[ \chi_{n,i}^{\geq k} \right] + L(n,m,k,\epsilon) > \delta/2 \right) 1(\Omega_{m,n})  \\
&\leq \frac{1(\Omega_{m,n}) }{(\delta/2 - L(n,m,k,\epsilon))^2} \, \mathbb{E}\left[ \left( \frac{N_n^{\geq k} }{n} - \frac{1}{n} \sum_{i=1}^n \mathbb{E}\left[ \chi_{n,i}^{\geq k} \right] \right)^2 \right],
\end{align*}
where
$$L(n,m,k,\epsilon) := H(n,m,k,\epsilon) - \hat H(m,k) + \frac{1}{n} \sum_{i=1}^n \rho_+^{\geq k}(\kappa_m;{\bf X}_i) \rho_-^{\geq k}(\kappa_m; {\bf X}_i) - \rho^{\geq k}(\kappa_m) $$
satisfies $\lim_{\epsilon \downarrow 0} \limsup_{n \to \infty} L(n,m,k,\epsilon) 1(\Omega_{m,n}) = 0$ in probability. It remains to show that the expectation can be made arbitrarily small.  To see this, use Proposition~\ref{P.BernoullisCoupled} again to obtain that on the event $\Omega_{m,n}$,
\begin{align*}
&\mathbb{E}\left[ \left( \frac{N_n^{\geq k} }{n} - \frac{1}{n} \sum_{i=1}^n \mathbb{E}\left[ \chi_{n,i}^{\geq k} \right] \right)^2 \right]  \\
&= \frac{1}{n^2} \left\{ \sum_{i=1}^n \mathbb{E}\left[ \left( \chi_{n,i}^{\geq k} - \mathbb{E}\left[ \chi_{n,i}^{\geq k} \right] \right)^2 \right] +  \sum_{i=1}^n \sum_{j \neq i} \mathbb{E}\left[ \left( \chi_{n,i}^{\geq k} - \mathbb{E}\left[ \chi_{n,i}^{\geq k} \right] \right) \left( \chi_{n,j}^{\geq k} - \mathbb{E}\left[ \chi_{n,j}^{\geq k} \right] \right) \right] \right\}  \\
&\leq  \frac{1}{n^2}  \left( \sum_{i=1}^n \rho_+^{\geq k}(\kappa_m;{\bf X}_i) \rho_-^{\geq k}(\kappa_m; {\bf X}_i)  + n H(n,m,k,\epsilon)  \right) + K(n,m,k) + 3 H(n,m,k,\epsilon)  \\
&=: D(n,m,k,\epsilon,\delta) (\delta/2 - L(n,m,k,\epsilon))^2,
\end{align*}
where $K(n,m,k)$ is defined in Proposition~\ref{P.BernoullisCoupled} and satisfies $K(n,m,k) \xrightarrow{P} 0$ as $n \to \infty$. We have thus obtained that
\begin{align*}
&\lim_{\epsilon \downarrow 0} \limsup_{n \to \infty} P\left( \frac{\mathcal{C}_1(G_n(\kappa, \varphi_n))}{n} - \rho(\kappa) > \delta \right) \\
& \leq \lim_{\epsilon \downarrow 0} \limsup_{n \to \infty} \left\{ E\left[ \mathbb{P}\left( \frac{\mathcal{C}_1(G_n(\kappa, \varphi_n))}{n} - \rho(\kappa) > \delta \right)  1(\Omega_{m,n}, D(n,m,k,\epsilon,\delta) \leq 1) \right] \right. \\
&\hspace{15mm} \left. + P(\Omega_{m,n}^c)  + P\left( D(n,m,k,\epsilon,\delta) > 1 \right) \right\} \\
&\leq  \lim_{\epsilon \downarrow 0} \limsup_{n \to \infty}  \left\{ E\left[ D(n,m,k,\epsilon,\delta) 1(\Omega_{m,n},  D(n,m,k,\epsilon,\delta) \leq 1) \right] + P(\Omega_{m,n}^c) + P\left( D(n,m,k,\epsilon,\delta) > 1 \right) \right\} \\
&\leq  \frac{4}{\delta^2} \cdot 3 \hat H(m,k)  ,
\end{align*}
and we have used the observation that $\lim_{n \to \infty} P(\Omega_{m,n}^c) = 0$ by Assumption~\ref{reg}(a). Taking the limit as $m \to \infty$ and then as $k \to \infty$ completes the proof of the upper bound. 

\bigskip

{\em Proof of the phase transition:} It follows from Lemma~\ref{L.PhaseTransition}.
\end{proof}

We end the paper with the proof of Proposition~\ref{P.Rank-1}, which states the main results for the rank-1 kernel case. 

\begin{proof}[Proof of Proposition~\ref{P.Rank-1}]
The first two statements follow immediately from noting that $E[\kappa_-({\bf X})] E[\kappa_+({\bf X})] = \iint_{\mathcal{S}^2} \kappa({\bf x}, {\bf y}) \mu(d{\bf x}) \mu(d{\bf y})$. The third one from noting that \linebreak $\lambda_+({\bf X}) = \kappa_+(\bold{X})E[\kappa_-(\bold{X})]$ and $\lambda_-({\bf X}) = \kappa_-(\bold{X})E[\kappa_+(\bold{X})]$. 

To establish (d) assume first that $\rho(\kappa) > 0$. Now use Lemma~\ref{approx_reducible} (applied to $\kappa({\bf x}, {\bf y}) = \kappa_+({\bf y})$ and $\kappa({\bf x}, {\bf y}) = \kappa_-({\bf x})$ separately) to obtain that there exist a sequence of kernels $\{ \kappa_m^-({\bf x}): m \geq 1\}$ and $\{ \tilde \kappa_m^+({\bf x}): m \geq 1\}$ such that: 1) $0 \leq \kappa_m^\pm({\bf x}) \leq  \kappa_\pm({\bf x})$ for all ${\bf x} \in \mathcal{S}$, 2) each is piecewise constant taking only a finite number of values, and 3) $\kappa_m^\pm({\bf x}) \nearrow \kappa_\pm({\bf x})$ for a.e.~${\bf x} \in \mathcal{S}$ as $m \to \infty$.  Now set $B_m = \{ {\bf x} \in \mathcal{S}: \kappa_m^+({\bf x}) > 0, \kappa_m^-({\bf x}) > 0 \}$ and define 
$$\kappa_m({\bf x}, {\bf y}) = \kappa_m^-({\bf x}) \kappa_m^+({\bf y}) 1( {\bf x} \in B_m, {\bf y} \in B_m).$$
Note that $\kappa_m$ is regular finitary and is strictly positive on $B_m \times B_m$. Hence, the only set $A \subseteq B_m$ satisfying $\kappa_m = 0$ on $A \times (A^c \cap B_m)$ is $A = \varnothing$ or $A^c \cap B_m = \varnothing$, implying the irreducibility of $\kappa_m$ on $B_m \times B_m$. Moreover, since $\kappa_+ > 0$ and $\kappa_- > 0$ a.e.~in order for $\kappa$ to be irreducible, we have that $\kappa_m \nearrow \kappa$ as $m \to \infty$. 

Next, use Lemma~\ref{L.MonotoneKernels} to obtain that $\rho(\kappa) = \lim_{m \to \infty} \rho(\kappa_m)$, and therefore, $\rho(\kappa_m) > 0$ for some $m$ sufficiently large. By Proposition~\ref{P.FinitaryPhaseTransition} this implies that the spectral radii of the operators $T_{\kappa_m}^+$ and $T_{\kappa_m}^-$  are strictly larger than one.  Now note that the functions $f_m^+({\bf x}) = \kappa_m^+({\bf x})$ and $f_m^-({\bf x}) = \kappa_m^-({\bf x})$ are nonnegative and satisfy
\begin{align*}
T_{\kappa_m}^+ f^+_m({\bf x}) &= \int_{\mathcal{S}} \kappa_m^-({\bf y}) \kappa_m^+({\bf x}) f_m^+({\bf y}) \mu(d{\bf x}) = \kappa_m^+({\bf x}) \int_{\mathcal{S}} \kappa_m^-({\bf y}) f_{m}^+({\bf y}) \mu(d{\bf y}) \\
&= f_m^+({\bf x}) \int_{\mathcal{S}} \kappa_m^-({\bf y}) \kappa_{m}^+({\bf y}) \mu(d{\bf y}) ,
\end{align*}
and therefore, $r_m := \int_{\mathcal{S}} \kappa_m^-({\bf y}) \kappa_{m}^+({\bf y}) \mu(d{\bf y})$ is an eigenvalue of $T_{\kappa_m}^+$. Similarly, $r_m$ is an eigenvalue of $T_{\kappa_m}^-$ associated to the nonnegative eigenfunction $f_m^-$. Since we may assume that $\kappa_m^+({\bf x})$ and $\kappa_m^-({\bf x})$ are different from zero for sufficiently large $m$, then Proposition~\ref{P.FinitaryPhaseTransition} gives that $r_m = r(T_{\kappa_m}^\pm) > 1$. Taking the limit as $m \to \infty$ gives that
$$E[\kappa_-({\bf X}) \kappa_+({\bf X})] = \lim_{m \to \infty} r_m > 1.$$
For the converse, note that $E[\kappa_-({\bf X}) \kappa_+({\bf X})]  > 1$ and the monotone convergence theorem imply that $r_m > 1$ for some $m$ sufficiently large. For this $m$, Proposition~\ref{P.FinitaryPhaseTransition} gives that $r_m$ is the spectral radius of $T_{\kappa_m}^+$ and $T_{\kappa_m}^-$, and also that $\rho(\kappa_m) > 0$. Lemma~\ref{L.MonotoneKernels} now gives that $1 < \rho(\kappa_m) \nearrow \rho(\kappa)$ as $m \to \infty$.  
\end{proof}

\bibliography{component}
\bibliographystyle{plain}
\end{document}